\documentclass[11pt,a4paper]{amsart}

\usepackage[margin=3cm,top=2.5cm,bottom=2cm]{geometry}

\usepackage{amsmath, amssymb, amsthm, mathrsfs}
\usepackage{graphicx,enumitem,microtype} 
\usepackage[dvipsnames]{xcolor}

\usepackage[colorlinks=true,linkcolor=red!70!black,urlcolor=MidnightBlue,citecolor=MidnightBlue]{hyperref}

\usepackage[utf8]{inputenc} 
\usepackage[T1]{fontenc}
\usepackage{lmodern}

\newcommand\N{{\mathbb N}}
\newcommand\R{{\mathbb R}}

\renewcommand{\d}{\mathrm{d}}
\newcommand{\dv}{\mathrm{d} v}
\newcommand{\dt}{\mathrm{d} t}

\newcommand{\dx}{\mathrm{d} x}

\def\AA{{\mathcal A}}
\def\BB{{\mathcal B}}
\def\CC{{\mathcal C}}
\def\DD{{\mathcal D}}
\def\EE{{\mathcal E}}

\def\HH{{\mathcal H}}

\def\LL{{\mathcal L}}
\def\MM{{\mathcal M}}

\def\OO{{\mathcal O}}

\def\RR{{\mathcal R}}

\def\TT{{\mathcal T}}
\def\UU{{\mathcal U}}
\def\VV{{\mathcal V}}

\def\VV{{\mathcal V}}
\def\XX{{\mathcal X}}

\def\BBB{{\mathscr B}}
\def\CCC{{\mathscr C}}
\def\DDD{{\mathscr D}}
\def\EEE{{\mathscr E}}

\def\HHH{{\mathscr H}}

\def\KKK{{\mathscr K}}

\def\MMM{{\mathscr M}}

\def\RRR{{\mathscr R}}
\def\SSS{{\mathscr S}}
\def\TTT{{\mathscr T}}

\def\XXX{{\mathscr X}}

\def\Wloc{W_{\mathrm{loc}}}

\def\Lloc{L_{\mathrm{loc}}}

\def\Hloc{H_{\mathrm{loc}}}

\def\eps{{\varepsilon}}
\def\wto{{\,\rightharpoonup\,}}

\newcommand{\la}{\langle}
\newcommand{\ra}{\rangle}
\newcommand{\rra}{\rangle\!\rangle}
\newcommand{\lla}{\langle\!\langle}

\newcommand{\Nt}{|\hskip-0.04cm|\hskip-0.04cm|}

\DeclareMathOperator{\Div}{div}

\newtheorem{theo}{Theorem}[section]
\newtheorem{prop}[theo]{Proposition}
\newtheorem{lem}[theo]{Lemma}
\newtheorem{cor}[theo]{Corollary}
\newtheorem*{thm*}{Theorem}
\theoremstyle{remark}
\newtheorem{rem}[theo]{Remark}

\newtheorem*{ex*}{Example}
\theoremstyle{definition}

\numberwithin{equation}{section}
\newcommand{\be}{\begin{equation}}
\newcommand{\ee}{\end{equation}}
\newcommand{\ba}{\begin{aligned}}
\newcommand{\ea}{\end{aligned}}
\newcommand{\beqn}{\begin{equation}}
\newcommand{\eeqn}{\end{equation}}
\newcommand{\bear}{\begin{eqnarray}}
\newcommand{\eear}{\end{eqnarray}}
\newcommand{\bean}{\begin{eqnarray*}}
\newcommand{\eean}{\end{eqnarray*}}
\newcommand{\bal}{\begin{aligned}}
\newcommand{\eal}{\end{aligned}}

\title[The Landau equation in a domain]{The Landau equation in a domain}
\author[K. Carrapatoso]{Kleber Carrapatoso}
\author[S. Mischler]{St\'ephane Mischler}

\address[K.~Carrapatoso]{Centre de Math\'ematiques Laurent Schwartz, \'Ecole polytechnique, Institut Polytechnique de Paris, 91128 Palaiseau cedex, France}
\email{kleber.carrapatoso@polytechnique.edu}

\address[S.~Mischler]{Centre de Recherche en Math\'ematiques de
  la D\'ecision (CEREMADE, CNRS UMR 7534),
  Universit\'es PSL \& Paris-Dauphine, Place de Lattre de
  Tassigny, 75775 Paris 16, France \& Institut Universitaire de France (IUF)}
\email{mischler@ceremade.dauphine.fr}

\date{\today}


\thanks{K.C.\ was partially supported by the Project CONVIVIALITY ANR-23-CE40-0003 of the French National Research Agency.}

\begin{document}

\begin{abstract}
This work deals with the Landau equation in a bounded domain with the Maxwell reflection condition on the boundary 
for any (possibly smoothly position dependent)  accommodation coefficient and for the full range of interaction potentials, including the Coulomb case. We establish the global existence and a constructive asymptotic  decay of solutions in a close-to-equilibrium regime. 
This is the first existence result for a Maxwell reflection condition on the boundary and that generalizes the similar results established for the Landau equation for other geometries in 
\cite{GuoLandau1,GS1,GS2,MR3625186,MR4076068}. We also answer to Villani's program \cite{MR2116276,MR2407976} about constructive accurate rate of convergence to the equilibrium {(quantitative H-Theorem)} for solutions to collisional kinetic equations satisfying a priori uniform bounds.  
The proofs rely on the study of a suitably linear problem for which we prove that the associated operator is hypocoercive, the associated 
semigroup is ultracontractive, and finally that it is asymptotically stable in many weighted $L^\infty$ spaces. %We next conclude thanks to the Schauder-Tychonoff fixed point theorem. 
\end{abstract}

\maketitle

\tableofcontents

%%%%%%%%%%%%%%%%%%%%%%%%%%%%%%%%%%%%%%%%%
\section{Introduction and main results}

%%%%%%%%%%%%%%%%%%%%%%
\subsection{The Landau equation in a domain}
In this paper we are concerned with the existence  and long-time behavior in a perturbative regime for the Landau equation  (1936, \cite{Landau,LifchitzPitaevskii}) in a bounded domain, which is a fundamental model in kinetic theory describing the evolution of a dilute plasma.
We thus consider the Landau equation
\beqn\label{eq:Landau_F}
\partial_t F  = - v \cdot \nabla_x F +Q(F,F)
\quad\text{in}\quad (0,\infty) \times \Omega \times \R^3
\eeqn
for a distribution $F=F(t,x,v) \ge 0$ of particles which at time $t \ge 0$ and position $x \in \Omega \subset \R^3$ move with velocity $v \in \R^3$. The Landau equation in the interior of the domain \eqref{eq:Landau_F} is complemented with the Maxwell reflection condition \cite{Maxwell,MR1313028} on the incoming part of the boundary
\beqn\label{eq:reflect_F}
\gamma_- F = \RRR(\gamma_+F) \quad\text{on}\quad (0,\infty) \times \Sigma_- ,
\eeqn
as well as with an initial datum $F_{|t=0} = F_0$.

The Landau collision operator $Q$ in \eqref{eq:Landau_F} is a bilinear operator acting only on the velocity variable which, in the kinetic theory of gas, classically models the interacting through binary collisions. It is defined by one of the following equivalent formulations, using the convention of summation over repeated indices,
\begin{align}
Q(g,f) (v) 
&= \partial_{v_i} \int_{\R^3} a_{ij}(v-v_*) \left\{ g_* \partial_{v_j} f - f \partial_{v_j}g_* \right\} \d v_* \label{eq:oplandau0}\\
&= \partial_{v_i} \left\{ (a_{ij}*g) \partial_{v_j} f - (b_i * g) f \right\} \label{eq:oplandau1}\\
&= (a_{ij}*g) \partial_{v_i, v_j} f - (c*g) f \label{eq:oplandau2}\\
&= \partial^2_{v_iv_j} \left\{ (a_{ij}*g) f \right\}  - 2  \partial_{v_i} \left\{ (b_i * g) f  \right\}, \label{eq:oplandau3}
\end{align}
where $*$ stands for the convolution on the velocity variable $v \in \R^3$, the matrix $a$ is given by
$$
a_{ij} (z) = |z|^{\gamma+2} \left( \delta_{ij} - \frac{z_iz_j}{|z|^2} \right), \quad \gamma \in [-3,1],
$$ 
and
\begin{equation}\label{eq:defbic}
\begin{aligned}
b_i (z) &= \partial_{v_j} a_{ij}(z) = -2 |z|^\gamma z_i \\
c(z) &= \partial_{v_i , v_j} a_{ij}(z) = -2(\gamma+3) |z|^\gamma \quad \text{if} \quad -3 < \gamma \le 1 \\ 
c(z) &= \partial_{v_i , v_j} a_{ij}(z) = -8\pi \delta_0 \quad \text{if} \quad  \gamma =-3.
\end{aligned}
\end{equation} 
The parameter  $\gamma \in [-3,1]$ is supposed to be connected to the power of the interaction potential involved in the binary collisions. 
The cases $\gamma \in (0,1]$ correspond to hard potentials,  $\gamma \in [-2,0]$ to moderately soft potentials,  $\gamma \in (-3,-2)$ to very soft potentials, and $\gamma=-3$ to Coulomb potential. It is worth mentioning that the Coulomb potential is the most (if not only) physically relevant case.

\smallskip
The Maxwell reflection operator in \eqref{eq:reflect_F} is given by
\begin{equation}\label{eq:reflection}
\RRR (\gamma_{+} F) =(1-\iota) \SSS  \gamma_{\!+}  F + \iota \DDD  \gamma_{\!+}  F  , 
\end{equation}
where $\iota : \partial\Omega \to [0,1]$ is the accommodation coefficient that we assume to be a smooth function on $\partial\Omega$, $\SSS$ is the specular reflection operator, and $\DDD$ is the diffusive reflection operator defined below.
More precisely,  denoting by $n_x$ the outward unit normal vector at a point $x \in \partial\Omega$ of the boundary, we define the sets 
$$
\Sigma_{\pm}^x := \left\{ v \in \R^3;   \, \pm \, {v} \cdot n_x > 0 \right\} 
$$ 
of outgoing ($\Sigma_{+}^x$) and incoming ($\Sigma_{-}^x$) velocities, then the sets
\begin{equation}\label{eq:def:Sigma_Lambda}
\begin{aligned}
&\Sigma :=  \partial\Omega \times \R^3, \quad \Sigma_{\pm} := \left\{ (x,v) \in \Sigma;  \,  v \in \Sigma^x_{\pm} \right\}, 
\\
&\Gamma :=  (0,T) \times \Sigma, \quad \Gamma_{\pm} := (0,T) \times \Sigma_{\pm}, 
\quad T \in (0,\infty], 
\end{aligned}
\end{equation}
%$$
%\Sigma_{\pm} := \left\{ (x,v) \in \partial\Omega \times \R^3 \,\Big|\,  v \in \Sigma^x_{\pm} \right\}, 
%\quad
%\Gamma_{\pm} := (0,T) \times \Sigma_{\pm},
%\quad T \in (0,\infty]
%$$
and finally the outgoing and incoming trace functions 
\begin{equation}\label{eq:def:gammaf}
\gamma_\pm f := \mathbf{1}_{\Gamma_\pm}  \gamma f.
\end{equation}

The specular reflection operator $\SSS$ is defined by 
\begin{equation}\label{eq:def_SSS}
\SSS_x (g (x,\cdot))(v)  =  g (x , \VV_x v), \quad \VV_x v = v - 2 n_x (n_x \cdot v),
\end{equation}
and the diffusive operator  $\DDD$ is defined on $\Sigma_+$ by 
\begin{equation}\label{eq:def_DDD}
\DDD_x (g (x,\cdot))(v) = \MMM(v) \widetilde g (x), \quad \widetilde g (x) = \int_{\Sigma^x_+} g(x,w) \, (n_x \cdot w)_+ \, \d w,
\end{equation}
where $\MMM$ stands for the Maxwellian distribution
\begin{equation}\label{eq:def_MMM}
\MMM  := \sqrt{2 \pi} \mu , \quad  \mu(v) := (2\pi)^{-3/2} \exp(-|v|^2/2), \end{equation}
so that   $\widetilde{\MMM} = 1$ and $\mu$ is the standard Maxwellian function with integral one.
It is worth emphasizing that, for a dilute plasma or gas, it seems to be not completely clear which are the physically convenient reflection conditions to be imposed at the boundary of $\Omega$. 
However, the Maxwell reflection condition \eqref{eq:reflection} is one of the most commun and general reflection condition considered  in kinetic theory.

\smallskip

We shall suppose throughout the paper that $\Omega$ is a bounded open smooth and connected subset of $\R^3$. More precisely, we assume that there exists $\delta  \in W^{2,\infty}(\R^3)$ such that $\delta(x) = \mathrm{dist}(x,\partial \Omega)$ is the distance to the boundary in a neighborhood of $\partial\Omega$, and we denote
\begin{equation}\label{eq:def:OO_UU}
\OO := \Omega \times \R^3 \quad \text{and} \quad \UU := (0,T) \times \OO
\end{equation}
for $T \in (0,\infty]$.
Moreover we assume that $\iota$ is the restriction of a $W^{1,\infty}(\R^3)$ function.

%%%%%%%%%%%%%%%%%%%%%%%%%%%%%%%%%%%%%%%%%%%%
\subsection{Collisional invariants and conservation laws}
\label{subsec:conservationsLaw}
Let us briefly discuss at a formal level the physical properties of the solutions to the Landau equation \eqref{eq:Landau_F}-\eqref{eq:reflect_F}.
We refer to the introduction of \cite{MR4581432} for more details (see also \cite{MR1776840,MR2721875,MR2679358,MR4076068,Guoetal-specular2}). 

\smallskip\noindent
{\it The reflection operator.}  
Whatever is the accommodation coefficient $\iota$, we have  
\beqn\label{eq:invariantsBoundary1}
\int_{\R^3} \RRR (\gamma_+ F) \, (n_x \cdot v)_- \, \dv  =  \int_{\R^3} \gamma_+ F \, (n_x\cdot v)_+ \, \dv, %\quad \forall \, x \in \partial\Omega, 
\eeqn
which means that there is no flux of mass at the boundary (no particle leaves nor enters in the domain).
On the other hand, in the case of pure specular boundary condition $\iota \equiv 0$, we additionally have  
\beqn\label{eq:invariantsBoundary2}
\int_{\R^3} \RRR (\gamma_+ F) \, |v|^2 (n_x \cdot v)_- \, \dv  =  \int_{\R^3} \gamma_+ F \, |v|^2 (n_x\cdot v)_+ \, \dv, %\quad \forall \, x \in \partial\Omega, 
\eeqn
which means that  there is no flux of energy at the boundary in the case of the pure specular reflection boundary condition. 
Furthermore, still when $\iota \equiv 0$, we also have 
\beqn
\label{eq:invariantsBoundary3}
%&&\int_{\R^3} \RRR \gamma_+ F \,  v (n_x \cdot v)_- \, \dv  -  \int_{\R^3} \gamma_+ F \, v (n_x\cdot v)_+ \, \dv 
 \int_{\R^3} [\RRR (\gamma_+ F) \,  v (n_x \cdot v)_-    -    \gamma_+ F \, v (n_x\cdot v)_+ ] \, \dv 
=  - 2 n_x \int_{\R^3} \gamma_+ F \,   (n_x\cdot v)^2_+ \, \dv , 
%\quad \forall \, x \in \partial\Omega, 
\eeqn
which means that the flux of momentum at the boundary is normal to the boundary in the case of the pure specular reflection boundary condition. 
 
\smallskip\noindent
{\it The collisional operator.} From the formulation \eqref{eq:oplandau0}, we have
$$
\int_{\R^3} Q(F,F) \varphi \, \d v
= \frac12 \int_{\R^3}\! \int_{\R^3}  a_{ij}(v-v_*) \left\{ F_* \partial_{v_j} F - F  \partial_{v_j} F _* \right\} \left( \partial_{v_i} \varphi_* - \partial_{v_i} \varphi \right) \, \d v_* \, \d v,
$$
and thus the Landau operator enjoys the microscopic or collisional invariants
\beqn\label{eq:local-conservations}
\int_{\R^3} Q(F,F) \varphi \, \d v = 0, \quad \varphi = 1, \, v_i, \, |v|^2, 
\eeqn
where we use that $a_{ij} (z) z_j = 0$ for the energy identity.  The microscopic Landau operator formulation of the celebrated Boltzmann H-theorem may be expressed as 
$$
\int_{\R^3} Q(F,F) \log F \, \d v \le 0, \quad \forall \, F \ge 0, 
$$
with equality if, and only if, $F$ is a Gaussian function in $v$.  
 
\smallskip\noindent
{\it Macroscopic laws.} 
One easily obtains from \eqref{eq:local-conservations}, the  Green-Ostrogradski formula and   \eqref{eq:invariantsBoundary1} that any solution $F$ to the Landau equation \eqref{eq:Landau_F}-\eqref{eq:reflect_F}
satisfies 
%$f$ to equation \eqref{eq:dtf=Lf}--\eqref{eq:BdyCond} satisfies the conservation of mass
$$ 
\frac{\d}{\dt} \int_\OO F \, \d v \, \d x  
= \int_\OO (Q(F,F) - v \cdot \nabla_x F) \, \d v \, \d x
= 0, 
$$
so that the  total mass is conserved, namely 
$$
\lla   F(t,\cdot) \rra = \lla   F_0  \rra , \quad \forall \, t \ge 0, \quad   \lla G  \rra := \int_\OO G \, \d x \, \d v. 
$$

In the case of the specular reflection boundary condition ($\iota \equiv 0$), some additional conservation laws appear. On the one hand, one also has the conservation of energy
$$
\frac{\d}{\dt} \int_\OO |v|^2 F \, \dv \, \dx 
= \int_\OO |v|^2(Q(F,F)  - v \cdot \nabla_x F) \, \dv \, \dx
=0,
$$
because of \eqref{eq:local-conservations}, the Green-Ostrogradski formula again and \eqref{eq:invariantsBoundary2}. 
On the other hand, if the domain~$\Omega$ possesses a rotational symmetry,   we also have the conservation of the corresponding angular momentum. 
In order to be more precise, we define the set of all infinitesimal rigid displacement fields by 
\begin{equation}\label{eq:RR}
\RR := \left\{ x \in \Omega \mapsto Ax + b \in \R^3 \,; A  \in \MM^a_{3} (\R), \; b \in \R^3 \right\},
\end{equation}
where $\MM^a_{3} (\R)$ denotes the set of skew-symmetric $3 \times 3$-matrices with real coefficients,
as well as the  manifold of  infinitesimal rigid displacement fields preserving $\Omega$ by 
\begin{equation}\label{eq:RROmega}
\RR_\Omega = \left\{ R \in \RR \mid   R(x) \cdot n_x  = 0, \; \forall \, x \in \partial\Omega  \right\}.
\end{equation}
When the set $\RR_\Omega$ is not reduced to $\{ 0 \}$, that is when $\Omega$ has rotational symmetries, then for any $R \in \RR_\Omega$, one deduces the conservation of associated angular momentum 
\begin{align*}
&\frac{\d}{\dt} \int_\OO  R(x) \cdot v F \, \dv \, \dx 
= \int_\OO R(x) \cdot v  (- v \cdot \nabla_{x} F + Q(F,F)) \, \dv \, \dx \\
%&= \int_\OO \partial_{x_k} (Ax \cdot v) v_k  F \, \dv \, \dx 
%- \int_\Sigma Ax \cdot v \, { \gamma F } \, n_x \cdot v \, \dv \, \d\sigma_{\! x} \\ %f_{|\Sigma}
%& = \int_{\Sigma_+} Ax \cdot (\VV_x v - v) \gamma_{+}f \, |n_x \cdot v| \, \dv \, \d\sigma_{\! x} \\ 
&\quad  =  \int_\OO F   (v \cdot \nabla_{x}  (R(x) \cdot v)) \, \dv \, \dx  
-  2\int_{\Sigma_+} (R(x) \cdot n_x) \gamma_{+}f \, |n_x \cdot v|^2 \, \dv \, \d\sigma_{\! x} =0,
\end{align*}
because of  \eqref{eq:local-conservations}, the  Green-Ostrogradski formula, the fact that $R(x) = A x$ with $A$ is skew-symmetric, the identity   \eqref{eq:invariantsBoundary3}
and the fact that $R(x)$ is tangential to the boundary. Summing up, in the case of the specular reflection boundary condition ($\iota \equiv 0$), 
the  total energy and the angular momentum associated to  infinitesimal rigid displacement fields preserving $\Omega$ are conserved, namely 
$$
\lla   F(t,\cdot) |v|^2 \rra = \lla   F_0 |v|^2 \rra, \quad \lla   F(t,\cdot) R(x) \cdot v \rra = \lla   F_0 R(x) \cdot v \rra, \ \forall \, R \in \RR_\Omega,
$$
for any $t \ge 0$.  

\smallskip
Finally, using the above recalled microscopic formulation of the Boltzmann H-theorem, we deduce that global equilibria are global Maxwellian distributions that are independent of time and position.
The only mass normalized global Maxwellian distribution which is compatible with the Maxwell reflection condition \eqref{eq:reflection} is the distribution $\mu/|\Omega|$, with $\mu$ defined in \eqref{eq:def_MMM}, 
and we will fix this particular choice of equilibrium in all the paper. 
In view of the above discussion, we introduce the following conditions on the initial datum $F_0$ 
\begin{align}
&\lla   F_0 - \mu \rra =0 , \tag{C1}\label{eq:C1}\\ 
& \lla  (F_0 - \mu) |v|^2 \rra  =   \lla (F_0 - \mu) R(x) \cdot v  \rra  =0, \quad \forall \, R \in \RR_\Omega, \tag{C2}\label{eq:C2}
\end{align}
and we will assume that \eqref{eq:C1} always holds and that \eqref{eq:C2} additionally holds in the case of the specular reflection boundary condition ($\iota \equiv 0$).

%%%%%%%%%%%%%%%%%%%%%%%%%%%%%%%
\subsection{The main results}

In order to state our main result, we need to introduce some functional spaces. For a weight function $\omega : \R^3 \to (0,\infty)$ and an exponent $p \in [1,\infty]$, we define the weighted Lebesgue space 
$L^p_\omega = L^p(\omega) = L^p_\omega(\R^3)$ 
associated to the norm 
$$
\| g \|_{L^p_\omega} = \| \omega g \|_{L^p}, 
$$
and similarly the Lebesgue spaces $L^p_\omega(\OO) = L^p(\Omega;L^p_\omega)$. We fix 
\beqn
\label{eq:def-k0} 
k_0 > 8 + \gamma.
\eeqn

We call admissible weight function $\omega$, a  function
\be\label{eq:omega}
\ba
\omega &= \la v \ra^k := (1 + |v|^2)^{k/2}  \text{ with }   k > k_0 ; \\
\omega &= \exp(\kappa \la v \ra^s)  \text{ with }  s \in (0,2) \text{ and } \kappa > 0, 
\text{ or } s=2 \text{ and }  \kappa \in (0,1/2);
\ea
\ee
and throughout the paper we denote $s=0$ when $\omega$ is a polynomial weight and $k := \kappa s$ when $\omega$ is an exponential weight. For two admissible weight functions $\omega_1$ and $\omega_2$ (or inverse of admissible weight functions), we write $\omega_2 \prec \omega_1$ (or $\omega_1 \succ \omega_2$) if  $\lim_{|v| \to \infty} \frac{\omega_2}{\omega_1}(v) = 0$. Similarly, we write $\omega_2 \preceq \omega_1$ (or $\omega_1 \succeq \omega_2$) if $\omega_2 \prec \omega_1$ or $\lim_{|v| \to \infty} \frac{\omega_2}{\omega_1}(v) \in (0,\infty)$.

For any admissible weight $\omega$ we associate the decay function
\be\label{eq:Theta_omega}
\Theta_{\omega}(t) =
\left\{
\ba
&  C \left( \frac{\log \la t \ra}{\la t \ra} \right)^{\frac{(k - k_0)}{|\gamma|}} ,   & \quad   \text{if } \omega = \la v \ra^{k} \text{ and } \gamma \in [-3,0), \\
&  C \exp\left( - \lambda t \right) ,   & \quad   \text{if } \omega = \la v \ra^{k} \text{ and } \gamma \in [0,1], \\
& C \exp\left( -\lambda \, t^{ \min (1,\frac{s}{|\gamma|} )}  \right),   & \quad\text{if } \omega = e^{\kappa \la v \ra^s},  \\
\ea
\right.
\ee
for some constants $C, \lambda \in (0,\infty)$.  

\smallskip
Our first main result reads as follows. 

\begin{theo}\label{thm:stabNL-inhom}
For any admissible weight function $\omega$ in the sense of  \eqref{eq:omega}, there exists $\eps_0 >0$, small enough, so that, if $\| F_0 - \mu \|_{L^\infty_\omega(\OO)} \le \eps_0$ and 
$F_0$  satisfies the condition~\eqref{eq:C1} (as well as the additional condition  \eqref{eq:C2}  in the specular reflection case $\iota\equiv 0$ in \eqref{eq:reflection}), then there exists a global weak solution $F$ to \eqref{eq:Landau_F}--\eqref{eq:reflect_F} (in a sense which will be specified later) associated to the initial datum $F_0$ such that
\be\label{eq:bound-g-inhom} 
\sup_{t \ge 0} \| F(t) - \mu \|_{L^\infty_\omega(\OO)} \le \eps_0.
\ee
This solution also verifies the decay estimate 
\be\label{eq:decay-g-inhom}
\| F(t) - \mu \|_{L^\infty_{\omega_\sharp} (\OO)} \le   \Theta_{\omega}(t) \, \| F_0 - \mu \|_{L^\infty_\omega(\OO)}, \quad \forall \, t \ge 0,
\ee
with $\omega_\sharp = \omega$ if $\gamma +s \ge 0$ and $\omega_\sharp = \omega_0 := \la v \ra^{k_0}$ if $\gamma +s < 0$.

\end{theo}

We remark that by global weak solution $F$, we mean that the perturbation $f:= F - \mu$ is a global weak solution to the equation~\eqref{eq:Landau_perturb} below in the sense of Theorem~\ref{theo-SG-LLg}. It is worth emphasizing that the small constant $\eps_0$ and the decay function $\Theta_\omega$ are definitively constructive although we will not track the constants along the proof.

\smallskip

The well-posedness and convergence of solutions to collisional kinetic equations in a close-to-equilibirum setting has received a lot of attention in recent years. 
On the one hand, several results were obtained for kinetic equations in the torus. We refer for instance to \cite{MR0363332,UkaiAsano,Caflisch2,Guo-boltzmann,MR3779780} and the references therein for similar results for the cutoff Boltzmann equation. Concerning the Landau equation, we only mention \cite{GuoLandau1,GS1,GS2,CTW,MR3625186,MR4230064} and the references therein. Finally, for the non-cutoff Boltzmann equation we refer to \cite{MR2784329,AMUXY1,AMUXY2,HTT,MR4201411,MR4526062}.

On the other hand, in the case of a bounded domain the literature is scarser. The first results were obtained for the cutoff Boltzmann equation in \cite{MR2679358}, and then extended in \cite{MR3562318,MR3762275, MR3840911}. It was only recently that long-range interactions were considered: The work \cite{MR4076068,Guoetal-specular2} treated the Landau equation with specular boundary condition by introducing an extension method. Very recently, this method was then extended by \cite{deng2023noncutoff} to the non-cutoff Boltzmann equation with Maxwell boundary condition (but excluding the specular case). We also mention the work \cite{MR4755548} which considers conditional regularization of large solutions of the non-cutoff Boltzmann equation.

In particular our result in Theorem~\ref{thm:stabNL-inhom} extends the result of \cite{MR4076068,Guoetal-specular2} to general boundary conditions as well as to larger functional spaces, however we do not prove uniqueness. It is worth emphasizing that our boundary conditions are very general and in particular  we do not impose any restriction on the accommodation coefficient, as it is the case in \cite{MR4076068,Guoetal-specular2,MR2679358,MR3562318}. Our boundary conditions are similar but slightly more general than those considered in the recent paper \cite{deng2023noncutoff}. We also stress on the fact that the conditions on the initial datum $F_0$ are very natural and does not involve velocity derivative as it is the case in  \cite{GuoLandau1,GS1,GS2,MR4076068,Guoetal-specular2}. 
The drawback is that, as in \cite{deng2023noncutoff}, we are not able to prove the uniqueness of the solution for this class of initial data and solutions, but contrarily to \cite{GuoLandau1,MR2679358,MR3779780,MR3625186,MR4076068,Guoetal-specular2}.

\smallskip
As in many previous works, the proof relies on the $L^2$ exponential stability of the Maxwellian equilibrium $\mu$ obtained through hypocoercivity arguments which are by-now available for a general class of Boltzmann like collisional kinetic operators (see e.g.\ \cite{MR4581432}) and on some regularization properties of De Giorgi-Nash-Moser ultracontractivity type available for the Landau equation because of its hypoelliptic nature. These regularization properties
make possible to extend the  exponential stability property to a weighted  $L^\infty$ Lebesgue space and thus to deal with the nonlinearity of the equation. 

\smallskip
Although in many aspects our approach is similar to the one of our previous work \cite{MR3625186} dealing with the torus case, we stress on the two main new ideas that are introduced in the present paper. We will explain them with more details in the Section~\ref{subsec:strategy} below, but we summarize them now:

\smallskip
(1) On the one hand, we introduce a  energy estimate based on new multipliers, a first one being related to Darroz\`es-Guiraud convexity argument \cite{DarrozesG1966,MR1776840,MR2721875}, 
a second one being related to general trace results \cite{MR1765137} (see also \cite{MR4179249}), and a third one being related to Lions-Perthame's multiplier for the gain of velocity moment 
%in averaging lemmas 
\cite{MR1166050,MR3591133}, in order to deal with general reflection condition.  Roughly speaking, this  energy estimate tells us that the density does not concentrate near the boundary. Then this estimate is combined with hypocoercivity result in the spirit of \cite{MR4581432}, De Giorgi-Nash-Moser ultracontractivity result for kinetic Fokker-Planck equation in the spirit of \cite{MR2068847,MR3923847} and  enlargement space for semigroup decay trick in the spirit of \cite{MR3779780,MR4265692,MR3625186} in order to obtain the above mentioned exponential stability in a weighted  $L^\infty$ Lebesgue space.

\smallskip
(2) On the other hand, most of the argument is performed at the level of a linearized problem. The  considered  problem is however a time-dependent perturbation of the linearized equation around the steady state and it is thus different from the   linearized equation around the steady state itself which is usually considered.
The estimates for the time-dependent perturbation problem are not really more complicated to establish than for the linearized problem around the steady state itself, but the former makes possible to get a very direct and simple proof of the existence  and stability result as well as to avoid the control of velocity derivative on the initial datum contrarily to  \cite{MR4076068,Guoetal-specular2}.

\medskip
We next focus on Villani's program \cite{MR2116276,MR2407976} about constructive accurate rate of convergence to the equilibrium for solutions satisfying {\it a priori uniform bounds} in large spaces. 
More precisely, we consider a global weak solution $F$ to   the Landau equation \eqref{eq:Landau_F}--\eqref{eq:reflect_F}, in the sense of Theorem~\ref{theo-SG-LLg}, satisfying  
\beqn\label{eq:intro-iftheo1}
\| \omega_\infty F \|_{L^\infty( (0,\infty) \times \OO)} + \| \omega_\infty F \|_{L^\infty( (0,\infty) ; L^1(\OO))} \le C_0,  \quad
\inf_{(0,\infty) \times \Omega}  \int_{\R^3} F \, \d v  \ge \rho_0,  
\eeqn
for an admissible weight function $\omega_\infty$ and some constants $C_0,\rho_0 \in (0,\infty)$. We also assume that the conclusions \cite[Theorems 2 \& 3]{MR2116276} of the quantitative $H$-theorem theory developed by  Desvillettes and Villani hold true, namely 
\beqn\label{eq:intro-iftheo2}
\| F_t - \mu \|_{L^1(\OO)} \le \eps_1 (t) \to 0, \ \hbox{as} \ t \to \infty, 
\eeqn
for some polynomial function $\eps_1$, although \cite{MR2116276} establishes \eqref{eq:intro-iftheo2} only for the specular reflection boundary condition ($\iota \equiv 0$) but not for a general Maxwell condition (when $\iota\not\equiv0$). Our second main result answers to Villani's program by drastically improving the rate of convergence \eqref{eq:intro-iftheo2} up to the one given by the  linearized regime.

\begin{theo}\label{theo:iftheo} 
Assume $\gamma \in [-3,0]$. Any global weak solution $F$ to   the Landau equation \eqref{eq:Landau_F}--\eqref{eq:reflect_F} satisfying \eqref{eq:intro-iftheo1} and  \eqref{eq:intro-iftheo2}
also satisfies the more accurate decay estimate 
\be\label{eq:decay-g-inhom-bis}
\| F(t) - \mu \|_{L^\infty_{\omega_\sharp} (\OO)} \le   \Theta_{\omega}(t) , \quad \forall \, t \ge 0,
\ee
for any admissible weight function $\omega \prec \omega_\infty$. In the case $s+\gamma \ge 0$, this decay is exponentially fast.

\end{theo} 

It is likely that a variant of this result should be true also for $\gamma \in (0,1]$, but we do not follow this line of research in the present work.

%%%%%%%%%%%%%%%%%%%%%%%%%%%%%%%%%%%% 
\subsection{Strategy of the proof of the main result}\label{subsec:strategy}

Since we are concerned with the existence and long-time behavior of solutions in a regime near to the Maxwellian equilibrium, we introduce a small variation of distribution $f$ defined by  
$$
F = \mu + f.
$$
We next denote by $\CC$ the linearized collision operator
\begin{equation}\label{eq:def_CC}
\CC f = Q(\mu,f) + Q(f,\mu),
\end{equation}
and by $\LL$ the full linearized operator
\begin{equation}\label{eq:def_LL}
\LL f = - v \cdot \nabla_x f +  \CC f,
\end{equation}
so that the perturbation $f$ verifies the equation
\begin{equation}
\label{eq:Landau_perturb}
\left\{
\begin{aligned}
& \partial_t f =  \LL f + Q(f,f) \quad &\text{in}&\quad (0,\infty) \times \OO \\
& \gamma_{-} f  = \RRR \gamma_{+}  f \quad &\text{on}&\quad (0,\infty) \times \Sigma_{-} \\
& f_{| t=0} = f_0 ,
\end{aligned}
\right.
\end{equation}
with initial datum $f_0 = F_0 - \mu$ satisfying \eqref{eq:C1} (as well \eqref{eq:C2} in the specular reflection case $\iota\equiv 0$).
We then observe that, from \eqref{eq:local-conservations}, we have 
$$
\pi Q(f,f) = 0, 
$$
where $\pi$ stands the projector onto $\operatorname{Ker}(\CC) = \operatorname{span} \{ \mu , v_1 \mu , v_2 \mu , v_3 \mu , |v|^2 \mu  \}$  given by 
\beqn\label{def:pi}
\begin{aligned}
\pi f (x,v) 
&= \left( \int_{\R^3} f(x,w) \, \d w \right) \mu (v)
+ \left( \int_{\R^3} w \, f(x,w) \, \d w  \right) \cdot v \mu (v) \\
&\quad
+ \left( \int_{\R^3} \frac{|w|^2-3}{\sqrt{6}} \, f(x,w)  \, \d w   \right) \frac{(|v|^2-3)}{\sqrt{6}} \, \mu(v).
\end{aligned}
\eeqn 
As a consequence, the first equation in \eqref{eq:Landau_perturb} also writes 
$$
\partial_t f = \LL_f f, 
$$
with 
$$
\LL_g f := \LL f + Q^\perp(g,f),
$$
where we have set $ Q^\perp (g,f) :=  ( I - \pi) Q(g,f)$.

\medskip
For a given function $g=g(t,x,v)$ and for any $t_0 \ge 0$, we shall first consider the linear equation associated 
 to the operator $\LL_g$ defined by 
\begin{equation}\label{eq:linear_g}
\left\{
\begin{aligned}
& \partial_t f = \LL_g f \quad &\text{in}&\quad (t_0,\infty) \times \OO \\
%\LL f + Q(g,f) =: \LL_g f \quad &\text{in}&\quad (0,\infty) \times \Omega \times \R^3 \\
& \gamma_{-} f  = \RRR \gamma_{+}  f \quad &\text{on}&\quad (t_0,\infty) \times \Sigma_{-} \\
& f_{| t=t_0} = f_{t_0}  \quad &\text{in}&\quad   \OO.
\end{aligned}
\right.
\end{equation}
We introduce a splitting of the operator $\LL_g = \BB_g + \AA_g$, where we  define the dissipative part by 
\beqn\label{def:BBg}
\BB_g f :=  - v \cdot \nabla_x f + Q(\mu,f) + Q(g,f) - M \chi_R f ,
\eeqn
and the remainder part, which takes into account zero order and integral terms, by
\beqn\label{def:AAg}
\AA_g f :=   Q(f,\mu) - \pi Q(g,f) + M \chi_R f ,
\end{equation}
for some compactly supported smooth function $M \chi_R$ with constants $M,R>0$ to be chosen, namely $\chi_R(v) = \chi (v/R)$ for $\chi \in C^\infty_c(\R^3)$ such that $\mathbf 1_{B_1} \le \chi \le \mathbf 1_{B_2}$. We shall also consider the linear equation \eqref{eq:linear_g} associated to the operator $\BB_g$ instead of~$\LL_g$.

\smallskip
From now on, we fix some weight function 
\begin{equation}\label{eq:defomega0}
\omega_0 := \langle v \rangle^{k_0} ,
%\quad\text{with}\quad \Red j_0 \in (8+\gamma,9+\gamma),
\end{equation}
with $k_0$ defined in \eqref{eq:def-k0}, and we define the space 
$$
\XX_0 := L^\infty_{\omega_0}((0,\infty) \times \OO).
$$
We denote by $P_v$  the projection operator  on the $v$-direction for any given $v \in \R^3 \backslash \{ 0 \}$ defined by 
\beqn\label{eq:def:Pv}
P_v \xi = \left( \xi \cdot \frac{v}{|v|}   \right) \frac{v}{|v|}, \quad \forall \, \xi \in \R^3 ,
\eeqn
and we denote by $\widetilde \nabla_v f$  the anisotropic gradient of a function $f$ 
defined by 
\beqn\label{eq:def:anisoG}
\widetilde \nabla_v f = P_v \nabla_v f + \la v \ra (I-P_v) \nabla_v f.
\eeqn
We next define the {\it dissipation norm} $H^{1,*}_\omega$ associated to the norm of $L^2_\omega$ by 
$$
\| f \|_{H^{1,*}_v (\omega)}^2 :=   \| \la v \ra^{\frac{\gamma}{2}+\frac{s}2} f \|_{L^2_v(\omega)}^2  + \| \la v \ra^{\frac{\gamma}{2}} \widetilde \nabla_v (f\omega) \|_{L^2_v }^2,
$$
where we recall that $s=0$ when $\omega$ is a polynomial weight function.

\smallskip

At least for $g \in \XX_0$ small enough, we successively establish the following properties for both non-autonomous semigroups $S_{\LL_g} = S_{\LL_g}(t,t_0)$ and $S_{\BB_g} = S_{\BB_g}(t,t_0)$ associated to the 
 above equations. 

\smallskip\noindent
{\sl (1) The semigroup $S_{\BB_g}$ is bounded.} For any admissible weight function $\omega$ and exponent $p \in [1,\infty]$, there holds
\beqn\label{eq:SBBg_bd}
S_{\BB_g} : L^p_\omega(\OO) \to L^p_\omega(\OO), \hbox{ uniformly bounded}.
\eeqn
More precisely,  thanks to a multiplier trick, we exhibit an equivalent weight function $\widetilde\omega$ such that 
$S_{\BB_g}$ is a semigroup of contractions on $L^p_{\tilde \omega}(\OO)$, see Proposition~\ref{prop:BBg_Lp}. 

\smallskip\noindent
{\sl (2) The semigroup $S_{\BB_g}$ is ultracontractive.} For a class of admissible weight functions $\omega_2$ and $\omega$, there holds
\beqn\label{eq:SBBg_ultra}
S_{\BB_g}(t,t_0) : L^2_{\omega}(\OO) \to L^\infty_{\omega_2}(\OO), \hbox{ with bound } \OO((t-t_0)^{-\eta}), 
\eeqn
for any $t > t_0 \ge 0$ and for some $\eta>0$. Modifying again the weight function, we are indeed able to exhibit  a dissipation estimate associated to the $L^2(\widetilde\omega)$ norm which prevents the concentration near the boundary of the solution to the linear problem associated to $\BB_g$. Together with available gain of integrability estimates in the interior \cite{MR2068847,MR3923847} in the spirit of De Giorgi-Nash-Moser theory for parabolic equations, we then establish that $S_{\BB_g}$ is ultracontractive, see Theorem~\ref{theo:BBg_L2_Linfty}.

\smallskip\noindent  
{\sl (3) The operator $\LL_g$ is (weakly) hypocoercive:} % in $L^2(\omega_{eq})$, $\omega_{eq} := \mu^{-1/2}$:} 
there exist a constant $ \sigma_0 > 0$ and a twisted Hilbert norm $\Nt \cdot \Nt_{L^2(\mu^{-1/2})}$, %_{\omega_{eq}}(\OO)}$, 
equivalent to the usual $L^2_{xv} (\mu^{-1/2})$-norm such that for the associated scalar product $(\!( \cdot, \cdot  )\!)_{L^2(\mu^{-1/2})}$, %_{\omega_{eq}}}$, 
we have 
\beqn\label{eq:LLgff_L2mu}
(\!( \LL_g f,f  )\!)_{L^2(\mu^{-1/2})}  \le - \sigma_0 \| f \|^2_{L^2_x H^{1,*}_{v}(\mu^{-1/2})} , 
\eeqn
for any $f$ in the domain of $\LL_g$, see Theorem~\ref{theo:hypo}.

\smallskip\noindent
{\sl (4) The semigroup $S_{\LL_g}$ is decaying and enjoys compactness properties.}  For any admissible weight function $\omega$, there holds
\beqn\label{eq:SLLg_decay}
\| S_{\LL_g} (t,\tau) f_\tau \|_{L^\infty_{\omega_\sharp} (\OO)} \le C  \Theta_{\omega}(t - \tau) \| f_\tau \|_{L^\infty_{\omega}} \,  , \quad \forall \, t\ge \tau \ge 0, \quad
\forall f_\tau \in L^\infty_\omega,
\eeqn
with the notations of \eqref{eq:decay-g-inhom}, see Theorem~\ref{theo:LLg_Linfty}. That last estimate follows from the three previous steps together with an extension trick in the spirit of \cite{MR3779780,MR4265692,MR3625186}. 
There also holds, for any $T>0$, 
\beqn\label{eq:SLg-L2H1*}
 \int_0^T \| S_{\LL_g}(t,\tau) f_\tau \|_{L^2_xH_v^{1,*}(\omega)}^2  \, \d t \le   C(T) \| f_\tau \|^2_{L^2(\omega)}, 
\eeqn
as a consequence of a variant estimate of \eqref{eq:LLgff_L2mu}, from which we deduce a compactness property in $L^2$ thanks to a Aubin-Lions type argument, see Theorem~\ref{theo-SG-LLg}.

\smallskip\noindent
{\sl (5) Conclusion.} We finally consider the mapping 
$$
g \mapsto S_{\LL_g} f_0,
$$ 
for which we deduce from the last step that it leaves invariant a small ball of $L^\infty_{\omega_0}( (0,\infty) \times \OO)$ and it is continuous for the weak topology. 
We conclude to the existence of a fixed-point for that mapping thanks to the Schauder-Tychonoff fixed-point theorem and thus a solution to equation~\eqref{eq:Landau_perturb} which satisfies the announced decay property in Theorem~\ref{thm:stabNL-inhom}. 

\smallskip
The proof of Theorem~\ref{theo:iftheo} uses similar arguments as those described above.

%%%%%%%%%%%%%%%%%%%%%%%%%%%%%%%%%%%
\subsection{Structure of the paper} 

In Section~\ref{sec-toolbox}, we recall some more or less standard results we use in the next sections. In Section~\ref{sec-SLLg-first}, we establish some a priori bound in $L^p_\omega$ for the solutions to the linear problem \eqref{eq:linear_g} and we deduce the existence of an associated semigroup $S_{\LL_g}$. In Section~\ref{sec-SBBg-decay}, we establish the bound  \eqref{eq:SBBg_bd} and we deduce a decay estimate of the form \eqref{eq:SLLg_decay} for the  semigroup $S_{\BB_g}$. In Section~\ref{sec-SBBg-ultracontractivity}, we establish the ultracontractivity estimate  \eqref{eq:SBBg_ultra} for the  semigroup $S_{\BB_g}$. In Section~\ref{sec-LLg-hypo}, we establish the hypocoercivity estimate  \eqref{eq:LLgff_L2mu} for the  operator $\LL_g$. 
In Section~\ref{sec-proof-LLg}, we establish the decay property \eqref{eq:SLLg_decay} on $S_{\LL_g}$. We finally prove the main results Theorem~\ref{thm:stabNL-inhom} and Theorem~\ref{theo:iftheo} in the last Section~\ref{sec-proofMainTheo}.

%%%%%%%%%%%%%%%%%%%%%%%%%%%%%%%%%%%%%%%%%%%%%
\section{Toolbox}
\label{sec-toolbox}

We introduce in this section some more or less classical material that we will use several times in the sequel.

%%%%%%%%%%%%%%%%%
%\subsection{Notations}
%
%Throughout the article we shall write $a \lesssim b$ or $b \gtrsim a$ if there is some constant $C>0$ such that $a \le C b$; and $a \simeq b$ if both $a \lesssim b$ and $b \lesssim a$ hold.
%

%%%%%%%%%%%%%%%%%%%%%%%%%%%%%%%%%%%%%%%%%%%%%
\subsection{Estimates for the collision operator}

We recall some (variants of) classical results on the Landau collision operator.  
We denote
\be\label{eq:barabc}
\bar a_{ij} = a_{ij}*\mu, \quad
\bar b_i = b_i * \mu, \quad
\bar c = c*\mu,  
\ee
where $*$ stand for the convolution in the velocity variable $v$, and we remark in particular that  
\begin{equation}\label{eq:barc=-3}
\bar c  = -8\pi \mu  \quad\hbox{when} \quad \gamma = - 3.
\end{equation}

We recall the following result from \cite[Propositions 2.3 and 2.4]{MR1463805} and \cite[Lemma 3]{GuoLandau1} (see also \cite[Lemma 2.1(e)]{MR3625186}). 

%(see also \cite[Lemma~2.5]{MR3407515} and  \cite[Lemma~2.5]{MR3365830}).

\begin{lem}\label{lem:elementary-abc*bar} 
The matrix $\bar a(v)$ has a simple eigenvalue $\ell_1(v)>0$ associated with the eigenvector $v$ and a double eigenvalue $\ell_2(v)>0$ associated with the eigenspace $v^{\perp}$, 
so that 
$$
\bar a_{ij} \xi_i \xi_j = \ell_1(v) |P_v\xi|^2 + \ell_2(v) |(I-P_v)\xi|^2. 
$$
Furthermore, when $|v|\to +\infty$,  we have
$$
\ba
\ell_1(v) \sim  2 \la v \ra^\gamma, \quad
\ell_2(v) \sim  \la v \ra^{\gamma+2}, 
\ea
$$
and thus %, when $|v|\to +\infty$,  we have
$$
\bar a_{ij}v_iv_j \sim 2  \la v \ra^{\gamma+2}, \quad \bar a_{ii} \sim 2 \la v \ra^{\gamma+2}. 
%= \ell_1(v) + 2 \ell_2(v) \sim 2 \la v \ra^{\gamma} + 2 \la v \ra^{\gamma+2}. % \approx_{|v| \to \infty} 
%\bar a_{ii} = \ell_1(v) + 2 \ell_2(v) \sim 2 \la v \ra^{\gamma} + 2 \la v \ra^{\gamma+2}. % \approx_{|v| \to \infty} 
$$
On the other hand, there hold
\beqn\label{eq:barb}
  \bar b  =  - \ell_1(v) v  
\eeqn
and
\beqn\label{eq:barc>-3}
\bar c \sim 
\begin{cases}
-2(\gamma+3)  \langle v \rangle^\gamma \quad &\text{if}\quad  \gamma \in (-3,1] , \\
-8\pi \mu(v) \quad &\text{if}\quad  \gamma = -3,
\end{cases}
\eeqn
when $|v|\to +\infty$. % and   $\gamma \in (-3,1]$. 
\end{lem}
Introducing the symmetric matrix
$$
\mathbf{B}(v) := \sqrt{\ell_1(v)} \frac{v}{|v|} \otimes  \frac{v}{|v|} + \sqrt{\ell_2(v)} \left( I - \frac{v}{|v|} \otimes  \frac{v}{|v|} \right), 
$$
we see from the above discussion that 
\beqn\label{eq:Bnabla=tildenabla}
|\mathbf{B} \nabla_v f |^2 = \ell_1(v) |P_v\nabla_v f|^2 + \ell_2(v) |(I-P_v)\nabla_v f|^2 \simeq |\langle v \rangle^{\gamma/2} \widetilde \nabla_v f|^2. 
\eeqn

We reformulate \cite[Lemma 3.4]{CTW} %(see also \cite[Lemma~2.4]{MR4438904}) 
and part of \cite[Lemmas 2.4 and 4.1]{MR4438904}.

\begin{lem}\label{lem:elementary-abc*g}  
For any $g \in L^\infty(\omega_0)$, there hold
\bear
\label{aijg}
&&|(a_{ij}*g)| + |(a_{\ell j}*g)v_\ell | + |(a_{\ell,k}*g)v_\ell v_k| \lesssim \la v \ra^{\gamma+2} \| g \|_{L^\infty_{\omega_0}}, 
\\
\label{eq:b*g}
&&
%|(b_\ell*g)v_\ell | + |b_i*g| \langle v \rangle \lesssim \la v \ra^{\gamma+2}  \| g \|_{L^\infty_{\omega_0}}, %\Red \quad |(c*g)| \lesssim \la v \ra^{\gamma}\| g \|_{L^\infty_{\omega_0}}, 
 |b_i*g| \lesssim \la v \ra^{\gamma+1}  \| g \|_{L^\infty_{\omega_0}}, %\Red \quad |(c*g)| \lesssim \la v \ra^{\gamma}\| g \|_{L^\infty_{\omega_0}}, 
\\
\label{eq:c*g}
&& |c*g| \lesssim \la v \ra^{\gamma}\| g \|_{L^\infty_{\omega_0}}, 
 \eear
for any $v \in \R^3$ and $i,j = 1,2,3$.
%, as well as 
%\be\label{eq:c*g}
%\ba
% |(c*g)| &\lesssim \la v \ra^{\gamma}\| g \|_{L^\infty_{\omega_0}}, 
%% |(c*g)| &\lesssim \la v \ra^{\gamma}\| g \|_{L^\infty_{\omega_0}} \ \hbox{ when } \ \gamma \in (-3,1], \\
%%|(c*g)| &\lesssim \omega_0^{-1} \| g \|_{L^\infty_{\omega_0}} \ \hbox{ when } \ \gamma = 3, 
%\ea
%\ee
%for any $v \in \R^3$. 
Considering additionally some vector fields $F$ and $H$,  there holds
\beqn\label{aijFiHj}
|(a_{ij}*g) F_i H_j| \lesssim   \| g \|_{L^\infty_{\omega_0}} | \mathbf{B}(v) F | | \mathbf{B}(v) H|.
\eeqn
\end{lem}

\begin{proof}[Proof of Lemma~\ref{lem:elementary-abc*g}]
Thanks to \cite[Lemma 3.4]{CTW} and \cite[Lemmas 2.4]{MR4438904}, when $\gamma \in [-2,1]$ we have
$$
|(a_{ij}*g)| + |(a_{ij}*g)v_i| + |(a_{ij}*g)v_iv_j| \lesssim \la v \ra^{\gamma+2} \| g \|_{L^2_v(\la v \ra^{\gamma+11/2+0})}
$$
as well as
$$
|(b_i*g)v_i| + |b_i*g| \langle v \rangle \lesssim \la v \ra^{\gamma+2} \| g \|_{L^2_v(\la v \ra^{\gamma+11/2+0})}, %\quad \Red  |(c*g)| \lesssim \la v \ra^{\gamma} \| g \|_{L^2_v(\la v \ra^{\gamma+11/2+0})} ,
$$
and we conclude to \eqref{aijg} and \eqref{eq:b*g} thanks to the embedding $L^\infty_{\omega_0} \subset L^2_v(\la v \ra^{\gamma+11/2+0})$. In the case $\gamma \in [-3,-2]$, estimates \eqref{aijg} and \eqref{eq:b*g} are proven in \cite[Lemma~4.2]{MR3625186}.

The proof of \eqref{eq:c*g} when $\gamma = -3$ is straightforward from the very definition of $\omega_0$ in \eqref{eq:defomega0}. We next assume  $\gamma \in (-3,0)$. When $|v| \ge 1$, we proceed similarly as in the proof of \cite[Lemma~2.1(e)]{MR3625186} by introducing the splitting 
$$
|c * g| \lesssim \| g \|_{L^\infty_{\omega_0}} \left\{ \int_{|v-v_*| \le |v|/2} \frac{|v-v_*|^\gamma }{ \langle v_* \rangle^{k_0}} \, \d v_* + \int_{|v-v_*| > |v|/2} \frac{|v-v_*|^\gamma }{ \langle v_* \rangle^{k_0}}  \, \d v_* \right\}.
$$
For  the first term, we have $|v_*| > |v|/2$ on the domain of integration, so that 
\bean
 \int_{|v-v_*| \le |v|/2} \frac{|v-v_*|^\gamma }{ \langle v_* \rangle^{k_0}} \, \d v_*
&\le&
\frac{1}{\langle v/2 \rangle^{k_0}}  \int_{|v-v_*| \le |v|/2} |v-v_*|^\gamma \, \d v_*
\\
&\lesssim&
\frac{ |v|^{3+\gamma}  }{ \langle v  \rangle^{k_0}}  \lesssim \langle v \rangle^\gamma, 
\eean
because $k_0 > 3$. For  the second term, we have 
\bean
 \int_{|v-v_*| > |v|/2} \frac{|v-v_*|^\gamma }{ \langle v_* \rangle^{k_0}} \, \d v_*
&\le&
|v/2|^\gamma
 \int_{\R^3} \frac{\d v_* }{ \langle v_* \rangle^{k_0}} \lesssim \langle v \rangle^\gamma. 
 \eean
For $|v| \le 1$, we just write 
$$
|c * g| \lesssim \| g \|_{L^\infty_{\omega_0}}  \int_{\R^3} \frac{|v_*|^\gamma }{ \langle v-v_* \rangle^{k_0}} \, \d v_*
 \lesssim \| g \|_{L^\infty_{\omega_0}}  \int_{\R^3} \frac{|v_*|^\gamma }{ \langle  v_* \rangle^{k_0}} \, \d v_*
 \lesssim \| g \|_{L^\infty_{\omega_0}}.
$$
We conclude the proof of \eqref{eq:c*g} in the case  $\gamma \in (-3,0)$ by gathering these estimates. When $\gamma \in [0,1]$, we write 
$$
|c * g| \lesssim \| g \|_{L^\infty_{\omega_0}}   \int_{\R^3}  \frac{ |v|^\gamma + |v_*|^\gamma  }{ \langle v_* \rangle^{k_0}} \, \d v_*, 
$$
and we immediately deduce \eqref{eq:c*g} by observing that $\gamma-k_0 < -3$. The proof of \eqref{aijFiHj} follows from \eqref{aijg} exactly as in the proof of \cite[Lemma~4.1]{MR4438904}.
\end{proof}

We define 
$$
\AA_0 f : = Q(f,\mu) = (a_{ij} * f) \partial^2_{v_iv_j} \mu - (c*f)\mu.
$$
We recall the results of  \cite[Lemma~2.12]{CTW} and  \cite[Lemma~2.5]{MR3625186}.

\begin{lem}\label{lem:borneA0}
For any admissible weight function $\omega$ and any exponent $p \in [1,\infty]$, there holds
$$
  \AA_0 : L^p(\omega) \to L^p(\mu^{-\vartheta}), \quad \forall \, \vartheta \in (0,1), 
$$
In particular, we have 
\beqn\label{eq:AA0L2L2}
 | \langle \AA_0 f, f \rangle_{L^2_\omega} | \lesssim \| \langle v \rangle^{(\gamma-1)/2} f \|^2_{L^2_{\omega}} .
\eeqn
\end{lem}

We state now some variants of well-known estimates on the Landau operator.

\begin{prop}\label{prop:EstimOpLandau3}
For any admissible weight function $\omega$ as defined in \eqref{eq:omega}, there holds
\begin{equation}\label{eq:Qgff_omega}
 \la Q(g,f) , f \ra_{L^2_v (\omega)}  \lesssim \| g \|_{L^\infty_{\omega_0}}   \| f \|_{H^{1,*}_v (\omega)}^2,
\end{equation}
and
\begin{equation}\label{eq:Qgfh_omega}
\begin{aligned}
 \left| \langle   Q(g,f) , h \ra_{L^2_v (\omega)} \right|  &\lesssim \| g \|_{L^\infty_{\omega_0}} \|  \langle v \rangle^{\gamma/2} f \|_{L^2_v(\omega)}   \Bigl( \| \nabla^2_v(h \omega) \|_{L^2_v (\la v \ra^{\gamma/2+2})} \\
&\quad
+ \| \nabla_v(h \omega) \|_{L^2_v (\la v \ra^{\gamma/2+1})} 
+ \|\omega h \|_{L^2_v (\la v \ra^{\gamma/2+s})} \Bigr).
\end{aligned}
\end{equation} 
In particular we have
\begin{equation}\label{eq:piQgf_omega}
\left| \langle \pi Q(g,f) , f \ra_{L^2_v (\omega)} \right|   \lesssim \| g \|_{L^\infty_{\omega_0}}   \| \langle v \rangle^{\gamma/2}f  \|_{L^2_v (\omega)}^2
\end{equation} 
and
\begin{equation}\label{eq:Qperpgf_omega}
\left| \langle Q^\perp(g,f) , f \ra_{L^2_v (\omega)} \right|   \lesssim \| g \|_{L^\infty_{\omega_0}}   \| f \|_{H^{1,*}_v (\omega)}^2.
\end{equation} 

\end{prop}

\begin{proof}[Proof of Proposition~\ref{prop:EstimOpLandau3}]
%We almost repeat the proof of \cite[Proposition~4.2, eq.~(4.14)]{MR4438904}. 
Using the shorthands
$$
\widetilde a_{ij} := a_{ij} * g, \quad 
\widetilde b_i  :=  b_i* g, \quad 
\widetilde c  :=  c* g, \quad
\partial_{v_i} \omega = v_i \wp \omega, \quad \wp  :=  k \langle v \rangle^{s-2}, 
$$
with the same conventions for $k$ and $s$ as in \eqref{eq:omega}, we split the proof into three steps.

\medskip\noindent
\textit{Step 1.} We first write, using the formulation \eqref{eq:oplandau1} for $Q(g,f)$ and one integration by parts,
\bean
\la Q(g,f) , f \ra_{L^2_\omega} 
&=& \int_{\R^3}\partial_{v_i} \left\{  \widetilde a_{ij} \partial_{v_j} f - \widetilde b_i  f \right\} f \, \omega^2 \, \d v
\\
&=&  - \int_{\R^3} \widetilde a_{ij} \left\{ \partial_{v_j} (f\omega) - f \omega v_j \wp \right\}  \left\{  \partial_{v_i} (f \omega)  +  (f \omega) v_i \wp  \right\}  \, \d v  \\
&& + \int_{\R^3}   \widetilde b_i  f \omega   \left\{  \partial_{v_i} (f \omega)  +  (f \omega) v_i \wp  \right\}  \, \d v,
\eean
from which we get, performing another integration by parts in the first term of the second integral,
\bean
 \la Q(g,f) , f\ra_{L^2_\omega} 
   &=& 
   - \int_{\R^3} \widetilde a_{ij}   \left\{ \partial_{v_j} (f\omega) - f \omega v_j \wp \right\} \left\{ \partial_{v_i} (f\omega) + f \omega v_i \wp \right\}     \, \d v
\\
 && - \frac12  \int_{\R^3}   \widetilde c  (f \omega)^2    \, \d v
   +  \int_{\R^3}   \widetilde b_i    (f \omega)^2 v_i \wp    \, \d v.
\eean
Using \eqref{aijFiHj}, \eqref{eq:Bnabla=tildenabla}, $|\mathbf{B}(v) v| = |v| \sqrt{\ell_1(v)}$ and Lemmas~\ref{lem:elementary-abc*bar} and~\ref{lem:elementary-abc*g}, we have 
\bean
&&\left|  \widetilde a_{ij}  \left\{ \partial_{v_j} (f\omega) - f \omega v_j\wp  \right\}    \left\{  \partial_{v_i} (f \omega)  +  f \omega v_i \wp  \right\} \right|
\\
&&\quad \lesssim \| g \|_{L^\infty_{\omega_0}} \left(|\mathbf{B} \nabla(f\omega)|^2 + |\mathbf{B} v|^2 |f \wp \omega|^2 \right) \\
&&  \quad \lesssim \| g \|_{L^\infty_{\omega_0}}   \left( \langle v \rangle^\gamma |\widetilde \nabla(f\omega)|^2 + \langle v \rangle^{\gamma + 2s-2} |f\omega|^2 \right) .
 \eean
We deduce \eqref{eq:Qgff_omega} thanks to Lemma~\ref{lem:elementary-abc*g}.
  
\medskip\noindent
\textit{Step 2.} We now use the formulation \eqref{eq:oplandau3} for $Q(g,f)$ to write
\bean
\la Q(g,f) , h \ra_{L^2_\omega} 
&=& \int_{\R^3} \left\{ \partial_{v_i,v_j} (\widetilde a_{ij}  f) - 2 \partial_{v_i} (\widetilde b_i  f) \right\}  h \, \omega^2 \, \d v
\\
&=& \int_{\R^3}   \widetilde a_{ij}  f \omega  \left\{ \partial_{v_i,v_j} ( h  \omega)  + 2\partial_{v_j} ( h  \omega) v_i \wp  + h  \partial_{v_i}(v_j \wp \omega )  \right\} \, \d v
\\
&&+2 \int_{\R^3}  \widetilde b_i  f \omega \left\{ \partial_{v_i} ( h  \omega)  +  h \omega v_i \wp  \right\} \, \d v .
\eean
We observe that
\bean
&&\left|  \widetilde a_{ij}  \left\{ \partial_{v_i,v_j} ( h  \omega)  + 2\partial_{v_j} ( h  \omega) v_i \wp  + h  \partial_{v_i}(v_j \wp \omega ) \right\}     \right|
\\
&&\quad \lesssim \| g \|_{L^\infty_{\omega_0}} \la v \ra^{\gamma+2} \left(  |\nabla_v^2 (h \omega)| + \la v \ra^{s-2} |\nabla_v (h \omega)|  + \la v \ra^{s-2} |h \omega| \right)  ,
\eean
and we conclude to \eqref{eq:Qgfh_omega} by using Lemma~\ref{lem:elementary-abc*g}, writing $\omega = (\omega \la v \ra^{\gamma/2}) \la v \ra^{-\gamma/2}$ and applying the Cauchy-Schwarz inequality.

\medskip\noindent
\textit{Step 3.}
Observe now that  from \eqref{eq:Qgfh_omega}, for any polynomial function $\xi=\xi(v)$ such that $\xi \omega^{-1} \in L^2_v(\la v \ra^{\gamma/2+s})$, $\nabla_v(\xi \omega^{-1}) \in L^2_v(\la v \ra^{\gamma/2+1})$ and $\nabla_v^2(\xi \omega^{-1}) \in L^2_v(\la v \ra^{\gamma/2+2})$, we have
\begin{equation}\label{eq:Qgf-moments}
\left| \int_{\R^3} \xi(v) Q(g,f)(v) \, \d v \right|
= \la Q(g,f) ,  \xi \omega^{-2} \ra_{L^2_\omega} 
\lesssim \| g \|_{L^\infty_{\omega_0}} \| \la v \ra^{\gamma/2} f \|_{L^2_v(\omega) } .
\end{equation}
We finally write
$$
\la \pi Q(g,f) , f \ra_{L^2_v(\omega)} \lesssim \| \la v \ra^{-\gamma/2} \pi Q(g,f) \|_{L^2_v(\omega)} \| \la v \ra^{\gamma/2} f \|_{L^2_v(\omega)}
$$
and observe that from the very definition of $\pi$ in \eqref{def:pi}
$$
\| \la v \ra^{-\gamma/2} \pi Q(g,f) \|_{L^2_v(\omega)} \lesssim \sum_{i=0}^4 \left| \int_{\R^3} \xi_i(v) Q(g,f)(v) \, \d v \right|, 
$$
with $\xi_0 := 1$, $\xi_i := v_i$, $i=1,2,3$, $\xi_4 := |v|^2$, which implies~\eqref{eq:piQgf_omega}. Recalling the definition of $Q^\perp = (I-\pi) Q$, we thus deduce \eqref{eq:Qperpgf_omega} from the estimates~\eqref{eq:Qgff_omega} and \eqref{eq:piQgf_omega}.
\end{proof}

%%%%%%%%%%%%%%%%%%%%%%%%%%%%%%%%%%%%%%%%% 
\subsection{Estimates for second order linear operators}

Consider the parabolic operator $\mathbf L$ acting only on the velocity variable $v \in \R^3$ defined by
\begin{equation}\label{eq:OperatorbfL}
\mathbf{L} g = \sigma_{ij} \partial_{v_i, v_j}g + \nu_i \partial_{v_i} g + \eta g, 
\end{equation}
where $\sigma_{ij}=\sigma_{ij}(v)$ is a symmetric matrix, $\nu_i =\nu_i(v)$ a vector field and $\eta=\eta(v)$ a scalar function, and we use the convention of summation over repeated indices. We observe that the dual operator of $\mathbf{L}$ is
\begin{equation}
\mathbf{L}^* h = \sigma_{ij} \partial_{v_i, v_j} h + \left(2 \partial_{v_j} \sigma_{ij} -\nu_i \right)\partial_{v_i} h + \left( \partial_{v_i, v_j} \sigma_{ij} - \partial_{v_i} \nu_i + \eta \right) h.
\end{equation}

We present a variant of \cite[Lemma~3.8]{MR3779780}, \cite[Lemma~3.8]{MR3488535}, \cite[Lemma~2.1]{MR4265692}, see also \cite[Lemma 7.7]{sanchez:hal-04093201}.

\begin{lem}\label{lem:dissipativity_Lp}
For any $p\in[1,+\infty)$ and any weight function $\omega=\omega(v)$, there holds 
$$
\begin{aligned}
\int_{\R^3} (\mathbf{L} g) |g|^{p-2} g \omega^p \, \d v 
%&= - \frac{4(p-1)}{p^2} \int \sigma_{ij} \partial_{v_i} (g|g|^{p/2-1}) \partial_{v_j} (g|g|^{p/2-1}) \omega^p \, \d v + \int \varpi_{\omega,p}^0 |g|^p \omega^p \, \d v \\
&= - \frac{4(p-1)}{p^2} \int_{\R^3} \sigma_{ij} \partial_{v_i} G \partial_{v_j} G \, \d v + \int_{\R^3} \varpi_{\omega,p}^\mathbf{L} \, |g|^p \omega^p \, \d v,
\end{aligned}
$$
with $G := \omega^{p/2} g|g|^{p/2-1}$ and 
\begin{equation}\label{eq:varpiL_omega_p}
\begin{aligned}
\varpi_{\omega,p}^\mathbf{L} (v) 
&= 2\left( 1- \frac{1}{p} \right) \sigma_{ij} \frac{\partial_{v_i} \omega}{\omega} \frac{\partial_{v_j} \omega}{\omega} + \left( \frac{2}{p}-1 \right)\sigma_{ij} \frac{\partial_{v_i, v_j} \omega}{\omega} + \frac{2}{p} \partial_{v_j} \sigma_{ij} \frac{\partial_{v_i} \omega}{\omega} \\
&\quad
- \nu_i \frac{\partial_{v_i} \omega}{\omega} + \frac{1}{p} \partial_{v_i, v_j} \sigma_{ij} - \frac{1}{p} \partial_{v_i} \nu_i + \eta.
\end{aligned}
\end{equation}
\end{lem}

\begin{rem} 
We also define $\varpi_{\omega,\infty}^\mathbf{L}$ by the above formula \eqref{eq:varpiL_omega_p} with the convention $1/\infty=0$.
\end{rem}

\begin{proof}[Proof of Lemma~\ref{lem:dissipativity_Lp}]
Setting $\Phi'(s)=|s|^{p-2} s$, we compute
$$
\begin{aligned}
\int_{\R^3} (\mathbf{L} g) \Phi'(g) \omega^p \, \d v 
&= \int_{\R^3} \sigma_{ij} \partial_{v_i, v_j}g \Phi'(g) \omega^p \, \d v  + \int_{\R^3} \nu_i \partial_{v_i} g \Phi'(g) \omega^p \, \d v 
+ \int_{\R^3} \eta g \Phi'(g) \omega^p \, \d v  \\
&=: T_1 + T_2 + T_3,
\end{aligned}
$$
and we denote $h = \omega g$ in the sequel. For the term $T_2$, we write $ \partial_{v_i} (h \omega^{-1}) = \omega^{-1} \partial_{v_i} h -  \omega^{-2} h \partial_{v_i} \omega$, and thus
$$
\begin{aligned}
T_2 
&= 
\int_{\R^3} (\nu_i \partial_{v_i} h) \, \Phi'(h) \, \d v   - \int_{\R^3} (\nu_i \partial_{v_i} \omega) \,  \omega^{-1}\, h \Phi'(h) \, \d v  \\
%&= 
%\frac1p \int_{\R^3} \nu_i \partial_{v_i} (|h|^p)  - \int \nu_i \partial_{v_i} \omega \,  \omega^{-1}\, |h|^p \, \d v \\
&=
  -\frac1p \int_{\R^3} (\partial_{v_i} \nu_i)  \, |h|^p \, \d v 
- \int \left( \nu_i \frac{\partial_{v_i} \omega}{\omega} \right)   |h|^p \, \d v ,
\end{aligned}
$$
thanks to an integration by parts in last line.

For the term $T_1$, we use integration by parts to obtain
$$
\begin{aligned}
T_1 & = 
\int_{\R^3} \sigma_{ij}\partial_{v_i, v_j}(h \omega^{-1}) \Phi'(h) \omega \, \d v   \\
&= 
- \int_{\R^3} \partial_{v_i} (h m^{-1}) (\partial_{v_j} \sigma_{ij}) \Phi'(h) \, \omega \, \d v 
- \int_{\R^3} \partial_{v_i} (h \omega^{-1}) \sigma_{ij} \, \partial_{v_j} (\Phi'(h) \omega ) \, \d v \\
&=: T_{11} + T_{12}.
\end{aligned}
$$
Observing that
$$
\begin{aligned}
&\partial_{v_j} (h \omega^{-1})\partial_{v_i} \left(  \Phi'(h) \omega  \right) 
= (p-1)\partial_{v_i} h \partial_{v_j} h \, |h|^{p-2}
+ \frac{1}{p} \partial_{v_i} m  \partial_{v_j} (|h|^p) \, \omega^{ -1} \\
&\qquad\quad
- \frac{p-1}{p} \partial_{v_i} (|h|^p) \partial_{v_j} \omega \, \omega^{-1}
-   \partial_{v_i} \omega \partial_{v_j} \omega \, \omega^{-2} \, |h|^p
\end{aligned}
$$
and using the symmetry of $\sigma_{ij}$, it follows
$$
\begin{aligned}
T_{12} &= -(p-1)\int_{\R^3} \sigma_{ij} \partial_{v_i} h \partial_{v_j} h  \, |h|^{p-2} \, \d v \\
&\quad - \left[   \frac{2}{p}  -1  \right] \int_{\R^3} \sigma_{ij} \partial_{v_i} \omega \partial_{v_j}(|h|^p) \, \omega^{-1}  \, \d v 
  +  \int_{\R^3} \sigma_{ij} \partial_{v_i} \omega \partial_{v_j} \omega \, \omega^{ -2}\, |h|^p \, \d v .
\end{aligned}
$$
Integrating by parts the second term above gives
$$
\begin{aligned}
T_{12} &= -(p-1)\int_{\R^3} (\sigma_{ij} \partial_{v_i} h \partial_{v_j} h)  \, |h|^{p-2} \, \d v 
+\kappa_1(p) \int_{\R^3} \left(\partial_{v_j} \sigma_{ij} \frac{\partial_{v_i} \omega}{\omega}\right)   \, |h|^p \, \d v \\
&\quad +\kappa_1(p) \int_{\R^3} \left( \sigma_{ij} \frac{\partial_{v_i, v_j} \omega}{\omega}\right)  \,  |h|^p \, \d v 
+ \kappa_2(p) \int_{\R^3} \left(\sigma_{ij} \frac{\partial_{v_i} \omega }{\omega} \frac{\partial_{v_j} \omega}{\omega} \right) \,   |h|^p\, \d v ,
\end{aligned}
$$
with
$\kappa_1(p) := \tfrac2p - 1$ and $\kappa_2(p) := 2 \left( 1 - \tfrac1p \right)$. 
On the other hand, for $T_{11}$ we obtain, thanks to an integration by parts,
$$
\begin{aligned}
T_{11} &= -   \int_{\R^3} \partial_{v_i} (h \omega^{-1}) (\partial_{v_j} \sigma_{ij}) \Phi'(h) \, \omega  \, \d v \\
&= -\int_{\R^3} \partial_{v_i} h \Phi'(h) (\partial_{v_j} \sigma_{ij})\, \d v   - \int_{\R^3} \partial_{v_i} (\omega^{-1}) (\partial_{v_j} \sigma_{ij}) h \Phi'(h) \omega \, \d v  \\
&= \frac1p \int_{\R^3} (\partial_{v_i, v_j} \sigma_{ij})    \, |h|^p \, \d v 
+ \int_{\R^3} \left(\partial_{v_j} \sigma_{ij} \frac{\partial_{v_i} \omega}{\omega} \right)  \, |h|^p \, \d v .
\end{aligned}
$$
Gathering previous estimates gives
$$
\begin{aligned}
\int_{\R^3} (\LL g) \Phi'(g) \omega^p \, \d v 
= -(p-1) \int_{\R^3} (\sigma_{ij} \partial_{v_i} h \partial_{v_j} h ) \, |h|^{p-2} \, \d v 
+ \int_{\R^3} \varpi^\mathbf{L}_{\omega,p} \, \omega^p \, |g|^p \, \d v ,
\end{aligned}
$$
where
$$
\begin{aligned}
\varpi^\mathbf{L}_{\omega,p} (v)
&:= \kappa_2(p) \left(\sigma_{ij} \frac{\partial_{v_i} \omega }{\omega} \frac{\partial_{v_j} \omega}{\omega} \right)\\
&\quad
+ \kappa_1(p) \left( \sigma_{ij} \frac{\partial_{v_i, v_j} \omega}{\omega}\right) 
+ (1+ \kappa_1(p))\left(\partial_{v_j} \sigma_{ij} \frac{\partial_{v_i} \omega}{\omega}\right)
- \left( \nu_i \frac{\partial_{v_i} \omega}{\omega} \right) \\
&\quad 
+ \frac1p (\partial_{v_i, v_j} \sigma_{ij}) - \frac1p (\partial_{v_i} \nu_i) + \eta,
\end{aligned}
$$
from which identity \eqref{eq:varpiL_omega_p} follows by observing that $ 4 \partial_{v_i} (h |h|^{p/2-1}) \partial_{v_j} (h |h|^{p/2-1}) = p^2  (\partial_{v_i} h \partial_{v_j} h ) |h|^{p-2} $. 
\end{proof}

\begin{rem}\label{rem:varpi_bar_omegaTER}
For latter references, we observe that 
$$
\varpi^\mathbf{L}_{\omega,p} = \varpi^{\mathbf{L}^*}_{m,q}
$$
when $\nu_i = 0$ in the definition of $\mathbf{L}$,  $1/q+1/p=1$ and $m = \omega^{-1}$. 
\end{rem}

%%%%%%%%%%%%%%%%%%%%%%%%%%%%%%%%%%%%%%%%%%%
\subsection{Trace results for Kolmogorov type equations in a $L^2$ framework}
\label{subsec:trace}

We consider a general Kolmogorov type equation
\beqn\label{eq:KolmogorovTrace}
\partial_t g + v \cdot \nabla_x g =  \mathbf{L}_0 g +  G \quad\text{in } (0,T) \times \OO, 
\eeqn
for $T >0$, where
\begin{equation}\label{eq:def:L0}
\mathbf{L}_0 g := \partial_{v_i} ( \sigma_{ij} \partial_{v_j}g) + \nu_i \partial_{v_i} g,  
\end{equation}
for a positive symmetric matrix $\sigma = \sigma(t,x,v)$, a  vector field $\nu=\nu(t,x,{v})$, a source term $G = G(t,x,{v})$ and we assume
\beqn\label{eq:Kolmogorov-hyp}
\sigma_{ij} , \,  \nu_i  \in L^\infty_{tx}L^\infty_{{\rm loc},v}. 
%\quad    \nu_i    \in L^\infty_{tx}L^\infty_{{\rm loc},v}.
%? ( \Div_v \nu     \in L^\infty_{tx}L^\infty_{{\rm loc},v}??) \Black.
\eeqn

We adapt some trace results for solutions to  the Vlasov-Fokker-Planck equation developed in  \cite[Section~4.1]{MR2721875}, see also \cite[Theorem~11.1]{sanchez:hal-04093201}, and which are mainly a consequence of the two following facts:
\begin{itemize}
\item  If $g \in L^2_{tx}H^1_{v}$ is a weak solution to the Kolmogorov equation \eqref{eq:KolmogorovTrace},  then it is a renormalized solution; 
 
\item If $g \in L^\infty_{txv}$ with $\nabla_v g \in L^2_{txv}$ is a weak solution to the Kolmogorov equation \eqref{eq:KolmogorovTrace}, then it admits a trace $\gamma g \in L^\infty$  in a renormalized sense.
\end{itemize}

We introduce some notations. We denote 
\beqn\label{eq:defdxi1&2}
 \d\xi^1 :=  |n_x \cdot v| \, \dv \, \d\sigma_{\!x} \ \hbox{ and } \
 \d\xi^2 := (n_x \cdot   v)^2 \, \dv \, \d\sigma_{\!x} 
\eeqn
the measures on the boundary set $\Sigma$. 
We denote by $\mathfrak B_1$  the class of renormalizing functions $\beta \in \Wloc^{2,\infty}(\R)$ such that $\beta',\beta'' \in L^\infty(\R)$;
and by $\mathfrak B_2$  the class of renormalizing functions $\beta \in \Wloc^{2,\infty}(\R)$ such that $\beta'' \in L^\infty(\R)$.
We  define the operators
\bear\label{def:bfM0}
 \mathbf{M}_0 g &:=&  \partial_t g + v \cdot \nabla_x g- \mathbf{L}_0 g,
\\ \label{def:bfM0*}
\mathbf{M}_0^* \varphi &:=& - \partial_t \varphi - v \cdot \nabla_x \varphi - \mathbf{L}^*_0 \varphi, 
\eear
where
$$
\mathbf{L}^*_0 \varphi:=  \partial_{v_j} (\sigma_{ij} \partial_{v_i} \varphi) -  \partial_{v_i} ( \nu_i  \varphi).
$$
is the formal adjoint of $\mathbf{L}_0$.
For a $\sigma$-finite and $\sigma$-compact Borel measure space $E = (E,\EEE,d\mu)$, we  write $g \in L(E)$ if $g : E \to \R$ is a Borel function 
%$\beta (g) \in \Lloc^1(E)$ for any $\beta \in W^{2,\infty}(\R)$  
and $g \in C([0,T];L(E))$ if $\beta (g) \in C([0,T];\Lloc^1(E))$ for any $\beta \in W^{2,\infty}(\R)$. We recall that for $T >0$ we denote $\UU = (0,T) \times \OO$, $\Gamma = (0,T) \times \Sigma$ and $\Gamma_{\pm} = (0,T) \times \Sigma_{\pm}$.

\begin{theo}\label{theo-Kolmogorov-trace}
Let $T >0$.
We consider $g \in  L^2((0,T) \times \Omega;\Hloc^1(\R^d))  $, $G \in L^{2}_{tx}H^{-1}_{{\rm loc},v}+ \Lloc^1(\bar\UU)$, $\sigma_{ij}$, $\nu_i$  satisfying \eqref{eq:Kolmogorov-hyp} and we assume that
$g$ is a solution to the  Kolmogorov equation \eqref{eq:KolmogorovTrace} in the distributional sense.  

\smallskip
(1)  Then there exists  $\gamma g \in L(\Gamma)$  and $t \mapsto   g_t \in C([0,T];L(\OO))$ such that $g(t,\cdot) = g_t$ a.e.\ on $(0,T)$ and the following  Green  renormalized formula
\bear\label{eq:FPK-traceL2}
&& \int_{\UU} \left( \beta(g) \, \mathbf{M}^*_0 \varphi+ \beta''(g) \, \sigma_{ij} \partial_{v_i} g \partial_{v_j} g     \varphi \right)  \dv \, \dx \, \dt
 \\ \nonumber
&&\qquad
+  \int_{\Gamma} \beta(\gamma \, g) \, \varphi \,  (n_x \cdot {v}) \, \dv \, \d\sigma_{\! x} \, \dt 
+ \left[ \int_\OO \beta(g_t)  \varphi(t,\cdot) \, \dx \, \dv \right]_0^T =  \langle G  , \beta'(g)   \varphi \rangle
\eear
holds for any renormalizing function $\beta \in \mathfrak{B}_1$ and any test function  $\varphi \in \DD(\bar \UU)$.
%,  as well as  for any  renormalized function  $\beta \in \BB_2$ and any test functions $\varphi \in \DD_0(\bar \OO)$.   
It is worth emphasizing that  $\beta'(g) \varphi \in L^2_{tx}H^1_v \cap L^\infty_{txv}$ with compact support in $\bar\UU$ so that the duality product 
 $\langle G  , \beta'(g)   \varphi \rangle$ is well defined.  We will often write indifferently $g(t,\cdot) = g_t$. 

\smallskip
(2)  If furthermore $G \in L^{2}_{tx}H^{-1}_{{\rm loc},v}$, then $\gamma g \in \Lloc^2(\Gamma, \d\xi^2 \dt)$ and $g \in C([0,T];\Lloc^2(\OO))$. 

\smallskip
(3)  Alternatively to point (2), if furthermore $g_0 \in \Lloc^2(\bar\OO)$, $\gamma_- g  \in \Lloc^2( \Gamma; \d\xi \dt)$ and $G \in L^{2}_{tx}H^{-1}_{{\rm loc},v}$,  
%then $t \mapsto \gamma_t g \in C([0,T];\Lloc^2(\bar\OO))$,  
then $\gamma_+ g  \in \Lloc^2( \Gamma; \d\xi \dt)$, $ g \in C([0,T];\Lloc^2(\bar\OO))$ and \eqref{eq:FPK-traceL2}
holds  for any renormalizing function $\beta \in {\mathfrak B}_2$.  

\smallskip
(4) Alternatively to points (2) and (3), if furthermore $g \in \Lloc^\infty(\bar\UU)$ then $\gamma g \in \Lloc^\infty(\Gamma)$ and \eqref{eq:FPK-traceL2}
holds  for any renormalizing function $\beta \in \mathfrak{B}_2$.

\end{theo}

\begin{proof}[Proof of Theorem~\ref{theo-Kolmogorov-trace}] 
On the one hand, using standard regularization by convolution technique, for a  sequence of mollifiers $(\rho_\eps)$ in $\DD(\R^{2d})$, the function $g_\eps := \rho_\eps *_{x,v} g$ satisfies  
 $$
% \partial_t g_\eps  +   {v} \cdot \nabla_x  g_\eps -  
% \partial_{v_i} ( \sigma_{ij} \partial_{v_j} g_\eps) -  \partial_{v_i}(\nu_i   g_\eps)  = G_\eps
 \partial_t g_\eps  +   {v} \cdot \nabla_x  g_\eps -  
 \partial_{v_i} ( \sigma_{ij} \partial_{v_j} g_\eps) - \nu_i   \partial_{v_i} g_\eps   = G_\eps
$$
in the sense of $\DD'((0,T) \times \OO)$, with $G_\eps \to G$ in $L^{2}_{{\rm loc},tx}H^{-1}_{{\rm loc},v} + \Lloc^1(\bar\UU)$. More precisely, writing the source term as $G := G_0 + \partial_{v_i} G_i$,  with $G_0 \in \Lloc^1(\bar\UU)$ and $G_i \in \Lloc^2(\bar\UU)$ for any $i=1,2,3$, we have 
$G_\eps = G_{0\eps} + \partial_{v_i} G_{i\eps}$ with 
\beqn\label{eq:defG0eps}
%G_{0\eps} :=  G_0 * \rho_\eps + [v \cdot \nabla_x , *\rho_\eps] g     \to G_0 \ \hbox{ in }   \Lloc^2
G_{0\eps} :=  G_0 * \rho_\eps + [v \cdot \nabla_x , \rho_\eps *] g - [\nu_i, \rho_\eps *] \partial_{v_i} g   \to G_0 \ \hbox{ in }   \Lloc^1
\eeqn
and
$$
%G_{i\eps} :=  G_i * \rho_\eps   - [\sigma_{ij},*\rho_\eps] \partial_{v_j} g - [\nu_i,*\rho_\eps]   g \to G_i   \ \hbox{ in }  \Lloc^2, 
G_{i\eps} :=  G_i * \rho_\eps   - [\sigma_{ij},\rho_\eps *] \partial_{v_j} g \to G_i   \ \hbox{ in }  \Lloc^2, 
$$
where we use  the usual commutator notation $[A,B] := AB-BA$ and we use \cite[Lemma~II.1]{MR1022305} in order to justify that the second term converges to $0$ in \eqref{eq:defG0eps}. 
Because $g_\eps \in \Wloc^{1,1}(\bar\UU)$, the chain rules applies and gives 
\bean
&& \partial_t \beta(g_\eps)  +   {v} \cdot \nabla_x  \beta(g_\eps)  -  
 \partial_{v_i} ( \sigma_{ij} \partial_{v_j} \beta(g_\eps)) - \nu_i   \partial_{v_i} \beta(g_\eps) - \beta''(g_\eps) \sigma_{ij} \partial_{v_j}  g_\eps  \partial_{v_i} g_\eps
\\
 &&\quad =  G_{0\eps}  \beta'(g_\eps)  + \partial_{v_i} (G_{i\eps} \beta'(g_\eps)) -  \beta''(g_\eps) G_{i\eps} \partial_{v_i}  g_\eps 
\eean
in the sense of $\DD'((0,T) \times \OO)$ for any $\beta \in C^2 \cap W^{2,\infty}$. Because  now $\beta(g_\eps) \to \beta(g)$ in $\Lloc^2(\UU)$, 
$\beta'(g_\eps) \to \beta'(g)$ in $L^{2}_{{\rm loc},tx}H^{1}_{{\rm loc},v}$ and $(\beta''(g_\eps))$ is bounded in $L^\infty(\UU)$, $\beta''(g_\eps) \to \beta''(g)$ in $\Lloc^1(\UU)$, 
we may pass to the limit $\eps \to 0$ and we obtain 
\bean
&& \partial_t \beta(g)  +   {v} \cdot \nabla_x  \beta(g)  -  
 \partial_{v_i} ( \sigma_{ij} \partial_{v_j} \beta(g)) - \nu_i   \partial_{v_i} \beta(g) - \beta''(g) \sigma_{ij} \partial_{v_j}  g  \partial_{v_i} g
\\
 &&\quad =  G_{0}  \beta'(g)  + \partial_{v_i} (G_{i} \beta'(g)) -  \beta''(g) G_{i} \partial_{v_i}  g 
\eean
in the sense of $\DD'((0,T) \times \OO)$ for any $\beta \in C^2 \cap W^{2,\infty}$, and next for any $\beta \in W^{2,\infty}$. We emphasize here that in \eqref{eq:FPK-traceL2}, the meanning of the first term is 
\bean
 \int_{\UU}   \beta(g) \,  \mathbf{M}^*_0 \varphi  \,  \dv \, \dx \, \dt
&=& -  \int_{\UU} \beta (g) ( \partial_t \varphi + v \cdot \nabla_x \varphi + \mathbf{L}^*_0 \varphi)  \,  \dv \, \dx \, \dt, 
\eean
with 
\bean
 \int_{\UU}   \beta(g) \,  \mathbf{L}^*_0 \varphi  \,  \dv \, \dx \, \dt
&=&     \int_{\UU} ( \nu_i  \varphi - \sigma_{ji} \partial_{v_j} \varphi)  \partial_{v_i} \beta (g)  \,  \dv \, \dx \, \dt. 
\eean
Using that $h := \beta(g) \in L^\infty(\UU) \cap  L^{2}_{{\rm loc},tx}H^{1}_{{\rm loc},v}$ and the right-hand side belongs to $\Lloc^1(\UU) + L^{2}_{{\rm loc},tx}H^{-1}_{{\rm loc},v}$, we may straightforwardly adapt the proof of \cite[Theorem~4.2]{MR2721875} and we get that there exists $\gamma h \in L^\infty(\Gamma)$ and, for any $t \in [0,T]$, there exists $h_t \in L^\infty(\OO)$ such that $t \mapsto  h_t \in C([0,T];\Lloc^1(\bar\OO))$. 
Choosing $\beta$ increasing and defining $\gamma g := \beta^{-1} (\gamma h)$, $ g_t := \beta^{-1} ( h_t)$,  we obtain that the Green formula \eqref{eq:FPK-traceL2} holds true for any $\beta \in W^{2,\infty}$. 
The additional regularity and integrability properties on $ g_t$ and $\gamma g$ follow from this  Green formula as in  \cite[Section~4]{MR2721875}. We may thus extends the set of renormalizing functions $\beta \in \mathfrak{B}_i$
with $i=1$ or $i=2$, depending on the regularity assumptions. 
\end{proof}

We will also use the following stability result in the spirit of  \cite[Theorem 5.2]{MR2721875} and  the following duality result in the spirit of \cite[Proposition~3]{MR1765137}.

 \begin{prop}\label{prop:Kolmogorov-stability}   
Let us consider four  sequences $(g^k)$,  $(\sigma^k)$, $(\nu^k)$ and  $(G^k)$  and four functions  $g$, $\sigma$, $\nu$, $G$ which all  satisfy the requirements of Theorem~\ref{theo-Kolmogorov-trace}. 

\smallskip
(1) If $g^k  \wto g$  weakly in $\Lloc^2(\bar\UU)$, $\sigma^k \to \sigma$  strongly  in $\Lloc^2(\bar\OO)$, $\nu^k \to \nu$  strongly  in $\Lloc^2(\bar\OO)$ and  $G^k \wto G$ weakly in $L^{2}_{{\rm loc},x}H^{-1}_{{\rm loc},v}$, then  $g$ satisfies \eqref{eq:KolmogorovTrace} so that it admits a family of trace functions $ \gamma g \in \Lloc^2 (\Gamma; d\xi_2)$,   $ g_t \in \Lloc^2 (\OO)$, for any $t \in [0,T]$, and (up to the extraction of a subsequence) $\gamma g^k \wto  \gamma g$ weakly in $\Lloc^2 (\Gamma; d\xi_2)$, $g_t^k \wto  g_t$ weakly in $\Lloc^2 (\OO)$, for any $t \in [0,T]$.

\smallskip
(2) If  $g^k  \to g$  strongly in $ L^2((0,T) \times \Omega; \Hloc^1(\R^d))$,  $\sigma^k \wto \sigma$ weakly  in $\Lloc^2(\bar\OO)$, $\nu^k \wto \nu$  weakly  in $\Lloc^2(\bar\OO)$ and  $G^k \to G$ weakly in $L^{2}_{{\rm loc},x}H^{-1}_{{\rm loc},v}$, 
then  $g$ satisfies \eqref{eq:KolmogorovTrace} 
so that it admits a family of trace $ \gamma g \in \Lloc^2 (\Gamma; d\xi_2)$,  $g_t \in \Lloc^2 (\OO)$, for any $t \in [0,T]$, 
and (up to the extraction of a subsequence) $\gamma g^k \to  \gamma g$ a.e.\ and weakly in $\Lloc^2 (\Gamma; d\xi_2)$, $g_t^k \to  g_t$ a.e.\ and weakly in $\Lloc^2 (\OO)$, for any $t \in [0,T]$.

\end{prop}

\begin{proof}[Proof of Proposition~\ref{prop:Kolmogorov-stability}] 
We only check the proof of (2). The proof of (1) is identical by choosing $\beta(s) := s$.
%We observe that 
%$$
%\partial_t g^k + v \cdot \nabla_x g^k  = G_{0}^k + \Div_v G_1^k
%$$
%with $(g^k)$, $(\nabla_v g^k)$, $(G_0^k)$  and $(G_1^k)$ bounded in $L^{2}_{{\rm loc}}([0,T] \times \bar\OO)$  and we may use the $H^{1/3}_{t,x,v}(\R^{2d+1})$ regularity result \cite[Theorem~1.3]{MR1949176} on any truncated version of $(g^k)$
%in order to conclude that $(g^k)$ belongs to a compact set of $L^{2}_{{\rm loc}}([0,T] \times \bar\OO)$. 
For $\beta \in \mathfrak{B}_1 \cap C^2$ and $\varphi \in \DD((0,T) \times \bar \OO)$, we write the renormalized Green formula 
\bean
\int_\UU ( \beta(g^k) \mathbf{M}_0^* \varphi -\beta''(g^k)  \sigma^k_{ij} \partial_{v_j}  g^k  \partial_{v_i} g^k \varphi) + \int_{\Gamma} \beta(\gamma g^k) \varphi n_x \cdot v %\boldsymbol{\Gamma}
 =  \int_\UU \widetilde G^k \varphi, 
 \eean
 with $\widetilde G^k := G_{0}^k  \beta'(g^k)  + \partial_{v_i} (G_{i}^k  \beta'(g^k)) -  \beta''(g^k) G_i^k \partial_{v_i}  g^k$. 
 Observing that, up to the extraction of a subsequence, $\beta(g^k) \to \beta(g)$ a.e.\ and $\beta(\gamma g^k) \wto \bar \beta$ weakly in $ \Lloc^2 (\Gamma; \d\xi^2)$, we may pass to the limit in the above equation and we get 
\bean
\int_\UU ( \beta(g) \mathbf{M}_0^* \varphi - \beta''(g) \sigma_{ij} \partial_{v_j}  g  \partial_{v_i} g \varphi) + \int_\Gamma \bar\beta  \varphi (n_x \cdot v)
 =  \int_\UU \widetilde G \varphi, 
 \eean
 with $\widetilde G  := G_{0}   \beta'(g )  + \partial_{v_i} (G_{i}   \beta'(g )) -  \beta''(g ) G_i  \partial_{v_i}  g $. Thanks to Theorem~\ref{theo-Kolmogorov-trace},
  we thus have $\bar\beta = \beta(\gamma g)$ a.e.\ on $\Gamma$. Defining $\beta_2(s) := \beta(s)^2$ for $\beta \in W^{2,\infty} \cap C^2$  and observing that $\beta_2 \in W^{2,\infty} \cap C^2$, the above argument for both $\beta$ and $\beta_2$ implies $\beta(\gamma g^k) \wto \beta(\gamma g)$  and $\beta(\gamma g^k)^2 \wto \beta(\gamma g)^2$  both weakly in $ \Lloc^2 (\Gamma; \d\xi^2)$. We classically deduce $\beta(\gamma g^k) \to \beta(\gamma g)$  strongly  in $ \Lloc^2 (\Gamma; \d\xi^2)$, and thus, up to the extraction of a subsequence, $\gamma g^k \to  \gamma g$ a.e.\ by choosing $\beta$ one-to-one.
The proof of the result concerning the trace functions $g^k_t$ and $g_t$ on the sections $\{t \} \times \OO$ can be handled in a similar way and it is thus skipped.   
\end{proof}

\begin{prop}\label{prop:Kolmogorov-duality}
Let $T >0$.
Consider two solutions $f,h \in L^2((0,T) \times \Omega; \Hloc^1(\R^d))$
to the primal and the dual Kolmogorov equations 
\bean
\mathbf{M}_0 f =  F, \quad  \mathbf{M}^*_0 h  =  H, 
\eean
%\bean
%&&\mathbf{M}_0 f =  F\partial_t f +  v \cdot \nabla_x f  -   \mathbf{L}_0  f = F,  %\quad   \mathbf{L} f := \partial_{v_i} (\sigma_{ij} \partial_{v_j}f) + \partial_{v_i} ( \nu_i f) 
%\\  
%&&\mathbf{M}^* h :=   \partial_th  - v \cdot \nabla_x h  -  \mathbf{L}_0^*h  =  H, \quad   \mathbf{L}_0^*h :=   \partial_{v_j} (\sigma_{ij} \partial_{v_i}h) -  \nu_i \partial_{v_i} h, %+\eta h, 
%\eean
with $\mathbf{M}_0$ and $\mathbf{M}_0^*$ defined in \eqref{def:bfM0} and \eqref{def:bfM0*},  $F,H \in \Lloc^2(\bar \UU)$ and $\sigma_{ij}$, $\nu_i$  satisfying \eqref{eq:Kolmogorov-hyp}. %, $\sigma_{ji} = \sigma_{ij}$. % and $\eta \in  \Lloc^\infty(\bar \UU)$.
For  any renormalizing functions $\alpha, \beta \in W^{2,\infty}(\R)$ and any test function  $\varphi \in \DD(\bar \UU)$,
 there holds 
\bear\label{eq:cor:Kolmogorov-duality3}
&& \int_{\UU}   \alpha(f) \beta(h)  \,  \mathbf{M}_0^*  \varphi +  \int_{\Gamma} \alpha(\gamma \, f) \beta(\gamma \, h) \, \varphi \,   (n_x \cdot v)  
% \\ \nonumber
%&&\qquad
+ \left[ \int_\OO \alpha(f_t) \beta(h_t)  \varphi(t,\cdot)  \right]_0^T =  \int_\UU G    \varphi  , 
\eear
where $G \in \Lloc^1(\bar \UU)$ is defined by 
$$
G := \alpha'(f) F \beta (h) + \alpha(f) \beta'(h) H - \alpha''(f) \sigma_{ij} \partial_{v_i} f \partial_{v_j} f \beta(h) - 
\alpha(f) \beta''(h) \sigma_{ij} \partial_{v_i} h \partial_{v_j} h. 
$$
\end{prop}

\begin{proof}[Proof of Proposition~\ref{prop:Kolmogorov-duality}]
With the notations of Theorem~\ref{theo-Kolmogorov-trace}, the functions $f_\eps := f*_{x,v} \rho_\eps$ and $h_\eps := h*_{x,v} \rho_\eps$ satisfy
\bean
&&\partial_t f_\eps = - v \cdot \nabla_x f_\eps  +   \mathbf{L}_0  f_\eps  + F_\eps
\\
&&- \partial_th_\eps =   v \cdot \nabla_x h_\eps +   \mathbf{L}_0^* h_\eps  + H_\eps,
\eean
with $f_\eps \to f$,  $h_\eps \to h$ in $L^2((0,T) \times \Omega; \Hloc^{1}(\R^d))$
and $F_\eps \to F$,  $H_\eps \to H$ in $\Lloc^2(\bar\UU)$. From Proposition~\ref{prop:Kolmogorov-stability}--(2), we get $\gamma f_\eps \to \gamma f$, $\gamma h_\eps \to \gamma h$ a.e.\ on $\Gamma$ and  $f_{\eps t} \to f_t $, $h_{\eps t} \to h_t $ a.e.\ on $\OO$ for any $t \in [0,T]$. 

\smallskip
For $\alpha,\beta \in W^{3,\infty}(\R)$, we thus deduce that $\alpha(\gamma f_\eps)\beta(\gamma h_\eps) \to \alpha(\gamma f) \beta(\gamma h)$ in $\Lloc^1(\Gamma, d\xi^1dt)$ and  
$\alpha(f_{\eps t})\beta( h_{\eps t}) \to \alpha(f_t) \beta(h_t)$ in $\Lloc^1(\bar\OO)$ for any $t \in [0,T]$. 

\smallskip
On the other hand, we set $g_\eps  := \alpha (f_\eps) \beta(h_\eps)$ which satisfies 
$$
\partial_t g_\eps  + v \cdot \nabla_x g_\eps =  \mathbf{L}_0 g_\eps    + G_\eps, 
$$
with $G_\eps$ defined similarly as for $G$. Because $g_\eps \to g := \alpha (f) \beta(h)$ in $L^2((0,T) \times \Omega;H^1_{v})$, $G_\eps \to G$ in $\Lloc^1(\bar\UU)$,
we may use  Proposition~\ref{prop:Kolmogorov-stability}--(2), and we deduce that $\gamma g_\eps \to \gamma g$. Because $\gamma g_\eps = \alpha (\gamma f_\eps) \beta(\gamma h_\eps)$ and using the previous convergence, we deduce that $\gamma g = \alpha(\gamma f) \beta(\gamma  h)$.  We similarly prove $g_t = \alpha(f_t) \beta(h_t)$ for any $t \in [0,T]$. The identity \eqref{eq:cor:Kolmogorov-duality3} is thus noting but the non-renormalized Green formula
\eqref{eq:FPK-traceL2} applied to $g$.  
\end{proof}

%%%%%%%%%%%%%%%%%%%%%%%%%%%%%%%%%%%%%%%%%%%
\subsection{Well-posedness for Kolmogorov type equations}
\label{subsec:Well-posedness}

We consider the Kolmogorov type equation, for $T >0$, 
\beqn\label{eq:KolmogorovWP}
\partial_t f + v \cdot \nabla_x f =  \mathbf{L} f + \KKK[f] \quad \hbox{in} \ (0,T) \times \OO,  
\eeqn
with a general parabolic operator in the velocity variable 
\begin{equation}\label{eq:OperatorKolmorovL}
\mathbf{L} f = \sigma_{ij} \partial_{v_i, v_j}f + \nu_i \partial_{v_i} f + \eta f , 
\end{equation}
and an abstract (integral in the velocity variable) operator $\KKK$, 
which is complemented with the Maxwell reflection boundary condition \eqref{eq:reflect_F} and an initial datum $f(0) = f_0$ in $\OO$. 
We make the same assumptions \eqref{eq:Kolmogorov-hyp}  on the coefficients $\sigma$, $\nu$ and we also assume 
\beqn
\eta \in L^\infty_{t,x}L^\infty_{{\rm loc},v}
\quad\hbox{and}\quad 
\sigma_{ij} \zeta_i \zeta_j \ge \sigma_0 |\zeta|^2, \ \forall \, \zeta \in \R^3, 
\eeqn
for some $\sigma_0 > 0$.
We next assume that the problem behave adequately in a weighted $L^2$ framework. More precisely,   for some possible perturbation $\widetilde\omega = \theta \omega$  
of a weight function $\omega : \R^3 \to (0,\infty)$,  we assume 
$$
0 < \theta_0 \le \theta \le \theta_1 < \infty, \quad |\nabla_x \theta| + |\nabla_v \theta| \lesssim \theta  \langle v \rangle^{-1}, 
$$
 the function $ \varpi_{\tilde\omega,2}^\mathbf{L}$ defined by \eqref{eq:varpiL_omega_p} satisfies 
\beqn\label{eq:varpiL_omega_2bdd}
- \lambda_1  \varsigma \le \varpi_{\tilde\omega,2}^\mathbf{L}  \le   \lambda_0 -  \varsigma, 
\eeqn
for a function $ \varsigma : \R^3 \to \R_+$ and some constants $\lambda_i >  0$. 
We also assume that the nonlocal operator $\KKK$ is bounded in $L^2_\omega(\R^3)$ and more precisely satisfies 
\beqn\label{eq:KolmogorovWP-hypKKK}
 \sup_{(0,T) \times \Omega} \| \KKK \|_{\BBB(L^2_v(\omega))}  = C_{\KKK,2} < \infty, 
\eeqn
and the reflection operator $\RRR$ satisfies 
\beqn\label{eq:KolmogorovWP-hyp2}
\RRR : L^2(\Sigma_+;  \d\xi^1_{\tilde\omega} ) \to L^2(\Sigma_-;  \d\xi^1_{\tilde\omega}), \quad   \| \RRR \|_{L^2(\Sigma;  \d\xi^1_{{\tilde\omega}} )} \le 1, 
\eeqn
 where we denote here and below
 \beqn\label{eq:defdxi1&2omega}
 \d\xi^1_{\varrho} := \varrho |n_x \cdot v| \, \d v \, \d\sigma_{\!x} \ \hbox{ and } \
 \d\xi^2_\omega := \omega^2 \la v \ra^{-2} (n_x \cdot v)^2 \, \d v \, \d \sigma_{\!x}, 
\eeqn 
with $\varrho :=  \omega$ or $\varrho := \tilde\omega$.
For further references, we define 
\begin{equation}\label{eq:def:H1dagger}
 \| g \|^2_{H^{1,\dagger}_{\tilde\omega}(\UU)}  :=  \int_{\UU} \left\{  \sigma_{ij} \partial_{v_i} (\widetilde\omega g) \partial_{v_j} (\widetilde\omega g)  +  \varsigma  \widetilde\omega^2 g^2 \right\} .\end{equation}
We next assume that the problem behave nicely in a $L^1$ framework, namely 
\beqn\label{eq:KolmogorovWP-hypKL1}
 \sup_{(0,T) \times \Omega} \| \KKK \|_{\BBB(L^1_v(\R^3))} = C_{\KKK,1} , 
 \qquad 
 \varpi^\mathbf{L}_{1,1} \le \lambda_2, 
\eeqn
for some constants $C_{\KKK,1},  \lambda_2 \in [0,\infty)$, 
and  we recall that from the very definition \eqref{eq:reflection} (see also \eqref{eq:invariantsBoundary1}), we have 
\beqn\label{eq:KolmogorovWP-hypRL1}
\RRR : L^1(\Sigma_+; \d\xi^1) \to L^1(\Sigma_-; \d\xi^1), \quad   \| \RRR \|_{L^1(\Sigma;  \d \xi_1 )} \le 1.
\eeqn
We finally make a compatibility hypothesis on the two weighted $L^2$ and $L^1$ frameworks by assuming
\beqn\label{eq:sigma_nu_eta_KKK_omega}
\langle v \rangle  \omega^{-1} \in L^2(\R^2),  
\quad
(|\sigma_{ij}| + |\partial_{v_j} \sigma_{ij}| +  |\partial^2_{v_iv_j} \sigma_{ij}|  + |\nu_i| + |  \partial_{v_i} \nu_i| + |\eta|) \omega^{-1} \in  L^2(\UU). 
\eeqn
For further reference, we define the Hilbert space $\HHH$ associated to 
the Hilbert norm $\| \cdot \|_\HHH$ defined by 
$$
\| f \|^2_\HHH := \| f \|^2_{L^2_\omega} +  \| f \|^2_{H^{1,\dagger}_\omega}
$$
with $\| \cdot \|_{H^{1,\dagger}_\omega}$ being defined in \eqref{eq:def:H1dagger}.

\begin{theo}\label{theo-Kolmogorov-WellP}
Let $T >0$.
Under the above conditions, for any $f_0 \in L^2_\omega(\OO)$, there exists a  unique  weak solution $f \in C([0,T];L^2_\omega) \cap  \HHH$ %L^2((0,T) \times \Omega;\Hloc^1(\R^d))$ 
to the Kolmogorov equation \eqref{eq:KolmogorovWP} complemented with the Maxwell reflection boundary condition \eqref{eq:reflect_F} and associated to the initial datum $f_0$. 
More precisely, the function $f$ satisfies  equation \eqref{eq:KolmogorovWP} in the sense of distributions in $\DD'(\UU)$ with trace functions, defined thanks to Theorem~\ref{theo-Kolmogorov-trace}, satisfying $\gamma f \in L^2(\Gamma,d\xi^2_\omega)$ as well as  the Maxwell reflection boundary condition \eqref{eq:reflect_F} pointwisely and $f(t,\cdot) \in L^2_\omega$, $\forall \, t \in [0,T]$, as well as the  initial condition $f(0,\cdot) = f_0$ pointwisely. 
\end{theo}

The proof follows similar lines as in \cite{MR0872231}  (see also \cite{MR2072842}, \cite[Sec. 8 \& Sec. 11]{sanchez:hal-04093201} and \cite{CGMM**}) and it is thus only sketched. 

\begin{proof}[Proof of Theorem~\ref{theo-Kolmogorov-WellP}]
We split the proof into four steps. 

\medskip\noindent
\textit{Step 1.} Given $\mathfrak f \in L^2(\Gamma_-; \d\xi^1_\omega)$, we solve the inflow problem
\begin{equation}\label{eq:linear_gGinflow}
\left\{
\begin{aligned}
& \partial_t f + v \cdot \nabla_x f =\mathbf{L} f  \quad &\text{in}&\quad \UU \\
& \gamma_{-} f  = \mathfrak f \quad &\text{on}&\quad \Gamma_{-} \\
& f_{| t=0} = f_0  \quad &\text{in}&\quad   \OO,
\end{aligned}
\right.
\end{equation}
thanks to Lions' variant of the Lax-Milgram theorem \cite[Chap~III,  \textsection 1]{MR0153974}. More precisely,  we define  the bilinear form
$\EEE : \HHH \times C_c^1(\UU \cup \Gamma_-) \to \R$,  by 
\bean
\EEE(f,\varphi) 
&=& \int_\UU f (\lambda+\partial_t + v \cdot \nabla_x -\mathbf{L})^*  ( \varphi \widetilde\omega^2)
\\
&:=& \int_\UU (\lambda f - \mathbf{L} f) \varphi \widetilde\omega^2   -   \int_\UU f ( \partial_t \varphi + v \cdot \nabla_x  \varphi ) \widetilde\omega^2 . 
\eean
We observe that this one is coercive, namely thanks to Lemma~\ref{lem:dissipativity_Lp} and \eqref{eq:varpiL_omega_2bdd} there holds
\bean
\EEE(\varphi,\varphi) 
&=& \int_\UU (\lambda \varphi - \mathbf{L} \varphi) \varphi \widetilde\omega^2  +  \frac12 \int_\OO \varphi(0,\cdot)^2 \widetilde\omega^2  + \frac12 \int_{\Gamma_-} (\gamma_- \varphi)^2 \, \d\xi^1_{\tilde\omega}
\\
&\ge& (\lambda-\lambda_0) \| \varphi \|^2_{L^2_{\tilde\omega}} +  \| g \|^2_{H^{1,\dagger}_{\tilde\omega}} 
% \| \varphi \|^2_{\dot H^1_{\sqrt{\sigma}\omega}} %+ \lambda_1 \| \varphi \|^2_{L^2_{\sqrt{ \varsigma}\omega}}   
+ \tfrac12 \| \varphi(0) \|^2_{L^2_{\tilde\omega}} + \tfrac12 \| \varphi \|^2_{L^2(\Gamma_-; \d\xi^1_{\tilde\omega})}, 
\eean
for any $\varphi \in C_c^1(\UU \cup \Gamma_-)$. Taking $\lambda>\lambda_0$, %with $\kappa := \min (\lambda-C,1)>0$, for $\lambda > 0$ large enough. 
the above mentioned Lions' theorem implies the existence of a function $f_\lambda \in \HHH$
which satisfies the variational equation 
$$
\EEE(f_\lambda,\varphi) = \int_{\Gamma_-} \mathfrak{f} e^{-\lambda t} \varphi \widetilde\omega^2 \, \d\xi^1  + \int_\OO f_0 \varphi(0,\cdot) \widetilde\omega^2  , \quad \forall \, \varphi \in C_c^1(\UU \cup \Gamma_-). 
$$
Defining $f := f_\lambda e^{\lambda t}$ and using Theorem~\ref{theo-Kolmogorov-trace}, we deduce that $f \in \HHH \cap C([0,T];L^2_\omega(\OO))$ is a renormalized solution to the inflow problem \eqref{eq:linear_gGinflow} and that $\gamma f \in L^2(\Gamma; \d\xi^1_{\tilde\omega})$. From the renormalization formulation, we have the uniqueness of such a solution (see also Step~4 below). 
Directly from \eqref{eq:KolmogorovWP-hypKKK}, %\eqref{eq:KolmogorovWP-hyp1}, 
we also deduce the energy estimate 
\bean
&&\| f_t \|^2_{L^2_{\tilde\omega}} + \int_0^t \left( \| \gamma f_s \|^2_{L^2(\Gamma_-; \d\xi^1_{\tilde\omega})} +  2  \| f \|^2_{H^{1,\dagger}_{\tilde\omega}} \right) \, e^{\lambda_0(t-s)} \, \d s 
\\
&&\qquad\le
 \| f_0 \|^2_{L^2_{\tilde\omega}} e^{\lambda_0 t} + \int_0^t  \| \mathfrak{f}_s \|^2_{L^2(\Gamma_-; \d\xi^1_{\tilde\omega})}   \, e^{\lambda_0(t-s)} \, \d s. 
\eean

\medskip\noindent
\textit{Step 2.} For any $\alpha \in (0,1)$ and $ h \in \HHH \cap C([0,T];L^2_\omega(\OO))$ such that $\gamma h \in L^2(\Gamma; \d\xi^1_{\omega})$, % {\Red Est-ce bien def??}, 
we then consider the modified Maxwell reflection boundary condition problems
$$%\begin{equation}\label{eq:linear_gGk}
\left\{
\begin{aligned}
& \partial_t f + v \cdot \nabla_x f = \mathbf{L}f \quad &\text{in}&\quad \UU \\
& \gamma_{-} f   =  \alpha \RRR \gamma_{+}  h  \quad &\text{on}&\quad \Gamma_{-} \\
& f(t=0, \cdot) = f_0  \quad &\text{in}&\quad   \OO, 
\end{aligned}
\right.
$$%\end{equation}
for which a solution $f \in \HHH \cap C([0,T];L^2_\omega(\OO))$ such that $\gamma f \in L^2(\Gamma; \d\xi^1_{\omega})$ is given by the first step. Thanks to the energy estimate stated in the first step, we immediately see that 
the mapping $h \mapsto f$ is $\alpha^{1/2}$-Lipschitz for the norm defined by 
$$
%\| f \|^2 := 
\sup_{t \in [0,T]} \left\{ \| f_t \|^2_{L^2_{\tilde\omega}}  e^{- \lambda_0 t }+ \int_0^t  \| \gamma f_s \|^2_{L^2(\Gamma_-; \d\xi^1_{\tilde\omega})}  \, e^{- \lambda_0s} \, \d s \right\}.
$$
From the Banach fixed point theorem, we deduce the existence of a unique fixed point to this mapping. 

\medskip\noindent
\textit{Step 3.} For a sequence $\alpha_k \in (0,1)$, $\alpha_k \nearrow 1$, we next consider the sequence $(f_k)$ obtained in Step~2 as the solution to the modified Maxwell reflection boundary condition problem
\begin{equation}\label{eq:linear_gak}
\left\{
\begin{aligned}
& \partial_t f_k + v \cdot \nabla_x f_k = \mathbf{L}f_k \quad &\text{in}&\quad (0,T) \times \OO \\
& \gamma_{-} f_k   =  \alpha_k \RRR \gamma_{+} f_k  \quad &\text{on}&\quad (0,T) \times \Sigma_{-} \\
& f_k(t=0, \cdot) = f_0  \quad &\text{in}&\quad   \OO, 
\end{aligned}
\right.
\end{equation}
which,  from the energy estimate stated at the end of Step~1, satisfies 
\bear\label{eq:linear_gak_bdd}
&&\quad \| f_{kt} \|^2_{L^2_{\tilde\omega}} + \int_0^t \left\{ (1-\alpha_k) \| \gamma f_{ks} \|^2_{L^2(\Gamma_-; \d\xi^1_{\tilde\omega})}
+  2
 \| f_{ks} \|^2_{H^{1,\dagger}_{\tilde\omega}} 
\right\} \, e^{\lambda_0(t-s)} \, \d s 
\le
 \| f_0 \|^2_{L^2_{\tilde\omega}} e^{\lambda_0 t}  , 
\eear
for any $t \in (0,T)$ and any $k \ge 1$. Choosing $\beta(s) := s^2$ and $\varphi := (n_x \cdot v) \la v \ra^{-2} \omega^2(v)$ in the Green formula \eqref{eq:FPK-traceL2}, we additionally have 
$$
\int_{\Gamma} (\gamma f_k)^2 \, \d\xi^2_\omega \, \d t %\hat \omega^2 (n_x \cdot v)^2 dvd\sigma_{\!x}dt 
\lesssim   \| f_0 \|^2_{L^2_\omega} e^{\lambda_0 T}. %\Red \| g_k \|_{\HHH}^2.
$$
%for any $0 < \tau < T/2$, $\Gamma_\tau := (\tau,T-\tau) \times \Sigma$. 
From the above estimate we deduce that, up to the extraction of a subsequence, there exist $f \in \HHH \cap L^\infty(0,T;L^2_\omega(\OO))$ and $\mathfrak{f}_\pm \in L^2(\Gamma_{\pm}; \d\xi^2_{\omega} \d t)$ such that 
$$
f_k \wto f \hbox{ weakly in } \ \HHH \cap L^\infty(0,T;L^2_\omega(\OO)), 
\quad
\gamma_\pm f_k \wto \mathfrak{f}_\pm \hbox{ weakly in } \ L^2(\Gamma; \d\xi^2_{\omega} \d t). 
$$
 From the condition \eqref{eq:sigma_nu_eta_KKK_omega}, %$\langle v \rangle  \omega^{-1} \in L^2(\R^3)$}, 
 we have $L^2(\Gamma; \d\xi^2_{\omega}) \subset L^1(\Gamma; \d\xi^1)$. 
Together with the assumption~\eqref{eq:KolmogorovWP-hypRL1}, 
we deduce that $\RRR(\gamma f_{k+})  \wto \RRR(\mathfrak{f}_+)$   weakly in $L^1(\Gamma_-; \d\xi^1)$. 
On the other hand, from Proposition~\ref{prop:Kolmogorov-stability}--(1), we have $\gamma f_k \wto \gamma f$ weakly in $\Lloc^2 (\Gamma; \d\xi^2_\omega)$. %so that ${\mathfrak g} = \gamma g$. 
Using both convergences in the boundary condition $\gamma_- f_k = \RRR(\gamma_+ f_k)$, we obtain $\gamma_- f = \RRR(\gamma_+ f)$.
We may thus pass to the limit in equation \eqref{eq:linear_gak} and we obtain that $f \in C([0,T]; L^2_\omega) \cap \HHH$ is a renormalized solution to the Kolmogorov equation \eqref{eq:KolmogorovWP} complemented with the Maxwell reflection boundary condition \eqref{eq:reflect_F} and associated to the initial datum $f_0$.

\medskip\noindent
\textit{Step 4.} We consider now two solutions $f_1$ and $f_2 \in  C([0,T];L^2_\omega) \cap \HHH$ to  the Kolmogorov equation  \eqref{eq:KolmogorovWP}--\eqref{eq:reflect_F} and associated to the same initial datum $f_0$, so that the function $f := f_2 - f_1 \in  C([0,T]; L^2_\omega) \cap \HHH$  is a solution to the Kolmogorov equation  \eqref{eq:KolmogorovWP}-\eqref{eq:reflect_F} associated to the initial datum $f(0) = 0$.
Choosing $\varphi := \chi_R$,  with $\chi_R(v) := \chi(v/R)$, $\mathbf{1}_{B_1} \le \chi \in \DD(\R^3)$,  and $\beta \in C^2(\R)$, $\beta''  $ with compact support,  in \eqref{eq:FPK-traceL2},
%nd repeating the proof of  Lemma~\ref{lem:dissipativity_Lp}, 
we have 
\bean
&&  \int_\OO \beta(f_T)  \chi_R   +  \int_{\Gamma} \beta(\gamma f) \,  \chi_R \,  (n_x \cdot v)  +  \int_{\UU}  \beta''(f) \, \sigma_{ij} \partial_{v_i} f \partial_{v_j} f   
\\
 &&\qquad =  \int_{\UU}  \left\{  \beta(f) \left(   \partial^2_{ij} (\sigma_{ij}   \chi_R)  - \partial_{v_i} (\nu_i \chi_R) \right)  + (\eta f + \KKK[f]) \beta'(f) \chi_R \right\} . 
\eean
%
%\bean
%&& \int_{\UU} \left( \beta(g) \, \mathbf{L}^* \omega_1  + \beta''(g) \, \sigma_{ij} \partial_{v_i} g \partial_{v_j} g     \omega_1 \right) \, d{v} dx dt
% \\ \nonumber
%&&\qquad
%+  \int_{\Gamma} \beta(\gamma \, g) \,   n_x \cdot {v} \,  d{v} d\sigma_{\! x}dt 
%+ \int_\OO \beta(g_T)  \omega_1 dxdv   = 0. 
%\eean
We assume $0 \le \beta (s) \le |s|$, $|\beta'(s)| \le 1$, % $0 \le s \beta'(s) \le |s|$ 
and $\beta'' \ge 0$ so that we may get rid of the last term at the left-hand side of the above identity, and we use the bound \eqref{eq:sigma_nu_eta_KKK_omega} in order pass to the limit $R \to \infty$. 
We obtain  
\bean
&& \int_\OO \beta(f_T)  + \int_{\Gamma} \beta(\gamma f) \,       (n_x \cdot {v})  
%\\
%&&\quad 
\le 
%\int_{\Gamma_+} |\RRR (\gamma_+ g)| \,    |n_x \cdot {v}| 
%+ \lambda_0  \int_{\UU}  |g| . 
%\\&&
\int_{\UU}  \left\{\beta(f) \left(   \partial^2_{ij} \sigma_{ij}     - \partial_{v_i} \nu_i  \right)  + \eta f \beta'(f) + |\KKK[f]| \right\}  .
\eean
Passing to the limit $\beta(s) \nearrow |s|$ such that $0 \le s \beta'(s) \nearrow |s|$, we deduce 
\bean
&& \int_\OO |f_T|  + \int_{\Gamma} |\gamma f| \, (n_x \cdot v) 
%\\
%&&\quad 
\le 
%\int_{\Gamma_+} |\RRR (\gamma_+ g)| \,    |n_x \cdot {v}| \,  d{v} d\sigma_{\! x}dt 
%+ \lambda_0  \int_{\UU}  |g| \,     d{v} dx dt. 
%\\&&
\int_{\UU}    |f|  \left(  \partial^2_{ij} \sigma_{ij}     - \partial_{v_i} \nu_i   + \eta   + C_{\KKK,1}   \right)  ,
\eean
where we have used $L^2(\Gamma ; \d \xi^2_\omega \d t) \subset L^1(\Gamma ; \d \xi^1 \d t)$ in order to justify the convergence of the integral on the boundary.
Using finally  \eqref{eq:KolmogorovWP-hypKL1} and  \eqref{eq:KolmogorovWP-hypRL1}, we deduce 
$$
\int_\OO |f_T|  
\le (C_{\KKK,1}+ \lambda_2)  \int_0^T \!\! \int_\OO   |f| , 
$$
and we conclude to $f = 0$ thanks to Gr\"onwall's lemma. 
\end{proof}

\subsection{Decay estimates in a weakly dissipative framework}
\label{subsec:weakDissipFramework}

In this section, we formulate some elementary decay estimates which are essentially picked up from \cite[Lemma~3.1]{MR3625186} and which will be useful for handling the weakly dissipative framework corresponding to the case $s + \gamma < 0$, which always holds when $\gamma \in [-3,-2)$.  We also refer to \cite{Caflisch1,Caflisch2,MR1751701,RockWang,GuoLandau1,MR2460938,MR2499863,GS1} for previous works dealing with such a situation and to the recent papers \cite{MR4265692,MR3625186,MR4534707,sanchez:hal-04093201} for more discussion and more  references. 
We start with a variant of the Gr\"onwall lemma. 

\begin{lem}\label{lem:Gronwall} Let us consider three continuous functions $u$, $v$ and $w : \R_+ \to \R_+$  satisfying  $u \le v \le w$ and the three following properties 
\bean
&&v_t ' +  \sigma  u_t   \le 0,  
\quad w_{t}   \le C  w(0), 
\quad \eps_R v_t \le u_t + \vartheta_R w_t,   
\eean
for some constants $C,\sigma > 0$ and for any $t > 0$, where $R \mapsto \eps_R,\vartheta_R$ are  two positive functions  such that $\eps_R\to0$ and $\vartheta_R/\eps_R \to 0$ as $R \to \infty$. Then 
$$
v_t \le \Gamma_t w_0, \quad \forall \, t \ge 0, % \quad \Gamma_t :=   \inf_{R>0} \Gamma_t(R), \quad \Gamma_t(R) :=   e^{- \sigma \eps_R t} +\frac{\vartheta_R}{\eps_R} C .
%\left( e^{- \sigma \eps_R t} +\frac{\vartheta_R}{\eps_R} C \right).
$$
with 
\beqn\label{eq:lemGronwall}
\Gamma_t :=   \inf_{R>0} \Gamma_t(R), \quad \Gamma_t(R) :=   e^{- \sigma \eps_R t} +\frac{\vartheta_R}{\eps_R} C .
 \eeqn
\end{lem}

\begin{proof}[Proof of Lemma~\ref{lem:Gronwall}] The three pieces of information together imply
$$
v'_t + \sigma \eps_R v_t \le \sigma  \vartheta_R w_t \le  \vartheta_R \sigma Cw_0
$$
for any $R>0$.
Using the classical Gr\"onwall's lemma, we deduce 
$$
v_t \le e^{- \sigma \eps_R t} v_0 +    \frac{\vartheta_R}{\eps_R} C w_0 (1 - e^{- \sigma \eps_R t} ), 
$$
from which we immediately conclude. 
\end{proof}

We now apply the previous decay estimate in a concrete situation we will encounter several times in the sequel.

\begin{prop}\label{prop:GalWeakDissipEstim} Let us assume that $j \in C(\R_+;L^p_\varrho(\OO))$, $p \in [1,\infty)$,   satisfies 
$$
\frac{\d}{\dt} \| j_t \|^p_{L^p_{\varrho_2}} + \sigma \| j_t \|^p_{L^p_{\varrho_1}} \le 0, 
\quad 
 \| j_t \|^p_{L^p_{\varrho}} \le C  \| j_0 \|^p_{L^p_{\varrho}},  
$$
for some admissible or inverse of admissible weight functions $\varrho,\varrho_2: \R^3 \to (0,\infty)$ such that $1 \le  \varrho_2/\varrho \nearrow  \infty$ as $|v| \to \infty$ and $\varrho_1 := \varrho_2 \langle v \rangle^{(s_2+\gamma)/p}$, $s_2+\gamma < 0$. Here $s_2 \in [0,2]$ is the parameter associated to $\varrho_2$ as defined in \eqref{eq:omega}. Then 
$$
  \| j_t \|_{L^p_{\varrho_2}} \le \Theta_{\varrho,\varrho_2} (t) \| j_0 \|_{L^p_{\varrho}}, \quad \forall \, t \ge 0, 
$$
for some decay function $\Theta_{\varrho,\varrho_2}$ that we will make precise in some particular cases.

\smallskip
(1) If $\varrho := e^{\kappa |v|^2}$, $\kappa \in (1/4,1/2)$, and $\varrho_2 := e^{\tfrac14 |v|^2}$, then 
$$
 \Theta_{\varrho,\varrho_2} (t) = (1+C) e^{- \lambda t^{ 2/ {|\gamma|}}}, \quad \lambda := \sigma^{2/{|\gamma|}} (p(\kappa-\tfrac14))^{ {|2+\gamma|}/{|\gamma|}}p^{-1}. 
$$

\smallskip
(2) If $\varrho := e^{-\kappa |v|^s}$ is the inverse of an admissible weight with $s \in (0,2]$ and $\varrho_2 := e^{-\kappa_2 |v|^s}$ with $\kappa_2 \in (\kappa,\infty)$ if $s \in (0,2)$ and
$\kappa_2 \in (\kappa,1/2)$ if $s=2$, then 
$$
\Theta_{\varrho,\varrho_2}(t) = e^{- \lambda t^{s/|\gamma|}}, \quad \lambda > 0.
$$

\smallskip
(3) If $\varrho := \langle v \rangle^{-k}$ and $\varrho_2 := \langle v \rangle^{-k_2}$, $k_2 > k > k_0$, then 
$$
\Theta_{\varrho,\varrho_2}(t) = C \left[ \frac{\log \la t \ra}{\la t \ra} \right]^{\frac{k_2-k}{|\gamma|}}. 
$$

\end{prop}

\begin{proof}[Proof of Proposition~\ref{prop:GalWeakDissipEstim}] Because $\varrho_2/\varrho_1$ and $\varrho/\varrho_2$ are increasing,  we have 
$$
\eps_R \varrho_2^p \le  \varrho_1^p +  \vartheta_R \varrho^p, 
$$
with $\eps_R :=  \varrho_1(R)^p/\varrho_2(R)^p =  \langle R \rangle^{s_2+\gamma}$ and $ \vartheta_R/\eps_R := \varrho_2(R)^p/\varrho(R)^p$, so that the 
three conditions in Lemma~\ref{lem:Gronwall} are satisfied by $u:= \| j \|^p_{L^p_{\varrho_1}}$, $v :=  \| j \|^p_{L^p_{\varrho_2}}$ and 
$w :=  \| j \|^p_{L^p_{\varrho}}$.  Using  the definition \eqref{eq:lemGronwall} of $\Gamma_t(R)$, we have 
$$
\Gamma_t \le \Gamma_t(R) = e^{- \sigma \langle R \rangle^{s_2+\gamma} t} +\frac{\vartheta_R}{\eps_R} C, \quad \forall \, R > 0, 
$$
and we make an appropriate choice of $R = R(t)$ depending on the case we face to.

\smallskip
\noindent
{\it Case (1).} We   take $R := (\alpha t)^{1/|\gamma|}$, $\alpha := \tfrac\sigma{p(\kappa-1/4)}$, in the definition \eqref{eq:lemGronwall} of $\Gamma_t(R)$, 
so that 
$$
\Gamma_t \le e^{- \sigma t   (\alpha t) ^{  {(2+\gamma)/ |\gamma|}}  } +e^{- p (\kappa-\tfrac14)) (\alpha t)^{{2 / |\gamma|}}} C .
$$

\smallskip
\noindent
{\it Case (2).}  We make the same choice as in Case 1, and we conclude similarly. 
 
\smallskip
\noindent
{\it Case (3).}  We take $\la R \ra = [ \lambda^{-1} (\log \la t \ra)^{-1} \la t \ra    ]^{1 / |\gamma|}$, with $\lambda>0$ to be chosen later, in the definition \eqref{eq:lemGronwall} of $\Gamma_t(R)$. We then get 
\bean
\Gamma_t 
&\le& \Gamma_t(R) \le  e^{- \sigma   R  ^{\gamma} t} + \la R \ra^{-p(k_2-k)} C
\\
&\le&  e^{- \sigma \lambda \log \la t \ra} + C \lambda^{p \frac{(k_2-k)}{|\gamma|}} \left[\frac{\log \la t \ra}{\la t \ra} \right]^{p \frac{(k_2-k)}{|\gamma|}},
\eean 
from which we conclude by taking $\lambda = \frac{p(k_2-k)}{\sigma |\gamma|}$.  
\end{proof}

%%%%%%%%%%%%%%%%%%%%%%%%%%%%%%%%%%%
\section{A first glance at  $S_{\LL_g}$}\label{sec-SLLg-first}

In this section, we consider the equation associated to the linear operator $\LL_g$, we establish some micro and macroscopic dissipativity estimates in a $L^2$ framework and next deduce the well-posedness of the associated linear equation. 

\subsection{Microscopic dissipativity estimates}

Let us introduce the main microscopic dissipativity part of the operator $\LL_g$ defined by 
$$
\BB_0 f := Q(\mu,f) = \bar a_{ij} \partial^2_{v_i v_j}f  - \bar c f, 
$$
where we recall the shorthand \eqref{eq:barabc}.

\begin{lem}\label{lem:Psi_m}	
For any exponent $p \in [1,\infty]$ and any admissible weight function $\omega$, the function $\varpi_{\omega,p}^{\BB_0}$ defined in \eqref{eq:varpiL_omega_p} satisfies 
\be\label{eq:lem:Psi_m-1}
\limsup_{|v| \to \infty}  \left[  \langle v \rangle^{-\gamma - s}  \varpi_{\omega,p}^{\BB_0} (v)  \right] \le \kappa_{\omega,p},
\ee
with $\kappa_{\omega,p} < 0$, and more precisely
\be\label{eq:lem:Psi_m-2}
\kappa_{\omega,p} :=
\begin{cases}
-2k + 2 \left( 1 -  \frac{1}{p} \right) (\gamma+3)  &  \text{if  }  s=0, \\ - 2 k      &    \text{if  }   s \in (0,2), \\ 
  - 2 k  (1-k)    &   \text{if  }  s=2.  
\end{cases}
\ee

\end{lem}

\begin{proof}[Proof of Lemma~\ref{lem:Psi_m}] 
From the very definition \eqref{eq:varpiL_omega_p} of $\varpi_{\omega,p}^{\BB_0}$, % in Lemma~\ref{lem:dissipativity_Lp}, % and with  the notations of the proof of Proposition~\ref{prop:EstimOpLandau3},,  
we have 
\begin{equation}\label{eq:varpiB0_omega_p}
\begin{aligned}
\varpi_{\omega,p}^{\BB_0} (v) 
&= 2\left( 1- \frac{1}{p} \right) \bar a_{ij}  \frac{\partial_{v_i} \omega}{\omega} \frac{\partial_{v_j} \omega}{\omega} + \left( \frac{2}{p}-1 \right) \bar a_{ij} \frac{\partial_{v_i, v_j} \omega}{\omega} + \frac{2}{p} \bar b_i \frac{\partial_{v_i} \omega}{\omega}  + \left(\frac{1}{p}   - 1 \right)  \bar c.
\end{aligned}
\end{equation}
Similarly as in the proof of Proposition~\ref{prop:EstimOpLandau3}, we observe that 
\begin{equation}\label{eq:dim_dijmBIS}
\frac{\partial_{v_i} \omega}{\omega} =   v_i \wp, \quad
\frac{\partial_{v_i, v_j} \omega}{\omega} =  \wp\delta_{ij} +  v_i v_j \wp \left(\wp +  \frac{ s -2}{ \langle v \rangle^2} \right), \quad  \wp  :=  k \langle v \rangle^{s-2}, 
\end{equation}
which implies, with the help of Lemma~\ref{lem:elementary-abc*bar}, 
\bean
\varpi_{\omega,p}^{\BB_0} (v) 
&=& 
\bar a_{ij} v_iv_j \wp \left[   \wp  + \left( \frac{2}{p}-1 \right)    \frac{ s -2}{ \langle v \rangle^2}  \right]
 + \left( \frac{2}{p}-1 \right)\bar a_{ii} \wp+ \frac{2}{p} \bar b_i v_i \wp  + \left( \frac{1}{p} - 1 \right) \bar c
\\
&\underset{|v| \to \infty}{\sim}& 
2k \langle v \rangle^{\gamma+s}  \left[   \wp  + \left( \frac{2}{p}-1 \right)    \frac{ s -2}{ \langle v \rangle^2}  - 1 \right]   + 2 \left(1- \frac{1}{p}\right)  (\gamma+3) \la v \ra^{\gamma} .
\eean
By particularizing the different possible values of the parameters $\gamma$ and $s$, we immediately conclude to \eqref{eq:lem:Psi_m-1}--\eqref{eq:lem:Psi_m-2}. 
\end{proof}

For a given function $g=g(t,x,v)$, we  now introduce $\CC^+_g$ the local collision part   
\beqn\label{eq:defC+gf}
\CC^+_g f : = Q(\mu,f) +   Q(g,f) 
\eeqn
of the  linearized operator $\LL_g$. For an admissible weight function $\omega$, we define the modified weight function
$$
\widetilde \omega = \widetilde \omega(x,v) := \theta \omega,
$$
for a nonnegative function $\theta = \theta(x,v)$ such that 
\beqn\label{eq:bound_theta}
\left| \frac{\nabla_v \theta}{\theta} \right| \lesssim \la v \ra^{-1- \alpha}, \quad
\left| \frac{\nabla_v^2 \theta}{\theta} \right| \lesssim \la v \ra^{-2- \alpha},
\eeqn
for any $(x,v) \in \OO$ and  some $\alpha>0$. 
 
\begin{lem}\label{lem:varpiC+g_tilde_omega_p}
For any exponent $p \in [1,\infty]$ and any admissible weight function $\omega$, there exists $C_{\XX_0,1} > 0$ and a positive  function $\psi$ on $\R^3$ satisfying $\psi(v) \to 0$ as $|v| \to \infty$
  % for any $\kappa < \kappa_{\omega,p}$ there exist $M,R > 0$ 
  such that 
the function $\varpi_{\tilde\omega,2}^{\CC^+_g}$ defined in \eqref{eq:varpiL_omega_p} satisfies 
\be\label{eq:lem:varpiC+g_tilde_omega_p}
 \la v \ra^{-\gamma-s} \varpi_{\tilde\omega,p}^{\CC^+_g} (v)
\le 
 \kappa_{\omega,p}  +  C_{\XX_0,1} \| g \|_{\XX_0}+ \psi, \quad \forall (x,v) \in \OO.
\ee
\end{lem}

\begin{proof}[Proof of Lemma~\ref{lem:varpiC+g_tilde_omega_p}]
We split the proof into four steps. 

\smallskip\noindent
{\sl Step 1.} From the definition of $\varpi_{\cdot ,2}^{\BB_0}$ in \eqref{eq:varpiB0_omega_p} and observing that 
$$
\frac{\partial_{v_i} \widetilde \omega }{ \widetilde \omega} = \frac{\partial_{v_i}   \omega }{ \omega} + \frac{\partial_{v_i}\theta }{ \theta},
\quad
\frac{\partial_{v_iv_j} \widetilde \omega }{ \widetilde \omega} =   \frac{\partial_{v_iv_j}  \omega }{ \omega} +  2  \frac{\partial_{v_i}   \omega }{ \omega}  \frac{\partial_{v_j}\theta }{ \theta} 
+  \frac{\partial_{v_iv_j}\theta }{ \theta},
$$
we have 
$$
\begin{aligned}
\varpi_{\tilde \omega ,p}^{\BB_0}
 = \varpi_{\omega ,p}^{\BB_0}
+ 2  \bar a_{ij} \frac{\partial_{v_i} \omega}{\omega} \frac{\partial_{v_j} \theta}{\theta} 
+   2\left( 1- \frac{1}{p} \right) \bar a_{ij} \frac{\partial_{v_i} \theta}{\theta} \frac{\partial_{v_j} \theta}{\theta} 
+ \left( \frac{2}{p}-1 \right) \bar a_{ij} \frac{\partial_{v_i, v_j} \theta}{\theta}
+  \frac2p\bar b_i \frac{\partial_{v_i} \theta}{\theta} . 
\end{aligned}
$$
%\begin{aligned}
%\varpi_{\tilde \omega ,2}^{\BB_0}
%&= \varpi_{\omega ,2}^{\BB_0}
%+ 2p \bar a_{ij} \frac{\partial_{v_i} m}{m} \frac{\partial_{v_j} \theta}{\theta} 
%+ (p-1) \bar a_{ij} \frac{\partial_{v_i} \theta}{\theta} \frac{\partial_{v_j} \theta}{\theta} 
%+ \bar a_{ij} \frac{\partial_{v_i,v_j} \theta}{\theta}.
%\end{aligned}
%$$
Thanks to Lemma~\ref{lem:elementary-abc*bar},  \eqref{eq:dim_dijmBIS} and   \eqref{eq:bound_theta}, we get
$$
\left| \bar a_{ij} \frac{\partial_{v_i} \omega}{\omega} \frac{\partial_{v_j} \theta}{\theta} \right|
\lesssim k \la v \ra^{\gamma + s -2 - \alpha},
\quad
\left| \bar a_{ij} \frac{\partial_{v_i} \theta}{\theta} \frac{\partial_{v_j} \theta}{\theta} \right|
\lesssim   \la v \ra^{\gamma - 2 \alpha},
$$
and  
$$
\left| \bar a_{ij}   \frac{\partial_{v_iv_j} \theta}{\theta} \right|
\lesssim   \la v \ra^{\gamma  - \alpha},
\quad
\left| \bar b_{i} \frac{\partial_{v_i} \theta}{\theta}\right|
\lesssim \la v \ra^{\gamma-\alpha}.
$$
The identity and these estimates together imply
$$
\langle v \rangle^{-\gamma-s} \varpi_{\tilde \omega ,2}^{\BB_0}
\le \langle v \rangle^{-\gamma-s} \varpi_{\omega ,2}^{\BB_0} + C \la v \ra^{ -\alpha}.
$$

\smallskip\noindent
{\sl Step 2.} We now consider $\varpi^{Q(g,\cdot)}_{\omega,p}$ for an exponent $p \in [1,\infty]$. 
Arguing in the same way  as in the proof of Lemma~\ref{lem:Psi_m}, we have 
\bean
\varpi^{Q(g,\cdot)}_{\omega,p} (v) 
&=& 
(a_{ij} * g) v_iv_j \wp \left[   \wp  + \left( \frac{2}{p}-1 \right)    \frac{ s -2}{ \langle v \rangle^2}  \right]
\\
&& + \left( \frac{2}{p}-1 \right) (a_{ii} * g)\wp+ \frac{2}{p}  (b_i *g) v_i \wp  + \left( \frac{1}{p} - 1 \right) (c*g).
\eean
Thanks to Lemma~\ref{lem:elementary-abc*g} and the very definition of $\wp$, we deduce 
$$
|\varpi^{Q(g,\cdot)}_{\omega,p} | \lesssim  k^2 \la v \ra^{\gamma+s}  \| g \|_{L^\infty_{\omega_0}}. 
$$

\smallskip\noindent
{\sl Step 3.} We finally consider  $\varpi^{Q(g,\cdot)}_{\tilde\omega,p}$. Similarly as in Step 1, we compute
$$
\begin{aligned}
\varpi_{\tilde \omega ,p}^{Q(g,\cdot)}
&= \varpi_{\omega ,p}^{Q(g,\cdot)}
+ 2 ( a_{ij} * g) \frac{\partial_{v_i} \omega}{\omega} \frac{\partial_{v_j} \theta}{\theta} 
+   2\left( 1- \frac{1}{p} \right)( a_{ij} * g) \frac{\partial_{v_i} \theta}{\theta} \frac{\partial_{v_j} \theta}{\theta} 
\\
&+ \left( \frac{2}{p}-1 \right) ( a_{ij} * g) \frac{\partial_{v_i, v_j} \theta}{\theta}
+  \frac2p ( b_i  * g) \frac{\partial_{v_i} \theta}{\theta} . 
\end{aligned}
$$
Using the same estimates as in  Step 1 and the conclusion of Step 2, we find 
$$
|\varpi^{Q(g,\cdot)}_{\tilde\omega,p} | \lesssim  C_{\XX_0} \la v \ra^{\gamma+s}  \| g \|_{L^\infty_{\omega_0}}, 
$$
for some constant $C_{\XX_0}$.

\smallskip\noindent
{\sl Step 4.} Using that 
$$
\varpi^{\CC^+_g}_{\tilde\omega,p} = \varpi^{\BB_0}_{\tilde\omega,p} + \varpi^{Q(g,\cdot)}_{\tilde\omega,p}
$$
and  the estimates established in Step 1, in Step 3 and in Lemma~\ref{lem:Psi_m}, we obtain
\be\label{eq:varpi_omegatilde_CC+}
\limsup_{|v| \to \infty}  \left[   \la v \ra^{-\gamma-s} \varpi_{\tilde\omega,p}^{\CC^+_g} (v)    \right] \le \kappa_{\omega,p} + C_{\XX_0} \|g \|_{\XX_0}, 
\ee
from which we immediately conclude. 
\end{proof}

\subsection{Dissipativity estimate in $L^2$}

In this section, we establish some (possibly weak) dissipativity property for the solutions to the linear equation \eqref{eq:linear_g}.

\begin{prop}\label{prop:LLg_L2}
Consider an admissible weight function $\omega$. There exist constants $\eps_1, \sigma, R_1, M_1 > 0$ (only depending on $\omega$) and a modified  weight function $\widetilde \omega : \OO \to (0,\infty)$ with equivalent velocity growth as $\omega$ such that if  $\| g \|_{\XX_0} \le \eps_1$, then for any solution $f$ to the linear equation \eqref{eq:linear_g} associated to the linear operator $\LL_g$ and the reflection boundary condition \eqref{eq:reflect_F}, there holds
\beqn\label{eq:SLg-L2H1*BIS}
 \frac12 \frac{\d}{\d t} \|  f \|_{L^2_{x,v} (\tilde\omega)}^2 + \sigma \|  f \|_{L^2_x H^{1,*}_v (\tilde \omega)}^2  \le M_1  \|  f \|_{L^2_{x,v}(\Omega \times B_{R_1})}^2  . 
\eeqn

\end{prop}

It is worth emphasizing that depending of the value of $\gamma$ and the choice of the weight function $\omega$ this differential inequality provides the dissipativity property (exponential decay) of the norm (when 
$H^{1,*}_v (\omega) \subset L^2_v ( \omega)$) or not.

\begin{proof}[Proof of Proposition~\ref{prop:LLg_L2}]
We split the proof into six steps. 

\medskip\noindent
\textit{Step 1.}
We define the modified weight function $\omega_A = \omega_A(v)$ by
\begin{equation}\label{eq:def:omega_A}
\omega_A^2 = \chi_A \MMM^{-1} + (1-\chi_A) \omega^2
\end{equation}
where $\chi_A (v) = \chi (\tfrac{|v|}{A})$, $A \ge 1$ will be chosen later (large enough), and $\chi \in C^2(\R_+)$ with $\mathbf 1_{[0,1]} \le \chi \le \mathbf 1_{[0,2]}$.
We then define a second modified weight function $\widetilde{\omega} = \widetilde{\omega}(x,v)$ by
\begin{equation}\label{eq:def:tilde_omega}
  \widetilde \omega^2 = \left\{ 1 + \frac12 (n_x \cdot   v)   \langle v \rangle^{\gamma-3}\right\} \omega_A^2,  
\end{equation}
and we observe that
$$
1 \le \omega_A \le  c_A \omega \quad\text{and}\quad
\frac12 \omega_A^2 \le \widetilde \omega^2 \le \frac32 \omega_A^2,
$$
for some constant $c_A \ge 1$.
We finally remark that we can write
$$
\widetilde \omega = \theta \omega
$$
with
$$
\theta^2 = \left[1 + \frac12 (n_x \cdot   v)   \langle v \rangle^{\gamma-3} \right] \left[ 1 + \chi_A (\MMM^{-1} \omega^{-2}-1)\right]
$$
that satisfies, for any $A>0$,
$$
\left| \frac{\partial_{v_i} \theta}{\theta} \right| \lesssim \la v \ra^{-2}, \quad
\left| \frac{\partial_{v_i, v_j} \theta}{\theta} \right| \lesssim \la v \ra^{-3}.
$$
Given a solution $f$ to the linear equation  \eqref{eq:linear_g}, we  write
\begin{equation}\label{eq:ddt_f_L2}
\begin{aligned}
\frac12 \frac{\d}{\d t} \| f \|_{L^2_x L^2_v( \tilde \omega)}^2
&= \int_\OO  (\CC^+_g f) f \widetilde \omega^2
+ \int_\OO  (\CC^-_g f)  f \widetilde \omega^2 
+ \int_\OO f^2 v \cdot \nabla_x (\widetilde \omega ^2) \\
&\quad
- \int_{\Sigma} (\gamma f)^2 \widetilde \omega^2 (n_x \cdot v), 
\end{aligned}
\end{equation}
where $\CC^-_g$ stands for the nonlocal collision part    
\beqn\label{eq:defC-gf}
\CC^-_g f := Q(f,\mu) - \pi Q(g,f).
\eeqn
of the  linearized operator $\LL_g$, and $\CC^+_g$ is the local collision part defined in \eqref{eq:defC+gf}.

\smallskip\noindent
\textit{Step 2.} For the  first term at the right-hand side of~\eqref{eq:ddt_f_L2}, we may use Lemma~\ref{lem:dissipativity_Lp}
$$
\int_\OO (\CC^+_g f) f \widetilde \omega^2 
 = -   \int \widetilde a_{ij}   \partial_{v_i} ( \widetilde \omega f) \partial_{v_j}  ( \widetilde \omega f)  + \int_\OO \varpi_{\tilde \omega,2}^{\CC^+_g} f^2  \widetilde \omega^2,
 $$
where $\widetilde a_{ij} := a_{ij} * (g + \mu)$ and $\varpi_{\tilde \omega,2}^{\CC^+_g}$ satisfies \eqref{eq:lem:varpiC+g_tilde_omega_p} in Lemma~\ref{lem:varpiC+g_tilde_omega_p} thanks to the above estimates on $\theta$. 
We observe that  
\bean
\widetilde a_{ij}   \partial_{v_i} ( \widetilde \omega f) \partial_{v_j} ( \widetilde \omega f)
&=& \bar a_{ij}   \partial_{v_i} ( \widetilde \omega f) \partial_{v_j} ( \widetilde \omega f) + (a_{ij} * g)  \partial_{v_i} ( \widetilde \omega f) \partial_{v_j} ( \widetilde \omega f)
\\
&\ge& \left(\bar C - C_{\XX_0,2} \| g \|_{\XX_0} \right) \langle v \rangle^\gamma |\widetilde \nabla_v ( \widetilde \omega f) |^2, 
\eean
for a constant $\bar C > 0$ given by   Lemma~\ref{lem:elementary-abc*bar} and \eqref{eq:Bnabla=tildenabla},  and for a constant $C_{\XX_0,2} $ given by   \eqref{aijFiHj} and \eqref{eq:Bnabla=tildenabla}. For the second term at the right-hand side of~\eqref{eq:ddt_f_L2}, we may use  \eqref{eq:AA0L2L2} and \eqref{eq:piQgf_omega} in order to get 
$$
\int_\OO (\CC^-_g f) f \widetilde \omega^2 
 \le  \int_\OO ( C_{\AA_0} \langle v \rangle^{\gamma-1} + C_{\XX_0,3} \| g \|_{\XX_0} \langle v \rangle^{\gamma})  f^2  \widetilde \omega^2,
 $$
for constants $C_{\AA_0} , C_{\XX_0,3} >0$ given by Lemma~\ref{lem:borneA0} and Lemma~\ref{prop:EstimOpLandau3} respectively.

\medskip\noindent
\textit{Step 3.}
For the third term at the right-hand side of~\eqref{eq:ddt_f_L2} we observe that $\nabla_x (\widetilde \omega^2) = \frac12 D_x n_x  v \, \langle v \rangle^{\gamma-3}  \omega_A^2$ and  therefore
$$
\begin{aligned}
\int_\OO f^2 v \cdot \nabla_x (\widetilde \omega ^2) 
&= \frac12  \int_\OO f^2 \left( v \cdot D_x n_x  v \right)  \langle v \rangle^{\gamma-3} \omega_A^2   \\
&\le C_\TT \int_\OO f^2 \widetilde \omega^2 \langle v \rangle^{\gamma-1},
\end{aligned}
$$
for some constant $C_\TT >0$.

\medskip\noindent
\textit{Step 4.}
The boundary term in~\eqref{eq:ddt_f_L2} can be decomposed as
$$
\begin{aligned}
\int_{\Sigma} (\gamma f)^2 \widetilde \omega^2 (n_x \cdot v)
&= \int_{\Sigma}  (\gamma f)^2 \omega_A^2 (n_x \cdot v) 
+\frac12  \int_{\Sigma} (\gamma f)^2 \omega_A^2 (n_x \cdot  v)^2  \langle v \rangle^{\gamma-3}.
\end{aligned}
$$
On the one hand we have
$$
\int_{\Sigma} (\gamma f)^2 \omega_A^2 (n_x \cdot v)
= \int_{\Sigma_+} (\gamma_+ f)^2 {\omega_A}^2 |n_x \cdot v| - \int_{\Sigma_-} (\gamma_- f)^2 {\omega_A}^2 |n_x \cdot v|.
$$
Using the boundary condition in \eqref{eq:linear_g} together with the fact that $s \mapsto s^2$ is convex, we get
$$
\begin{aligned}
\int_{\Sigma_-} (\gamma_- f)^2 {\omega_A}^2 |n_x \cdot v|
&= \int_{\Sigma_-} \left\{ (1-\iota)\SSS \gamma_+ f + \iota \DDD \gamma_+ f \right\}^2 {\omega_A}^2 |n_x \cdot v| \\
&\le \int_{\Sigma_-}  (1-\iota)(\SSS \gamma_+ f)^2 {\omega_A}^2 |n_x \cdot v|
+ \int_{\Sigma_-} \iota  (\widetilde{\gamma_+ f})^2 \MMM^2 {\omega_A}^2 |n_x \cdot v|.
\end{aligned}
$$
Making the change of variables $v \mapsto \VV_x v$ in the integral over $\Sigma_-$ and observing that $|v| = |\VV_x|$ and $n_x \cdot v = - n_x \cdot \VV_x$, we have
$$
\begin{aligned}
\int_{\Sigma_-} (\gamma_- f)^2 {\omega_A}^2 |n_x \cdot v|
&\le \int_{\Sigma_+}  (1-\iota)(\gamma_+ f)^2 {\omega_A}^2 |n_x \cdot v|
+ \int_{\Sigma_+} \iota  (\widetilde{\gamma_+ f})^2 \MMM^2 {\omega_A}^2 |n_x \cdot v| .
\end{aligned}
$$
Altogether, we have established
$$
\int_{\Sigma} (\gamma f)^2   \omega_A^2 (n_x \cdot v )
\ge \int_{\Sigma_+} \iota(\gamma_+ f)^2 {\omega_A}^2 |n_x \cdot v| -\int_{\Sigma_+} \iota  (\widetilde{\gamma_+ f})^2 \MMM^2 {\omega_A}^2 |n_x \cdot v|.
$$
By the Cauchy-Schwarz inequality, there holds
$$
(\widetilde{\gamma_+ f})^2(x) \le K_2(\omega_A) \int_{\Sigma^x_+} (\gamma_+f)^2 \omega_A^2 (n_x \cdot v)_+
$$
with
$$
K_2(\omega_A) = \int_{\R^3}  \omega_A^{-2}  (n_x \cdot v)_+ \, \d v < \infty.
$$
Denoting
$$
K_1(\omega_A) = \int_{\R^3} \MMM^2 \omega_A^2  (n_x \cdot v)_+ \, \d v < \infty,
$$
we thus deduce
$$
\int_{\Sigma} (\gamma f)^2   \omega_A^2 (n_x \cdot v )
\ge  \int_{\partial \Omega} \iota \left[K_2(\omega_A)^{-1} - K_1(\omega_A)\right] (\widetilde{\gamma_+ f})^2.
$$
On the other hand, thanks to the Cauchy-Schwarz inequality, we have
$$
\int_{\Sigma} (\gamma f)^2 \omega_A^2 (n_x \cdot v)^2  \langle v \rangle^{\gamma-3}
\ge K_0(\omega_A)^{-1} \int_{\partial \Omega} (\widetilde{\gamma_+ f})^2,
$$
where we denote
$$
K_0(\omega_A) = \int_{\R^3}   \langle v \rangle^{3-\gamma} \omega_A^{-2} \, \d v < \infty.
$$
For the boundary term in~\eqref{eq:ddt_f_L2},
we finally obtain the following bound
\begin{equation}\label{eq:boundary_estimate_L2}
-\int_{\Sigma} (\gamma f)^2  \varphi^2 \widetilde{\omega}^2 (n_x \cdot v)
\le \int_{\partial \Omega} \iota \left[K_1(\omega_A) - K_2(\omega_A)^{-1} - \frac12 K_0(\omega_A)^{-1}  \right]  (\widetilde{\gamma_+ f})^2 .
\end{equation}
Observing that $\omega_A \to \MMM^{-1/2}$ when $A \to \infty$, we deduce that $K_0(\omega_A) \to K_0(\MMM^{-1/2}) >0$, $K_1(\omega_A) \to K_1(\MMM^{-1/2}) = 1$ and $K_2(\omega_A) \to K_2(\MMM^{-1/2}) = 1$ thanks to the normalization condition on $\MMM$. We therefore may choose $A>0$ large enough such that
\begin{equation}\label{eq:choice_A}
K_1(\omega_A) - K_2(\omega_A)^{-1} - \frac12 K_0(\omega_A)^{-1} \le 0.
\end{equation}

\smallskip\noindent
\textit{Step 5.} Coming back to \eqref{eq:ddt_f_L2}, throwing away the last term thanks to Step 4 and gathering the estimates of Step~2 and Step~3, we obtain 
$$
\begin{aligned}
\frac12 \frac{\d}{\d t} \| f \|_{L^2_{x,v}( \tilde \omega)}^2
&\leq -  (\bar C - C_{\XX_0,2} \| g \|_{\XX_0}) \| \langle v \rangle^{\gamma/2} \widetilde \nabla_v ( \widetilde \omega f) \|_{L^2_{x,v}}^2    
+ \int_\OO \widetilde \varpi_g \, f^2 \widetilde \omega^2 
\end{aligned}
$$
with 
\begin{equation}\label{eq:tilde_varpi_g_estimate}
 \la v \ra^{-\gamma-s} \widetilde \varpi_g
:= \kappa_{\omega,2}  +  C_{\XX_0,1} \| g \|_{\XX_0} + \psi + C_{\XX_0,3} \| g \|_{\XX_0} \langle v \rangle^{-s}  + (C_{\AA_0} + C_\TT)  \langle v \rangle^{-1-s} ,
\end{equation}
where we recall that $\kappa_{\omega,2} < 0$ is defined in \eqref{eq:lem:Psi_m-2} at Lemma~\ref{lem:Psi_m}.
Defining 
\begin{equation}\label{eq:eps1}
\eps_1 := \frac12 \min \left( \frac{\bar C}{  C_{\XX_0,2}} , \frac{|\kappa_{\omega,2}|}{C_{\XX_0,1} +   C_{\XX_0,3}} \right) >0,
\end{equation}
we deduce that 
\bean
 \la v \ra^{-\gamma-s} \widetilde \varpi_g 
&\le& \tfrac12  \kappa_{\omega,2}  + \psi + (C_{\AA_0} + C_\TT)  \langle v \rangle^{-1-s} 
\\
&\le& \tfrac13  \kappa_{\omega,2}  +  \la v \ra^{-\gamma-s} M_1 \mathbf{1}_{\Omega \times B_{R_1}} ,
\eean 
for some constants $M_1,R_1 > 0$. We conclude by choosing $\sigma := \min (\tfrac13  |\kappa_{\omega,2}|, \tfrac12 \bar C )$. 
\end{proof}

\subsection{The semigroup $S_{\LL_g}$}

We prove the well-posedness of the linear equation \eqref{eq:linear_g} associated to the operator $\LL_g$ in a weighted $L^2$ framework and the fact that we may associate to it a non-autonomous semigroup (or {\it evolution system} \cite[Chapter 5]{MR0710486}). For further reference, we introduce the set $\CCC_\iota$ associated to the conservation laws (C1) and (C2) and defined by 
\bean
&&\CCC_\iota := \{ f \in L^1_{xv}(\langle v \rangle^2);  \  \lla f \rra = 0  \}     \hbox{ if } \iota\not\equiv 0,
\\
&&\CCC_\iota := \{ f \in L^1_{xv}(\langle v \rangle^2);  \  \lla f \rra =  \lla f |v|^2 \rra =  \lla f R \cdot v \rra =  0, \, \forall \, R \in \RR_\Omega\}   \hbox{ if } \iota\equiv 0,
\eean
 and then define $\Pi^\perp = I-\Pi$, where $\Pi$ is the projector associated to the conservation laws set $\CCC_\iota$. More precisely, for $\iota \not\equiv 0$, $\Pi$ is the orthogonal projector on $\mu$ in $L^2_{xv}(\mu^{-1/2})$ and for $\iota \equiv 0$, $\Pi$ is the orthogonal projector  in $L^2_{xv}(\mu^{-1/2})$ on the subspace generated by 
$$
\{ \mu;  R(x) \cdot v \mu, \ R \in \RR_\Omega;  |v|^2 \mu \}. 
$$

\begin{theo}\label{theo-SG-LLg}
Consider an admissible weight function $\omega$ and a function $g \in \XX_0$ such that $\| g \|_{\XX_0} \le \eps_1$. 
We denote by $\widetilde{\omega} $ the modified weight function introduced in Proposition~\ref{prop:LLg_L2}.
%defined by \eqref{eq:def:tilde_omega} for a convenient choice of $A > 0$ such that the conclusion of Proposition~\ref{ } holds. 

For any  $t_0 \ge 0$ and $f_{t_0} \in L^2_\omega(\OO)$, there exists a  unique weak solution $f \in C([t_0,T];L^2_\omega) \cap  L^2((t_0,T) \times \Omega; H^{1,*}_{v}(\widetilde\omega))$, $\forall \, T > t_0$,  to the linear equation \eqref{eq:linear_g} associated to the initial datum $f_{t_0}$. This one satisfies the dissipativity estimate \eqref{eq:SLg-L2H1*} and it satisfies $f_t \in \CCC_\iota$ for any $t \ge 0$ if $f_0 \in \CCC_\iota$. 
 As in Theorem~\ref{theo-Kolmogorov-WellP}, the evolution PDE equation in \eqref{eq:linear_g} is satisfied in the distributional sense  and the trace and initial conditions  in \eqref{eq:linear_g} are satisfied pointwisely by the trace functions $\gamma f$ and $f(0,\cdot)$ provided by the trace Theorem~\ref{theo-Kolmogorov-trace}.

As a consequence,  the mapping $(t_0,t) \mapsto S_{\LL_g}(t,t_0) f_{t_0} := f_t$ defines a non-autonomous semigroup on $L^2(\omega)$  such that \eqref{eq:SLg-L2H1*} holds true. 
 The conservation laws may be   expressed by the fact that $S_{\LL_g}$ defines a semigroup on  $L^2(\omega) \cap \CCC_\iota$ or equivalently that the identity $\Pi^\perp S_{\LL_g} =  S_{\LL_g}  \Pi^\perp$ holds.  
\end{theo}

\begin{proof}[Proof of Theorem~\ref{theo-SG-LLg}] 
Because the condition on $g$ still holds after time translation, we may reduces the discussion to the case $t_0 = 0$. 
Thanks to the dissipativity estimate established in Proposition~\ref{prop:LLg_L2}, the well-posedness is a direct application of Theorem~\ref{theo-Kolmogorov-WellP} to the operators
$$
\mathbf{L} f := Q(\mu+g,f) = \widetilde a_{ij} \partial_{v_i , v_j} f - \widetilde c f,
$$
where $\widetilde a_{ij} = a_{ij} * (\mu+g)$, $\widetilde c = c* (\mu + g)$ and 
$$
\KKK[f] := Q(f,\mu) - \pi Q(g,f),
$$
in the space $L^2_{\tilde\omega}(\OO)$ which provides a unique solution $f  \in C([0,T],L^2(\omega)) \cap L^2((0,T) \times \Omega; H^{1,*}_{\tilde\omega})$, for all $ T > 0$,  to the linear equation \eqref{eq:linear_g} 
associated to any given initial datum $f(0) = f_0  \in L^2(\omega)$. From the well-posedness of this linear problem, we may associate a  semigroup $S_{\LL_g}$ by setting $t \mapsto S_{\LL_g}(t,0) f_0 := f_t$ for any $t \ge 0$. The estimate \eqref{eq:SLg-L2H1*} is a consequence of \eqref{eq:SLg-L2H1*BIS} and Gr\"onwall's lemma. The conservation laws $f_t \in \CCC_\iota$ follows from the discussion in Section~\ref{subsec:conservationsLaw}.
\end{proof}

%%%%%%%%%%%%%%%%%%%%%%%%%%%%%%%%%%%
\section{Decay estimates for $S_{\BB_g}$}\label{sec-SBBg-decay}

\subsection{Dissipativity estimate on $S_{\BB^*_g}$}
For a given function $g=g(t,x,v)$ such that $g \in \XX_0$,  we recall the splitting $\LL_g = \BB_g + \AA_g$ in \eqref{def:BBg}--\eqref{def:AAg}, namely 
\bean
\AA_g f &:=&  \CC^-_g f +  M \chi_R, % \quad \CC^-_g f := Q(f,\mu) - \pi Q(g,f) 
\\
\BB_g f &:=&  - v \cdot \nabla_x f + \CC^+_g f  - M \chi_R, % \quad \CC^+_g f  := Q(\mu,f) + Q(g,f)
\eean
with $M,R >0$ to be chosen later, and where we recall that $\CC^\pm_g$ have been defined in \eqref{eq:defC+gf} and \eqref{eq:defC-gf}. %respectively. 
We are interested in the decay property of the semigroup $S_{\BB_g}$ associated to  the primal problem
\begin{equation}\label{eq:BBg}
\left\{
\bal
\partial_t f &= \BB_g f    \quad&\text{in} \quad (0,T) \times \OO ,\\
\gamma_- f &=  \RRR \gamma_+ f   \quad &\text{on} \quad (0,T) \times \Sigma_- , \\ 
 f(0) &= f_0 \quad &\text{in} \quad \OO,  
\eal
\right.
\end{equation}
for any given initial datum $f_0$ and any $T > 0$. Most of the job will be done on  the dual semigroup $S_{\BB^*_g}$ associated to the  backward dual problem
\begin{equation}\label{eq:dual_BBg}
\left\{
\bal
- \partial_t h &=  \BB_g^* h   \quad&\text{in} \quad (0,T) \times \OO ,\\
\gamma_+ h &=  \RRR^* \gamma_- h   \quad &\text{on} \quad (0,T) \times \Sigma_+ , \\ 
 h(T) &= h_T \quad &\text{in} \quad \OO , 
\eal
\right.
\end{equation}
for any final datum $h_T$. Here the dual operator $\BB^*_g$ is defined by 
$$
\begin{aligned}
\BB_g^* h =  v \cdot \nabla_x h 
+ (\CC^+_g)^* h 
- M \chi_R h
\end{aligned}
$$
with
$$
(\CC^+_g)^* h = (a_{ij}*[\mu+g]) \partial_{v_i, v_j} h + 2 (b_{i}*[\mu+g]) \partial_{v_i} h ,
$$
and the dual reflection operator $\RRR^*$ is defined by 
\begin{equation}\label{eq:reflection*}
\RRR^*  = (1-\iota)  \SSS   + \iota \DDD^*  
\end{equation}
where  $\DDD^*$ is defined on $\Sigma_-$ by
$$ 
\DDD^* h (x) = \widetilde{\MMM h \,\,} \!(x) 
:= \int_{\Sigma^-_x} h(x,w) \MMM(w) (n_x \cdot w)_{-} \, dw.
%:= \int_{\R^d} h(x,w) \MMM(w) (n_x \cdot w)_{-} \, dw.
$$ 
 
\begin{prop}\label{prop:BBg_L2}
Consider an admissible weight function $\omega$. There exist constants $\eps_2, \sigma, R_2, M_2 > 0$ (only depending on $\omega$) and a modified  weight function $\widetilde \omega : \OO \to (0,\infty)$ with equivalent velocity growth as $\omega$ such that if  $\| g \|_{\XX_0} \le \eps_2$, then  any solution $f$ to the linear equation \eqref{eq:BBg} associated to $\BB_g$ satisfies, for any $M \ge M_2$ and $R \ge R_2$,
\beqn\label{eq:SBg-L2H1}
 \frac12 \frac{\d}{\d t} \|  f \|_{L^2_{x,v} (\tilde\omega)}^2 + \sigma \|  f \|_{L^2_x H^{1,*}_v (\tilde \omega)}^2  \le 0.
 \eeqn
\end{prop}

\begin{proof}[Proof of Proposition~\ref{prop:BBg_L2}]
Defining $\widetilde \omega$ as in the Step~1 of the proof of Proposition~\ref{prop:LLg_L2}, any solution $f$ to \eqref{eq:BBg} satisfies
$$
\begin{aligned}
\frac12 \frac{\d}{\d t} \| f \|_{L^2_x L^2_v( \tilde \omega)}^2
&= \int_\OO  \left[ (\CC^+_g f) - M \chi_R  f \right] f \widetilde \omega^2
+ \int_\OO f^2 v \cdot \nabla_x (\widetilde \omega ^2) 
- \int_{\Sigma} (\gamma f)^2 \widetilde \omega^2 (n_x \cdot v). 
\end{aligned}
$$
Arguing exactly as in the proof of Proposition~\ref{prop:LLg_L2}, we obtain
$$
\begin{aligned}
\frac12 \frac{\d}{\d t} \| f \|_{L^2_{x,v}( \tilde \omega)}^2
&\leq -  (\bar C - C_{\XX_0,2} \| g \|_{\XX_0}) \| \langle v \rangle^{\gamma/2} \widetilde \nabla_v ( \widetilde \omega f) \|_{L^2_{x,v}}^2    
+ \int_\OO \left( \widetilde \varpi_{\BB_g} - M \chi_R \right) f^2 \widetilde \omega^2 
\end{aligned}
$$
with now
\begin{equation}\label{eq:tilde_varpi_g_estimate_bis}
 \la v \ra^{-\gamma-s} \widetilde \varpi_{\BB_g}
:= \kappa_{\omega,p}  +  C_{\XX_0,1} \| g \|_{\XX_0} + \psi   + C_\TT  \langle v \rangle^{-1-s} , 
\end{equation}
where $\kappa_{\omega,2} < 0$ is defined in \eqref{eq:lem:Psi_m-2},  $C_{\XX_0,1}$ is defined in Lemma~\ref{lem:varpiC+g_tilde_omega_p} %at Lemma~\ref{lem:Psi_m} 
and $C_\TT >0$ is the constant appearing in the proof of Proposition~\ref{prop:LLg_L2}.
We then define
\begin{equation}\label{eq:eps2}
\eps_2 := \frac12 \min \left( \frac{\bar C}{  C_{\XX_0,2}} , \frac{|\kappa_{\omega,p}|}{C_{\XX_0,1}} \right) >0,
\end{equation}
and we deduce that 
\bean
 \la v \ra^{-\gamma-s} \widetilde \varpi_{\BB_g}
&\le& \tfrac12  \kappa_{\omega,p}  + \psi +  C_\TT  \langle v \rangle^{-1-s} 
\\
&\le& \tfrac13  \kappa_{\omega,p}  +  \la v \ra^{-\gamma-s} M_2 \mathbf{1}_{\Omega \times B_{R_2}} ,
\eean 
for some constants $M_2,R_2 > 0$. We conclude by observing that $\widetilde \varpi_{\BB_g} - M \chi_R \le \la v \ra^{\gamma+s} \kappa_{\omega,p}/3 $  and choosing $\sigma := \min (\tfrac13  |\kappa_{\omega,p}|, \tfrac12 \bar C )$. 
\end{proof}
  
\begin{prop}\label{prop:BB*g_Lq}
Let us consider an admissible weight function $\omega$ and an exponent $q \in \{1\} \cup [2,\infty)$. There exist constants $\eps_3,M_3,R_3, \sigma >0$ (only depending on $\omega$ and $q$) and a modified weight function $\widetilde m : \OO \to (0,\infty)$ with equivalent velocity decay as $m := \omega^{-1}$ such that if  $\| g \|_{\XX_0} \le \eps_2$, then any solution $h$ to the dual backward linear problem   \eqref{eq:dual_BBg} associated to $\BB^*_g$ satisfies, for any $M \ge M_3$ and $R \ge R_3$,
\beqn\label{eq:prop:BB*g_Lq}
- \frac1q \frac{\d}{\d t}  \|  h\|_{L^q(\tilde m)}^q + \sigma \| \langle v \rangle^{\frac{(\gamma+s)}{q} } h \|_{L^q(\tilde m) }^q \le 0  . 
\eeqn

\end{prop}

\begin{proof}[Proof of Proposition~\ref{prop:BB*g_Lq}]
Arguing in a similar way as during the proof of Proposition~\ref{prop:LLg_L2}, we split the proof into five steps. 

\medskip\noindent
\textit{Step 1.}
We first define the weight function $m_A = m_A(v)$ by
\begin{equation}\label{eq:def:m_A}
m_A^q = \chi_A\MMM + (1-\chi_A) m^q
\end{equation}
where $\chi_A (v) = \chi (\tfrac{|v|}{A})$, $A \ge 1$ will be chosen later (large enough), and $\chi \in C^2(\R_+)$ with $\mathbf 1_{[0,1]} \le \chi \le \mathbf 1_{[0,2]}$. We then define the modified weight $\widetilde{m} = \widetilde{m}(x,v)$ by
\begin{equation}\label{eq:def:tilde_m}
\widetilde m^q = m_A^q \left\{ 1- \frac{1}{2} (n_x \cdot v) \la v \ra^{\gamma-3} \right\} , %
 \end{equation}
and we observe that
$$
c_A^{-1}\MMM \le m_A^q \le c_A m^q \quad\text{and}\quad
\frac12 m_A^q \le \widetilde m^q \le  \frac32  m_A^q,
$$
for some constant $c_A >0$. We remark that we can write
$$
\widetilde m^q = \theta^q m^q
$$
with
$$
\theta^q = \left[ 1 - \frac{1}{2} (n_x \cdot  v) \la v \ra^{\gamma-3} \right] \left[ 1 + \chi_A (\MMM m^{-q}-1) \right],
$$
which satisfies, for any $A>0$,
$$
\left| \frac{\partial_{v_i} \theta}{\theta} \right| \lesssim \la v \ra^{-2}, \quad
\left| \frac{\partial_{v_i, v_j} \theta}{\theta} \right| \lesssim \la v \ra^{-3}.
$$

We may then write
\begin{equation}\label{eq:ddt_h_Lq}
\begin{aligned}
-\frac{1}{q} \frac{\d}{\d t} \| h \|_{L^q_{x,v}( \widetilde m)}^q
&= \int_\OO   [(\CC^+_g)^*h  - M\chi_R h]  h |h|^{q-2} 
  \widetilde m^q  \\
&\quad
+ \frac1q \int_\OO |h|^q \, v \cdot \nabla_x (\widetilde m^q) 
+ \int_{\Sigma} |\gamma h|^q \widetilde m^q (n_x \cdot v),
\end{aligned}
\end{equation}
and we estimate each term separately.

\medskip\noindent
\textit{Step 2.} For the first term  at the right-hand side of~\eqref{eq:ddt_h_Lq}, Lemma~\ref{lem:dissipativity_Lp} implies 
$$
\begin{aligned}
\int_{\OO}   \left[(\CC^+_g)^*h  - M\chi_R h \right]  h |h|^{q-2}  \widetilde m^q
&= -\frac{4(q-1)}{q^2} \int_{\OO} \widetilde a_{ij} \partial_{v_i} H \partial_{v_j} H \\
&\quad
+\int_{\OO} \left\{ \varpi^{(\CC^+_g)^*}_{\tilde m,q} - M \chi_R \right\} |h|^q \widetilde m^q,
\end{aligned}
$$
with $H := \widetilde m^{q/2} h|h|^{q/2-1}$ and $\widetilde a_{ij} := a_{ij} * (\mu+g)$. 
As in Step~2 of the proof of Proposition~\ref{prop:LLg_L2}, we have 
$$
\widetilde a_{ij} \partial_{v_i} H \partial_{v_j} H 
\ge (\bar C - C_{\XX_0,2} \| g \|_{\XX_0}) \la v \ra^\gamma |\widetilde \nabla_v H|^2,
$$
for positive constants $\bar C, C_{\XX_0,2}>0$. Thanks to the estimates on $\theta$ above, we can argue as in Step~1 of the proof of Lemma~\ref{lem:varpiC+g_tilde_omega_p} to deduce
$$
\la v \ra^{-\gamma+s} \varpi^{(\CC^+_g)^*}_{\tilde m,q} \le \la v \ra^{-\gamma+s}\varpi^{(\CC^+_g)^*}_{ m,q} + C \la v \ra^{-1},
$$
which together with Remark~\ref{rem:varpi_bar_omegaTER} imply that $\varpi^{(\CC^+_g)^*}_{\tilde m,q}$ also satisfies the estimate~\eqref{eq:lem:varpiC+g_tilde_omega_p} in Lemma~\ref{lem:varpiC+g_tilde_omega_p}, namely
$$
\la v \ra^{-\gamma-s} \varpi^{(\CC^+_g)^*}_{\tilde m,q} 
\le \kappa_{\omega, p} + C_{\XX_0,1} \| g \|_{\XX_0} + \psi,
$$
where $p$ is the conjugate exponent of $q$, that is $1/p+1/q=1$, and $\kappa_{\omega, p}$ is defined in \eqref{eq:lem:Psi_m-2}.

\medskip\noindent
\textit{Step 3.}
For the second term at the right-hand side of~\eqref{eq:ddt_h_Lq}, we observe that $\nabla_x (\widetilde m^q) = -\frac12 D_x n_x v \, \la v \ra^{\gamma-3} m_A^q $ and therefore
$$
\begin{aligned}
\int_\OO |h|^q v \cdot \nabla_x (\widetilde m^q) 
&= -\frac12 \int_\OO |h|^q (v \cdot D_x n_x  v) \la v \ra^{\gamma-3} \widetilde m^{q}  \\
&\le C_\TT \int_\OO |h|^q \widetilde m^{q-1}  \la v \ra^{\gamma-1}.
\end{aligned}
$$

\smallskip\noindent
\textit{Step 4.}
The boundary term in~\eqref{eq:ddt_h_Lq} can be decomposed as
\begin{equation}\label{eq:h_Lq_boundary}
\begin{aligned}
\int_{\Sigma} |\gamma h|^q \widetilde m^q (n_x \cdot v)
&= \int_{\Sigma}  |\gamma h|^q m_A^q  (n_x \cdot v )
-\frac12  \int_{\Sigma} |\gamma h|^q m_A^q (n_x \cdot v)^2 \la v \ra^{\gamma-3}.
\end{aligned}
\end{equation}

On the one hand, for the first term in \eqref{eq:h_Lq_boundary} we have
$$
\int_{\Sigma} |\gamma h|^q m_A^q  (n_x \cdot v) 
= \int_{\Sigma_+} |\gamma_+ h|^q m_A^q |n_x \cdot v| 
- \int_{\Sigma_-} |\gamma_- h|^q m_A^q |n_x \cdot v|.
$$
Using the boundary condition in \eqref{eq:dual_BBg} together with the fact that $s \mapsto |s|^q$ is convex, we get
$$
\begin{aligned}
\int_{\Sigma_+} |\gamma_+ h|^q m_A^q |n_x \cdot v|
&= \int_{\Sigma_+} \left| (1-\iota)\SSS \gamma_- h + \iota \DDD^* \gamma_+h \right|^q m_A^q |n_x \cdot v| \\
&\le \int_{\Sigma_+}  (1-\iota) |\SSS \gamma_- h|^q m_A^q |n_x \cdot v|
+ \int_{\Sigma_+} \iota  |\widetilde{\gamma_- h \MMM} |^q m_A^q |n_x \cdot v|.
\end{aligned}
$$
Making the change of variables $v \mapsto \VV_x v$ in the integral over $\Sigma_+$ yields
$$
\begin{aligned}
\int_{\Sigma_+} |\gamma_+ h|^q m_A^q |n_x \cdot v|
&\le \int_{\Sigma_-}  (1-\iota)|\gamma_- h|^q  m_A^q |n_x \cdot v|
+ \int_{\Sigma_-} \iota  |\widetilde{\gamma_- h \MMM}|^q m_A^q |n_x \cdot v|,
\end{aligned}
$$
and thus
$$
\int_{\Sigma} |\gamma h|^q m_A^q  (n_x \cdot v )
\le \int_{\Sigma_-} \iota  |\widetilde{\gamma_- h \MMM}|^q m_A^q |n_x \cdot v| - \int_{\Sigma_-}  \iota|\gamma_- h|^q  m_A^q |n_x \cdot v|.
$$
When $q=1$ we use that $m_A \ge \MMM$ to obtain
$$
\int_{\Sigma} |\gamma h| m_A (n_x \cdot v) 
\le \left[ K_1(m_A) - 1 \right] \int_{\partial\Omega} \iota  |\widetilde{\gamma_- h \MMM}|,
$$
where we denote
$$
K_1(m_A) = \int_{\R^3} m_A (n_x \cdot v)_- \, \d v < \infty.
$$
Otherwise when $q \ge 2$, by H\"older's inequality, we have
$$
|\widetilde{\gamma_- h \MMM}|^q (x) \le K_2(m_A) \int_{\Sigma_-^x} |\gamma_- h|^q m_A^q |n_x \cdot v|,
$$
with
$$
K_2(m_A) = \left( \int_{\R^3} \MMM^{\frac{q}{q-1}} m_A^{-\frac{q}{q-1}} (n_x \cdot v)_- \, \d v \right)^{(q-1)} < \infty.
$$
Denoting
$$
K_1(m_A) = \int_{\R^3} m_A^q (n_x \cdot v)_- \, \d v < \infty,
$$
we deduce
$$
\int_{\Sigma} |\gamma h|^q m_A^q (n_x \cdot v) 
\le \left[ K_1(m_A) - K_2(m_A)^{-1} \right] \int_{\partial\Omega} \iota  |\widetilde{\gamma_- h \MMM}|^q.
$$

On the other hand, for the second term in \eqref{eq:h_Lq_boundary}, using the boundary condition, we get
$$
\begin{aligned}
&\int_{\Sigma} |\gamma h|^q m_A^q (n_x \cdot v)^2 \la v \ra^{\gamma-3} \\
&\qquad
= \int_{\Sigma_-} |\gamma_- h|^q  m_A^q  (n_x \cdot  v)^2 \la v \ra^{\gamma-3}
+ \int_{\Sigma_+} |(1-\iota) \SSS \gamma_- h + \iota \DDD^* \gamma_- h|^q  m_A^q (n_x \cdot v)^2 \la v \ra^{\gamma-3} \\
&\qquad
= \int_{\Sigma_-} |\gamma_- h|^q m_A^q (n_x \cdot v)^2 \la v \ra^{\gamma-3} 
+ \int_{\Sigma_-} |(1-\iota) \gamma_- h + \iota \DDD^* \gamma_- h|^q m_A^q (n_x \cdot v)^2 \la v \ra^{\gamma-3} .
\end{aligned}
$$
If $q=1$ we write
$$
\begin{aligned}
\int_{\Sigma} |\gamma h| m_A (n_x \cdot v)^2 \la v \ra^{\gamma-3}
&\ge \int_{\Sigma_-} \iota | \widetilde{\gamma_- h \MMM} | m_A (n_x \cdot v)^2 \la v \ra^{\gamma-3} \\
& \ge K_0(m_A)  \int_{\partial\Omega} \iota  |\widetilde{\gamma_- h \MMM}|,
\end{aligned}
$$
with
$$
K_0(m_A) =  \int_{\R^3} m_A \la v \ra^{\gamma-3}  (n_x \cdot v)^2_-  \, \d v <\infty .
$$
In the case when $q \ge 2$, we use H\"older's inequality to write
$$
\begin{aligned}
\int_{\Sigma} |\gamma h|^q m_A^q (n_x \cdot v)^2 \la v \ra^{\gamma-3}
& \ge K_0(m_A)^{-1}  \int_{\partial\Omega} \iota  |\widetilde{\gamma_- h \MMM}|^q,
\end{aligned}
$$
where 
$$
K_0(m_A) = \left( \int_{\R^3} \MMM^{\frac{q}{q-1}} m_A^{-\frac{q}{q-1}} \la v \ra^{\frac{q-\gamma+3}{q-1}} (n_x \cdot  v)_-^{\frac{q-2}{q-1}} \, \d v \right)^{q-1} < \infty.
$$

With the convention $K_2(m_A) = 1$ when $q=1$, the boundary term~\eqref{eq:h_Lq_boundary} may finally be bounded in the following way
\begin{equation}\label{eq:boundary_estimate_Lq}
\int_{\Sigma} |\gamma h|^q  \widetilde{m}^q  (n_x \cdot v)
\le  \int_{\partial\Omega} \iota \left[K_1(m_A) - K_2(m_A)^{-1} - \frac12 K_0(m_A)^{-1}  \right] |\widetilde{\gamma_- h \MMM}|^q.
\end{equation}
Observing that $m_A \to \MMM^{\frac{1}{q}}$ when $A \to \infty$, we deduce that $K_0(m_A) \to K_0 (\MMM^{\frac{1}{q}}) >0$, $K_1(m_A) \to K_1(\MMM^{\frac{1}{q}}) = 1$ as well as $K_2(m_A) \to K_2(\MMM^{\frac{1}{q}}) = 1$ thanks to the normalization condition on $\MMM$. We may therefore choose $A>0$, large enough, such that
\begin{equation}\label{eq:choice_A_bis}
K_1(m_A) - K_2(m_A)^{-1} - \frac12 K_0(m_A)^{-1}  \le 0.
\end{equation}

\smallskip\noindent
\textit{Step 5.}
Coming back to \eqref{eq:ddt_h_Lq}, throwing away the last term thanks to Step~4 and gathering the estimates of Step~2 and Step~3, we obtain 
$$
\begin{aligned}
- \frac1q \frac{\d}{\d t} \| h \|_{L^q_{x,v} (\widetilde m)}^q + \frac{4(q-1)}{q^2} (\bar C - C_{\XX_0,2} \| g \|_{\XX_0}) \int_{\OO} \la v \ra^\gamma |\widetilde \nabla_v H|^2
&\leq  \int_\OO \left\{ \widetilde \varpi_{\BB^*_g} -M\chi_R \right\} |h|^q \widetilde m^q,
\end{aligned}
$$ 
with
$$
 \la v \ra^{-\gamma-s} \widetilde \varpi_{\BB^*_g}
:= \kappa_{\omega,p} + C_{\XX_0,1} \| g \|_{\XXX_0} + \psi +  C_\TT  \langle v \rangle^{-1-s}  .
$$
Arguing exactly as in Step~5 of the proof of Proposition~\ref{prop:LLg_L2}, we deduce that there are $\eps_3,M_3,R_3>0$ such that for all $\| g \|_{\XX_0} \le \eps_3$, any $M \ge M_3$ and $R \ge R_3$, there holds $(\bar C - C_{\XX_0,2} \| g \|_{\XX_0} )\ge \bar C/2$ and also $  \widetilde \varpi_g - M \chi_R  \le  \la v \ra^{\gamma+s} \kappa_{\omega,p} / 3$.
This concludes the proof with $\sigma >0$ as in  Step~5 of the proof of Proposition~\ref{prop:LLg_L2}.
\end{proof}

\subsection{Decay estimate for $S_{\BB_g}$}
 
We start with a first well-posedness result for the linear problem \eqref{eq:BBg} associated to $\BB_g$ which extends and improves the similar result Theorem~\ref{theo-SG-LLg} for 
the linear problem \eqref{eq:linear_g} associated to $\LL_g$.

\begin{prop}\label{prop-SG-BBg-L2Lp}
Consider an admissible weight function $\omega$ and a function $g \in \XX_0$ such that $\| g \|_{\XX_0} \le \eps_2$, where $\eps_2>0$ is given by Proposition~\ref{prop:BBg_L2}. 
There exists a non-autonomous semigroup $ S_{\BB_g}$ on $L^2(\omega)$ such that for any  $t_0 \ge 0$ and $f_{t_0} \in L^2_\omega(\OO)$, 
the function $f_t := S_{\BB_g}(t,t_0) f_{t_0}$  is the  unique solution in  $C([t_0,T];L^2_\omega) \cap  L^2((t_0,T) \times \Omega; H^{1,*}_{v}(\widetilde\omega))$, $\forall \, T > t_0$, to the equation \eqref{eq:BBg} associated  to the linear operator $\BB_g$ and to  the initial datum $f_{t_0}$.  
Furthermore, if $f_{t_0} \in L^p(\mu^{1/p-1})$ with $p \in [2,4]$, then $f_t \in L^p(\mu^{1-1/p})$ for any $t \ge t_0$.  
\end{prop}

\begin{proof}[Proof of Proposition~\ref{prop-SG-BBg-L2Lp}] %We fix $\omega_1 \ge \omega$ an admissible weight function such that $L^2_\omega \subset L^1_{\omega_1}$.

Repeating the proof of Theorem~\ref{theo-SG-LLg} and using the dissipative estimate for $\BB_g$ given by Proposition~\ref{prop:BBg_L2}, we obtain the existence and uniqueness of a solution in the $L^2(\omega)$ framework and then the existence of the associated semigroup $S_{\BB_g}$.

\smallskip
For dealing with the result in the $L^p(\mu^{1/p-1})$ framework, we use a very classical approximation argument. 
We assume that $f_0 \in L^p(\mu^{1/p-1}) \cap L^2(\omega)$ with $p \in { [2,4]}$, for some weight function $\omega$ such that $L^2(\omega) \subset L^2(\mu^{1/p-1})$, 
and we consider the associated solution $f \in  C([0,T],L^2(\omega))$ provided by the existence result in the $L^2(\omega)$ framework. For the sake of simplicity we only consider the case $p=4$ since 
it will be enough for our purpose and that anyway the case $p \in (2,4)$ can be easily deduced from that one. 
We fix  a function $\beta : \R \to \R_+$ convex and increasing linearly at the infinity. 
Setting $\phi := f/\MMM$ and using here and below the shorthands 
\bean
&&\widetilde a_{ij} :=  a_{ij} *(\mu + g), \quad 
\widetilde b_i  :=  b_i *(\mu + g), \quad 
\widetilde c  :=  c*  (\mu + g), 
\eean
we recall that  $f$ satisfies the PDE equation in \eqref{eq:BBg}  where $\BB_g$ is given by  
$$
\BB_g f = - v \cdot \nabla_x f + \widetilde a_{ij} \partial^2_{v_iv_j} f - \widetilde c f - M\chi_R f.
$$
We first observe that 
\bean
\int_{\R^3} (\widetilde a_{ij} \partial^2_{v_iv_j} f ) \beta'  (\phi)  
= -  \int_{\R^3}  \widetilde b_i   \beta'(\phi) \partial_{v_i} f 
-  \int_{\R^3}   \widetilde a_{ij} \partial_{v_i} f  \partial_{v_j} \beta'(\phi) =: T_1 + T_2,
 \eean
where we have performed one integration by part. For the first term, we have 
\bean
T_1 
&=&- \int_{\R^3}   \beta'(\phi) \phi   \widetilde b_i  \partial_{v_i}  \MMM   - \int_{\R^3} \beta'(\phi)  \partial_{v_i} \phi  \widetilde b_i   \MMM 
%\\
%&=&   \int \beta(\phi)   \partial_{v_i} (\MMM  \widetilde b_i ) - \int   \partial_{v_i} \MMM  \widetilde b_i  \phi  \beta'(\phi) 
\\
&=&   - \int_{\R^3}   \beta'(\phi)  \phi    \widetilde b_i  \partial_{v_i} \MMM  + \int_{\R^3} \beta(\phi)  [ \widetilde b_i \partial_{v_i} \MMM  + \widetilde c \MMM] , 
\eean
where we have used one integration by part again in the last line. In order to deal with the second term, we define $\psi := f \MMM^{1/p-1} = \phi \MMM^{1/p}$, and we directly compute
\bean
T_2
&=&  -   \int_{\R^3}  \beta''(\phi) \widetilde a_{ij}  \partial_{v_i} (\psi \MMM^{1-1/p}) \partial_{v_j} (\psi  \MMM^{-1/p})  
%\\
%&=&  -   \int  \beta''(\phi) \widetilde a_{ij}  \partial_{v_i} \psi  \partial_{v_j} \psi  \MMM^{1-2/p}   -   \int  \beta''(\phi) \widetilde a_{ij} ( \partial_{v_i} \psi )\MMM^{1-1/p}  \psi  \partial_{v_j} (\MMM^{-1/p})  
%\\
%&& -   \int  \beta''(\phi) \widetilde a_{ij}  \psi (\partial_{v_i} \MMM^{1-1/p}) ( \partial_{v_j} \psi ) \MMM^{-1/p} 
%  -   \int  \beta''(\phi) \widetilde a_{ij}   \psi^2   \partial_{v_i} (\MMM^{1-1/p}) \partial_{v_j} (   \MMM^{-1/p})  
%\\
%&=&  -   \int  \beta''(\phi) \widetilde a_{ij}  \partial_{v_i} \psi  \partial_{v_j} \psi  \MMM^{1-2/p}   
%  -   \int  \beta''(\phi) \widetilde a_{ij}  \partial_{v_i}( \phi \MMM^{1/p}) \MMM   \phi   \partial_{v_j} (\MMM^{-1/p})  
%\\
%&& -   \int  \beta''(\phi) \widetilde a_{ij}  \phi  (\partial_{v_i} \MMM^{1-1/p}) ( \partial_{v_j} (\phi \MMM^{1/p})) 
%  -   \int  \beta''(\phi) \widetilde a_{ij}   \phi^2 \MMM^{2/p}  \partial_{v_i} (\MMM^{1-1/p}) \partial_{v_j} (   \MMM^{-1/p})  
%\\
%&=&  -   \int  \beta''(\phi) \widetilde a_{ij}  \partial_{v_i} \psi  \partial_{v_j} \psi  \MMM^{1-2/p}   
%  -   \int  \beta''(\phi) \widetilde a_{ij} \phi  \partial_{v_i} \phi   [ \MMM^{1+1/p}   \partial_{v_j} \MMM^{-1/p}  +  ( \partial_{v_j}   \MMM^{1-1/p} ) \MMM^{1/p} ]
%\\
%&& -   \int  \beta''(\phi) \phi^2 \widetilde a_{ij} [ \MMM  \partial_{v_i}\MMM^{1/p}   \partial_{v_j} \MMM^{-1/p}
% +
% \partial_{v_i} \MMM^{1-1/p}  \partial_{v_j}  \MMM^{1/p} 
%  + \MMM^{2/p}  \partial_{v_i} \MMM^{1-1/p} \partial_{v_j}  \MMM^{-1/p} ] 
\\
&=&  -   \int_{\R^3}  \beta''(\phi) \widetilde a_{ij}  \partial_{v_i} \psi  \partial_{v_j} \psi  \MMM^{1-2/p}   
  -  \left(1 - \frac2p \right)  \int_{\R^3}  \widetilde a_{ij} \partial_{v_i} ( \phi \beta'(\phi) -  \beta(\phi))  \partial_{v_j} \MMM 
\\
&& + \frac1{p^2}   \int_{\R^3}  \beta''(\phi) \phi^2 \widetilde a_{ij}   \MMM^{-1}   \partial_{v_i}\MMM    \partial_{v_j} \MMM. 
\eean
Performing one integration by parts for dealing with the second term, we get 
\bean
T_2
&=&  -   \int_{\R^3}  \beta''(\phi) \widetilde a_{ij}  \partial_{v_i} \psi  \partial_{v_j} \psi  \MMM^{1-2/p}   
 + \frac1{p^2}   \int_{\R^3}  \beta''(\phi) \phi^2 \widetilde a_{ij}   \MMM^{-1}   \partial_{v_i}\MMM    \partial_{v_j} \MMM 
\\
&&+  \left(1 - \frac2p \right)  \int_{\R^3}  ( \phi \beta'(\phi) -  \beta(\phi)) (  \widetilde b_{i} \partial_{v_i}   \MMM +  \widetilde a_{ij}    \partial^2_{v_iv_j}\MMM ).
\eean
All together, we deduce that at least formally   
\bean
&&\frac{\d}{\d t} \int_{\OO} \beta (\phi) \MMM 
= -   \int_{\OO}  \beta''(\phi) \widetilde a_{ij}  \partial_{v_i} \psi  \partial_{v_j} \psi  \MMM^{1-2/p}   - \int_{\Sigma} \beta (\gamma \phi) \MMM (n_x \cdot v)
\\
&&\qquad+    \int_{\OO}    \phi \beta'(\phi)   (\widetilde \varpi  - M\chi_R) \MMM
%(  - \frac2p    \widetilde b_{i} \partial_{v_i}   \MMM +  \left(1 - \frac2p \right) \widetilde a_{ij}    \partial^2_{v_iv_j}\MMM  - ( \widetilde c  + M\chi_R) \MMM)
-   \int_{\OO}   \beta(\phi)  \widetilde\varpi   \MMM 
 + \frac1{p^2}   \int_{\OO}  \beta''(\phi) \phi^2 \widetilde a_{ij} v_iv_j \MMM , 
%%-   \int   \beta(\phi) (   - \frac2p   \widetilde b_{i} \partial_{v_i}   \MMM +  \left(1 - \frac2p \right) \widetilde a_{ij}    \partial^2_{v_iv_j}\MMM -  \widetilde c   \MMM ).
%\\
%&&
% %- \int    \phi  \beta'(\phi)   \widetilde b_i  \partial_{v_i} \MMM  
%% + \int \beta(\phi)  [ \widetilde b_i \partial_{v_i} \MMM  + \widetilde c \MMM]
\eean
with 
\bean
\widetilde \varpi 
&:=& - \frac2p    \widetilde b_{i} \frac{\partial_{v_i}   \MMM }{ \MMM}+  \left(1 - \frac2p \right) \widetilde a_{ij}    \frac{\partial^2_{v_iv_j}\MMM  }{ \MMM} - \widetilde c
\\
&=&   \frac2p    \widetilde b_{i}  v_i +  \left(1 - \frac2p \right) (\widetilde a_{ij}  v_i v_j - \widetilde a_{ii} ) - \widetilde c
\\
&=&   \langle v \rangle^{\gamma+2} \left(- \frac{4}{p} + \OO (\| g \|_{\XX_0}) \right) + \OO(  \langle v \rangle^{\gamma}) ,  
\eean
and 
\bean
\widetilde a_{ij} v_iv_j
&=&  2 \langle v \rangle^{\gamma+2} \left(1 + \OO (\| g \|_{\XX_0}) \right) + \OO(  \langle v \rangle^{\gamma}) ,  
\eean
where in the two last lines we have used Lemma~\ref{lem:elementary-abc*bar} and Lemma~\ref{lem:elementary-abc*g}.

Observing that $0 \le \beta(\phi) \MMM, \phi \beta'(\phi) \MMM \lesssim |f|$ so that each of the above integral term is well defined, that identity may be established rigorously from the Green formula \eqref{eq:FPK-traceL2} and a Stone-Weierstrass type argument. The boundary term is   nonpositive thanks to a Darroz\`es-Guiraud type inequality \cite{DarrozesG1966}. More precisely, we write
\bean
&&\int_\Sigma (- v \cdot n_x)  \beta( \frac{\gamma f}{\MMM}) \MMM
%= \int_{\Sigma_-} (v \cdot n_x)_-  \beta(\frac{\RRR \gamma_+ f }{ m}) \varphi 
%- \int_{\Sigma_+}(v \cdot n_x)_+ \beta (\frac{ \gamma_+ f }{ m}) \varphi 
\\
&&\quad \le \int_{\Sigma_-} (v \cdot n_x)_-   \left\{ \iota \beta( \frac{\DDD \gamma_+ f }{ \MMM})+ (1-\iota)  \beta(\frac{ \SSS \gamma_+ f }{ \MMM} ) \right\} \MMM
-  \int_{\Sigma_+} (v \cdot n_x)_+ \beta (\frac{ \gamma_+ f }{ \MMM}) \MMM
\\
&&\quad
= \int_{\Sigma_+} (v \cdot n_x)_+   \iota \left\{ \beta(    {\widetilde{\gamma_+ f}}   ) -   \beta (\frac{ \gamma_+ f }{ \MMM})   \right\} \MMM \le 0, 
\eean
where we have used the convexity of $\beta$ in the second line, the change of variable $v \mapsto R_x v$ in the next equality
and   the very definition of $ \widetilde{\gamma_+ f}$ as well as the Jensen inequality for the probability measure $\MMM (v \cdot n_x)_+ dv$ in order to get the last inequality. 
It is worth emphasizing again that because $\gamma f \in L^2(\Gamma; \d\xi^2_\omega \dt) \subset L^1(\Gamma; \d\xi^1 \dt)$ the above computation is licit.

We then take $p=4$ and $\beta = \beta_A$ the even function such that $\beta''_A(s) := 12 s^{p-2} \mathbf{1}_{s \le A}$ on $\R_+$, and next the primitives which vanish in the origin and which are thus defined by 
$\beta'(s) = 4 s^3 \mathbf{1}_{s \le A} + 4 A^3  \mathbf{1}_{s > A}$ and $\beta(s) = s^4 \mathbf{1}_{s \le A} + (4 A^3s - 3 A^4)  \mathbf{1}_{s > A}$. In particular, we verify that 
$0 \le 3 s^4 \mathbf{1}_{s \le A} + 3 A^4 \mathbf{1}_{s > A} = s \beta'(s) - \beta(s) \le 3 \beta(s)$ and  $\beta''(s) s^2 \le 12 \beta(s)$. 
We set 
$$
Z := \frac1{p^2}   \beta''(\phi) \phi^2 \widetilde a_{ij} v_iv_j  +  (\phi \beta'(\phi)  - \beta(\phi))  \widetilde \varpi. 
$$
For $|\phi| \le A$, we have 
\bean
Z &=& \left(  \frac34 \widetilde a_{ij} v_iv_j + 3  \widetilde \varpi  \right)  \phi^4
%\\
%&\le& \frac34 2 \langle v \rangle^{\gamma+2} \left(1 +  C \| g \|_{\XX_0} \right) +   \langle v \rangle^{\gamma+2} \left(- \frac{4}{p} + \eps + C \| g \|_{\XX_0}) \right) +C_\eps )
\\
&\le& \left[\langle v \rangle^{\gamma+2} \left(- \frac32 + \eps  +  C \| g \|_{\XX_0} \right)   +C_\eps \mathbf{1}_{B_R}   \right] \phi^4 \le C_\eps \beta(\phi)\mathbf{1}_{B_R},
\eean
for $\eps > 0$ and $\eps_0 > 0$ small enough. Similarly, for $|\phi| > A$, we have 
\bean
Z &=&(\phi \beta'(\phi)  - \beta(\phi))  \widetilde \varpi 
%\\
%&\le& \frac34 2 \langle v \rangle^{\gamma+2} \left(1 +  C \| g \|_{\XX_0} \right) +   \langle v \rangle^{\gamma+2} \left(- \frac{4}{p} + \eps + C \| g \|_{\XX_0}) \right) +C_\eps )
\\
&\le&  (\phi \beta'(\phi)  - \beta(\phi)) \left[\langle v \rangle^{\gamma+2} \left(- 1 + \eps  +  C \| g \|_{\XX_0} \right)   +C_\eps \mathbf{1}_{B_R} \right] 
\\
&\le&  3   \beta(\phi) C_\eps \mathbf{1}_{B_R} , 
\eean
for $\eps > 0$ and $\eps_0 > 0$ small enough.

\smallskip
Coming back to the above differential equation, we may through away the two first term at the RHS   and we may use the last bounds in order to get
\bean
 \frac{\d}{\dt} \int_\OO \beta_A(\phi) \MMM
\le   \int_\OO  \beta_A(\phi) \MMM (  3   C_\eps \mathbf{1}_{B_R} - M \chi_R )  \le 0.  
\eean
Using Gr\"onwall's lemma and next passing to the limit $A \to \infty$, we deduce that $\| f_t \|_{L^p(\mu^{1/p-1})} \le \|  f_0 \|_{L^p(\mu^{1/p-1})}$ for any $t \ge 0$. We extends the same result for any $f_0 \in L^p(\mu^{1/p-1})$ by a density argument. 
\end{proof}

We establish the counterpart of the previous result for the dual problem \eqref{eq:dual_BBg}.

\begin{prop}\label{prop-SG-BB*g-L2Lp}
Consider an admissible weight function $\omega$ and a function $g \in \XX_0$ such that $\| g \|_{\XX_0} \le \eps_3$, where $\eps_3>0$ is given in Proposition~\ref{prop:BB*g_Lq}. 
For any $h_T \in L^2(\omega^{-1})$, there exists a unique solution to  the dual problem \eqref{eq:dual_BBg} in an appropriate space that we make explicit during the proof.
Furthermore, if $h_{T} \in L^p(\omega^{-1/q})$ with $q \in { [2,4]}$, then $h_t \in L^q(\omega^{-1/q})$ for any $t \in [0,T)$. 
\end{prop}

\begin{proof}[Proof of Proposition~\ref{prop-SG-BB*g-L2Lp}] The proof is very similar to the proof of Proposition~\ref{prop-SG-BBg-L2Lp} and we thus just sketch it. 

\smallskip\noindent
 {\sl Step 1.} We define $m := \omega^{-1}$ and next $\widetilde m$ by \eqref{eq:def:m_A}--\eqref{eq:def:tilde_m}.  Because of the estimate established in Proposition~\ref{prop:BB*g_Lq} in the case $q=2$, we may use Theorem~\ref{theo-Kolmogorov-WellP} exactly as in the proof of Theorem~\ref{theo-SG-LLg} and we get that for any $h_T \in L^2(\widetilde m)$ there exists  $h \in C([0,T];L^2(\tilde m)) \cap L^2((0,T) \times \Omega; H^{1,*}_{v,\tilde m})$
unique solution to  the dual problem \eqref{eq:dual_BBg}.

\smallskip\noindent
 {\sl Step 2.} We proceed similarly and using the same notations as during the proof of   Proposition~\ref{prop:BB*g_Lq} and Proposition~\ref{prop-SG-BBg-L2Lp}. 
 Let us thus consider $h_T \in L^q \subset L^2(m)$ and the associated solution $h$ exhibited in the Step~1.
 We  fix  a function $\beta : \R \to \R_+$ convex, increasing linearly at the infinity and such that $s \beta'(s) \ge 0$. We compute 
\bean
-\frac{\d}{\dt} \int_\OO \beta(h) \widetilde m
&=&   \int_\Sigma (v \cdot n_x)  \beta(\gamma h) \widetilde m
  - \int_{\OO}   \beta'' (h)  \widetilde a_{ij} \partial_{v_i} h  \partial_{v_j} h \widetilde m
\\
&\quad&
+ \int_\OO \beta(h) (\widetilde a_{ij}  \partial^2_{v_iv_j}\widetilde m - \widetilde c \widetilde m - v \cdot \nabla_x \widetilde m) -  \int_\OO h \beta'(h) M \chi_R \widetilde m. 
\eean
Using that $\beta$ is convex and arguing as in Step~4 in the proof of  Proposition~\ref{prop:BB*g_Lq}, we have 
\bean
\int_\Sigma (v \cdot n_x)  \beta(\gamma h) \widetilde m 
\le  \int_{\partial\Omega}  \iota \left[ K_1(m_A) - 1- K_0(m_A) \right]  \beta(    {\widetilde{\MMM \gamma_- h }}  )
\le 0, 
\eean
for $A$ large enough. On the other hand, with the same notations as in Lemma~\ref{lem:Psi_m} and during its proof, we have 
\bean
\bar a_{ij} \partial^2_{v_iv_j} m
%&=& \bar a_{ij} \partial_{v_i} (- v_j \wp  m) 
%\\
&=& \bar a_{ij} \left[- \delta_{ij} \wp  + v_i v_j \wp^2 \left( 1 - \frac{s-2 }{ \langle v \rangle^2} \right) \right] m 
\\
&\underset{|v|\to \infty}{\sim}& \kappa_{\omega,1} \langle v \rangle^{\gamma+s} m, 
\eean
and because $\kappa_{\omega,1} < 0$, we may argue similarly as during the proof of Lemma~\ref{lem:varpiC+g_tilde_omega_p} and establish that 
\bean
\widetilde a_{ij}  \partial^2_{v_iv_j}\widetilde m - \widetilde c \widetilde m
\lesssim  \widetilde m, 
\eean
for $\| g \|_{\XX_0} \le \eps_3$, $\eps_3 > 0$ small enough. Coming back to the above differential equation, throwing away the two first and the last (all negative) terms and using the last estimate, we immediately obtain 
\bean
-\frac{\d}{\dt} \int_\OO \beta(h) \widetilde m
\lesssim  \int_\OO \beta(h)  \widetilde m, 
\eean
so that 
\bean
 \int_\OO \beta(h(0,\cdot)) \widetilde m \le C(T)  \int_\OO \beta(h_T)  \widetilde m, 
\eean
for a constant $C(T)$ independent of $\beta$. We conclude  in the same way as during the proof of  Proposition~\ref{prop-SG-BBg-L2Lp} by choosing the same appropriate sequence $\beta_R (s) \nearrow |s|^q$. 
\end{proof}

We are now in position for establishing the decay result for the semigroup $S_{\BB_g}$.  Recalling the definition of the decay function $\Theta_{\omega,\omega_\star}$ in Proposition~\ref{prop:GalWeakDissipEstim}, we shall hereafter abuse notation and write $\Theta_{\omega,\omega} (t) = e^{-\lambda t}$ for some $\lambda>0$ when $\omega$ is an admissible weight function verifying $\gamma+s \ge 0$.
 
\begin{prop}\label{prop:BBg_Lp}
Consider an admissible weight function $\omega$ and a function $g \in \XX_0$ such that $\| g \|_{\XX_0} \le \min(\eps_2,\eps_3)$, where $\eps_2, \eps_3>0$ are given in Propositions~\ref{prop:BBg_L2} and \ref{prop:BB*g_Lq} respectively.
For any $p \in [2,\infty]$, the semigroup $S_{\BB_g}$   exhibited in Proposition~\ref{prop-SG-BBg-L2Lp} extends to $L^p_\omega$ and more precisely  
\bear\label{eq-prop:BBg_Lp1}
&&\| S_{\BB_g}(t,\tau) f_\tau \|_{L^p_\omega} \lesssim  \| f_\tau \|_{L^p_\omega}
 \\ \label{eq-prop:BBg_Lp2}
 &&\| S_{\BB_g}(t,\tau) f_\tau \|_{L^p_{\omega_\star}} \le \Theta_{\omega,\omega_\star} (t-\tau)  \| f_\tau \|_{L^p_{\omega}}, 
 \eear
for any $t \ge \tau \ge 0$,  any $f_\tau \in L^p(\omega)$ and any admissible weight function $\omega_\star \preceq \omega$.
\end{prop}

\begin{proof}[Proof of Proposition~\ref{prop:BBg_Lp}] 
We shall prove that $S_{\BB^*_g}$ is the adjoint of $S_{\BB_g}$ and we next use the estimate established on $S_{\BB^*_g}$ and a duality argument. 
We set $m := \omega^{-1}$ and $q := p/(p-1) \in [1,2]$ the conjugate exponent associated to $p$. We split the proof into three steps.

\smallskip\noindent
{\sl Step 1.}
Consider $f_0 \in L^4(\mu^{-3/4})$ and $h_T \in L^4(\mu^{1/4})$ for $T > 0$.
We observe that if $f$ is the solution to the primal forward problem \eqref{eq:BBg} associated to $f_0$ given by Proposition~\ref{prop-SG-BBg-L2Lp}  and $h$ is a solution to the backward dual problem \eqref{eq:dual_BBg} associated to $h_T$ given by Proposition~\ref{prop-SG-BB*g-L2Lp}, we may apply Proposition~\ref{prop:Kolmogorov-duality} with the choice $\alpha(\sigma) = \beta(\sigma) = \sigma$, $\varphi \equiv 1$, and we get  
$$
\begin{aligned}
\int_\OO f(T) h_T  
&= \int_\OO f_0 h(0) - \int_0^T \!\! \int_\Sigma (n_x \cdot v) (\gamma f) (\gamma h) \, \d\xi^1  \d s. 
\end{aligned}
$$
The boundary term is well defined because on the one hand $f \in L^\infty(0,T; L^4(\mu^{-3/4}))$ from Proposition~\ref{prop-SG-BBg-L2Lp}  and thus $\gamma f  \in L^4(\d\xi^2_{\mu^{-3/4}}) \subset L^2(\d\xi^1_{\mu^{-1/2}})$ from Theorem~\ref{theo-Kolmogorov-trace} and the Cauchy-Schwarz inequality, and in the other hand $h \in L^\infty(0,T; L^4(\mu^{1/4}))$ from Proposition~\ref{prop-SG-BB*g-L2Lp}  and thus $\gamma h  \in L^4(\d\xi^2_{\mu^{1/4}}) \subset L^2(\d\xi^1_{\mu^{1/2}})$ from Theorem~\ref{theo-Kolmogorov-trace} and the Cauchy-Schwarz inequality. 
We next have 
\bean
\int_\Sigma (n_x \cdot v) (\gamma f) (\gamma h) \, \d\xi^1
&=& \int_{\Sigma_+} (n_x \cdot v)_+ (\gamma_+ f)  (\RRR^* \gamma_- h) \, \d\xi^1
\\
&\quad& + \int_{\Sigma_-} (n_x \cdot v)_- (\RRR \gamma_+ f) ( \gamma_- h) \, \d\xi^1= 0, 
\eean
where we have used the reflection conditions in \eqref{eq:BBg} and in  \eqref{eq:dual_BBg} in the first equality and the very definitions of the reflection operators $\RRR$ in \eqref{eq:reflection} and $\RRR^*$ in \eqref{eq:reflection*} in the second equality. 
We have thus established the duality identity
\beqn\label{eq:duality_identity}
\int_\OO f(T) h_T   
%= \int_\OO f(t) h(t) 
= \int_\OO f_0 h(0), 
\eeqn
and this one extends to any  $f_0 \in L^2(\omega)$ and  $h_T \in L^2(m)$ by a density argument.

\smallskip\noindent
{\sl Step 2.} We first emphasize that for $h_0 \in L^2(m)$, the differential inequality \eqref{eq:prop:BB*g_Lq} in particular implies
\beqn\label{eq:h_Lqm}
\| h(0)  \|_{L^q_{\tilde m}} \le \| h_T \|_{L^q_{\tilde m}}.
\eeqn
The computations in Proposition~\ref{prop:BB*g_Lq} can indeed be rigorously justified in the well-posedness framework introduced in Proposition~\ref{prop-SG-BB*g-L2Lp}.
We then write 
\bean
\| f(T) \|_{L^p_{\tilde m^{-1}}} 
&=& \sup_{h_T \in L^2_m;  \| h_T \|_{L^q_{\tilde m}} \le 1}   \int_\OO f(T) h_T  
\\
&=& \sup_{h_T \in L^2_m; \| h_T \|_{L^q_{\tilde m}} \le 1}   \int_\OO  f_0 h(0).
\\
&=& \sup_{h_T \in L^2_m; \| h_T \|_{L^q_{\tilde m}} \le 1} \| f_0 \|_{L^p_{\tilde m^{-1}}}  \|  h(0) \|_{L^q_{\tilde m}} \le   \| f_0 \|_{L^p_{\tilde m^{-1}}}  , 
\eean
where we have used a classical duality identity in the first line, the identity \eqref{eq:duality_identity} in the second line, 
the H\"older inequality and the estimate  \eqref{eq:h_Lqm} in the last line. 
Observing that $\omega \sim \widetilde m^{-1}$, we have established the first estimate \eqref{eq-prop:BBg_Lp1} for $\tau = 0$. The general case follows by time translation.

\smallskip\noindent
{\sl Step 3.} For $q=1,2$, the differential inequality \eqref{eq:prop:BB*g_Lq} also implies
 $$
- \frac1q \frac{\d}{\d t}  \|  h\|_{L^q  (\tilde m)}^q + \sigma \| h \|_{L^q(m \langle v \rangle^{(\gamma + s)/q})}^q \le 0. 
$$
If $\gamma+s <0$, using last estimate, the estimate \eqref{eq:h_Lqm} associated to a weight function $m_\star := \omega_\star^{-1}$ with $\omega_\star \prec \omega$, and Proposition~\ref{prop:GalWeakDissipEstim}, we deduce 
\begin{equation}\label{eq:h_Lqm1m2}
\| h(0)  \|_{L^q_{\tilde m}} \le \Theta_{m_\star,m}(T) \| h_T \|_{L^q_{\tilde m_\star}}.
\end{equation}
Otherwise if $\gamma+s \ge 0$, we immediately obtain
$$
\| h(0)  \|_{L^q_{\tilde m}} \le e^{-\lambda T} \| h_T \|_{L^q_{\tilde m}},
$$
that is \eqref{eq:h_Lqm1m2} with $m_\star = m$ and $\widetilde m_\star = \widetilde m$.

As in Step~2, we write 
\bean
\| f(T) \|_{L^p_{\tilde m_\star^{-1}}} 
&=& \sup_{h \in L^2_{m_\star}; \| h_T \|_{L^q_{\tilde m_\star}} \le 1}   \int_\OO  f_0 h(0).
\\
&\le&   \sup_{h \in L^2_{m_\star}; \| h_T \|_{L^q_{\tilde m_\star}} \le 1}    \| f_0 \|_{L^p_{\tilde m^{-1}}}  \|  h(0) \|_{L^q_{\tilde m}} 
\\
&\le&   \sup_{h \in L^2_{m_\star}; \| h_T \|_{L^q_{\tilde m_\star}} \le 1}     \| f_0 \|_{L^p_{\tilde m^{-1}}}  \Theta_{m_\star,m}(T) \| h_T \|_{L^q_{\tilde m_\star}}.
\\
&=&   \Theta_{\omega,\omega_\star}(T)     \| f_0 \|_{L^p_{\tilde m^{-1}}} , 
\eean
where we have used H\"older's inequality in the second line, the estimate   \eqref{eq:h_Lqm1m2} in the third line and the fact that 
$\Theta_{m_\star,m}(T) =  \Theta_{\omega,\omega_\star}(T)$ in the last line. That is nothing but \eqref{eq-prop:BBg_Lp2} for $p=2,\infty$ and $\tau = 0$. The general case for $\tau \ge 0$ and $p \in [2,\infty]$ follows by time translation and an interpolation argument. 
\end{proof}

%%%%%%%%%%%%%%%%%%%%%%%%%%%%%%%%%%%
\section{Ultracontractivity property of $S_{\BB_g}$}
\label{sec-SBBg-ultracontractivity}

%%%%%%%%
\subsection{De Giorgi-Nash-Moser type estimate}
\label{sec-SBBg-DeGiorgiNashMoser}

In this section we establish a De Giorgi-Nash-Moser type estimate of gain of integrability for solutions to equation \eqref{eq:dual_BBg} associated to $\BB_g^*$ in the spirit of \cite{MR3923847,MR2068847}.  This will be established in Theorem~\ref{theo:BB*g_L1_L2} below, as a consequence of a series of intermediate results. By a duality argument we shall finally obtain the ultracontractivity of $S_{\BB_g}$ in Theorem~\ref{theo:BBg_L2_Linfty}.
We start by our key estimate associated to the operator $\BB_g^*$.
 
\begin{prop}\label{prop:BB*g_ultracontractivity}
Consider an admissible weight function $\omega$ and define $m := \omega^{-1}$. There exist constants $\eps_4,M_4,R_4 > 0$ such that if  $\| g \|_{\XX_0} \le \eps_4$, then for any $T>0$, any solution $h$ to the linear equation \eqref{eq:dual_BBg} associated to $\BB^*_g$ on $(0,T)$ and any nonnegative test function $\varphi \in C^\infty_c ( (0,T) )$, there holds, for any $M \ge M_4$ and $R \ge R_4$,
\beqn\label{eq:prop:BB*g_ultracontractivity}
\int_0^T \!\!\int_{\OO} h^2 m^2 \la v \ra^{\gamma-3} \, \frac{ (n_x \cdot v)^2}{\delta^{1/2}} \, \varphi^2
+  \int_0^T \| h \|_{L^2_x H^{1,*}_v (m)}^2 \varphi^2 
\lesssim \int_0^T \| h \|_{L^2(m)}^2 \varphi (\varphi')_+ .
\eeqn

\end{prop}

\begin{rem}\label{rem:choice_MR}
Hereafter we fix contants $M \ge \max (M_1,M_2,M_3,M_4)$ and $R \ge \max(R_1,R_2,R_3,R_4)$ in the definition of the operators $\BB_g$ and $\AA_g$ in \eqref{def:BBg}--\eqref{def:AAg} in such a way that all previous results on $\BB_g$ and $\BB^*_g$ (Propositions~\ref{prop:BBg_L2}, \ref{prop:BB*g_Lq}, \ref{prop-SG-BBg-L2Lp}, \ref{prop-SG-BB*g-L2Lp}, \ref{prop:BBg_Lp}, and \ref{prop:BB*g_ultracontractivity}) are satisfied.
\end{rem}

\begin{proof}[Proof of Proposition~\ref{prop:BB*g_ultracontractivity}]
We define the modified weight function $m_A $ by \eqref{eq:def:m_A} with $q=2$ and then we define $\widetilde{m} = \widetilde{m}(x,v)$ by
$$
\widetilde m^2 = \left\{ 1 - \frac14 (n_x \cdot   v) \langle v \rangle^{\gamma-3}- \frac{\delta^{1/2}}{4 D^{1/2}} \, (n_x \cdot   v)   \langle v \rangle^{\gamma-3}  \right\} m_A^2, 
$$
where $D = \sup_{x \in \Omega} \delta$ is half the diameter of $\Omega$. We already observe that
$$
c_A^{-1}\MMM \le  m_A^2 \le c_A m^2 \quad\text{and}\quad
\frac12 m_A^2 \le \widetilde m^2 \le \frac32 m_A^2,
$$
for come constant $c_A>0$.
As in the proof of Proposition~\ref{prop:BB*g_Lq}, we remark that we can write
$$
\widetilde m^2 = \theta^2 m^2
$$
with
$$
\theta^2 = \left[1-\frac14 \left(1+ \frac{\delta^{1/2}}{D^{1/2}} \right) (n_x \cdot   v) \la v \ra^{\gamma-3} \right] \left[ \chi_A \MMM m^{-2} + 1-\chi_A \right] ,
$$
and we also have, for any $A>0$,
$$
\left| \frac{\partial_{v_i} \theta}{\theta} \right| \lesssim \la v \ra^{-2}, \quad \left| \frac{\partial_{v_i, v_j} \theta}{\theta}  \right| \lesssim \la v \ra^{-3}.
$$

\medskip\noindent
\textit{Step 1.}
We multiply \eqref{eq:dual_BBg} by $h \varphi^2 \widetilde{m}^2$ and we integrate in order to obtain
$$
\int_0^T \!\! \int_{\OO} \left\{ -\partial_t h - v \cdot \nabla_x h - (\CC^+_g)^* h- M \chi_R h \right\} h \varphi^2 \widetilde{m}^2 = 0, 
$$
or in other words
\begin{equation}\label{eq:h2m2}
\begin{aligned}
\frac12 \int_0^T \!\! \int_{\OO} h^2 \partial_t (\varphi^2) \widetilde{m}^2 
+\frac12 \int_0^T \!\! \int_{\OO} h^2  \varphi^2 v \cdot \nabla_x \widetilde{m}^2 
- \frac12 \int_0^T \!\! \int_{\Sigma} (\gamma h)^2  \varphi^2 \widetilde{m}^2 (n_x \cdot v)  \\
=\int_0^T \la   (\CC^+_g)^* h +  M \chi_R h  , h \ra_{L^2_{x,v} (\tilde{m})} \varphi^2. 
\end{aligned}
\end{equation}

\medskip\noindent
\textit{Step 2.} 
Arguing as in Step~2 of the proof of Proposition~\ref{prop:BB*g_Lq}, we have
$$
\begin{aligned}
\int_{\OO}   \left[(\CC^+_g)^*h  - M\chi_R h \right] h  \widetilde m^q
&\le - (\bar C - C_{\XX_0,2}\| g \|_{\XX_0}) \int_{\OO} \la v \ra^{\gamma}| \widetilde \nabla_v (\widetilde m h)|^2 \\
&\quad
+\int_{\OO} \left\{ \varpi^{(\CC^+_g)^*}_{\tilde m,q} - M \chi_R \right\} h^2 \widetilde m^2,
\end{aligned}
$$
where $\varpi^{(\CC^+_g)^*}_{\tilde m,q}$ satisfies the estimate~\eqref{eq:lem:varpiC+g_tilde_omega_p} in Lemma~\ref{lem:varpiC+g_tilde_omega_p}.

\medskip\noindent
\textit{Step 3.}
Observing that $\delta = 0$ on the boundary $\partial\Omega$, the boundary term in~\eqref{eq:h2m2} can be decomposed as
\begin{equation}\label{eq:boundary}
\begin{aligned}
-\frac12 \int_0^T \!\! \int_{\Sigma} (\gamma h)^2  \varphi^2 \widetilde{m}^2 (n_x \cdot v) 
&= -\frac12 \int_0^T \varphi^2 \int_{\Sigma}  (\gamma h)^2 m_A^2 (n_x \cdot v) \\
&\quad
+\frac18 \int_0^T \varphi^2 \int_{\Sigma} (\gamma h)^2 m_A^2 (n_x \cdot v)^2 \la v \ra^{\gamma-3}.
\end{aligned}
\end{equation}
Arguing as in Step~4 of the proof of Proposition~\ref{prop:BB*g_Lq}, we can choose $A>0$ large enough such that
$$
-\frac12 \int_0^T \!\! \int_{\Sigma} (\gamma h)^2  \varphi^2 \widetilde{m}^2 (n_x \cdot v) \ge 0.
$$

\medskip\noindent
\textit{Step 4.}
In order to deal with the second term at the left-hand side of~\eqref{eq:h2m2}, we define  $\psi := \delta^{1/2}  (n_x \cdot v) \la v \ra^{\gamma-3}$. Observing that $\la v \ra \psi \in L^\infty_{x,v}$, $\nabla_v \psi \in L^\infty_{x,v}$ and
$$
-v \cdot \nabla_x \psi = \frac{1}{2 \delta^{1/2}} (n_x \cdot v)^2 \la v \ra^{\gamma-3} - \delta^{1/2} (D_x n_x : v \otimes v) \la v \ra^{\gamma-3},
$$
we compute
$$
\begin{aligned}
v \cdot \nabla_x \widetilde{m}^2
&= \frac14 m_A^2 \la v \ra^{\gamma-3} \left\{  -(D_x n_x : v \otimes v)  + \frac{1}{2 D^{1/2} \delta^{1/2}} (n_x \cdot v)^2   - \frac{\delta^{1/2}}{D^{1/2}} (D_x n_x : v \otimes v) \right\} .
\end{aligned}
$$
Therefore we deduce
$$
\begin{aligned}
\frac12 \int_0^T \!\! \int_{\OO} h^2  \varphi^2 v \cdot \nabla_x \widetilde{m}^2
&\ge \frac{1}{16 D^{1/2}} \int_0^T \!\! \int_{\OO}  h^2  \varphi^2 m_A^2 \la v \ra^{\gamma-3} \frac{(n_x \cdot  v)^2}{\delta^{1/2}} 
- C_1 \int_0^T \!\! \int_{\OO}   h^2  \varphi^2 m_A^2 \la v \ra^{\gamma-1} ,
\end{aligned}
$$
fo some constant $C_1>0$.

\medskip\noindent
\textit{Step 5.}
Gathering previous estimates, it follows
$$
\begin{aligned}
\frac{1}{16 D^{1/2}} \int_0^T \!\! \int_{\OO}  h^2  \varphi^2 m_A^2 \la v \ra^{\gamma-3} \frac{(n_x \cdot  v)^2}{\delta^{1/2}} 
&+(\bar C - C_{\XX_0,2}\| g \|_{\XX_0}) \int_0^T \!\! \int_{\OO} \la v \ra^{\gamma}| \widetilde \nabla_v (\widetilde m h)|^2 \varphi^2\\
&+ \int_0^T \!\! \int_{\OO}  \left\{ M\chi_R - \widetilde \varpi_{\BB^*_g}  \right\} h^2 \widetilde m^2 \varphi^2
\le \int_0^T \!\! \int_{\OO} h^2 \widetilde m^2 \varphi (\varphi')_+
\end{aligned}
$$
with, using the notation of Lemma~\ref{lem:varpiC+g_tilde_omega_p},
$$
 \la v \ra^{-\gamma-s} \widetilde \varpi_{\BB^*_g}
:= \kappa_{\omega,2} + C_{\XX_0,1} \| g \|_{\XXX_0} + \psi +  (C_\TT+C_1)  \langle v \rangle^{-1-s}. 
$$
We can then conclude by arguing as in Step~5 of the proof of Proposition~\ref{prop:BB*g_Lq}. More precisely, we deduce that there are $\eps_4,M_4,R_4>0$ such that for all $\| g \|_{\XX_0} \le \eps_4$, any $M \ge M_4$ and $R \ge R_4$, there holds $(\bar C - C_{\XX_0,2} \| g \|_{\XX_0} )\ge \bar C/2$ and $ M\chi_R - \widetilde \varpi_{\BB^*_g} \ge  \la v \ra^{\gamma+s} |\kappa_{\omega,2}| / 3$, which completes the proof.
\end{proof}

We state and prove an elementary interpolation result which will be useful in the sequel.

\begin{lem}\label{lem:EstimL2poids}
For any function $f : \OO \to \R$, there holds 
\beqn\label{eq:EstimL2poids}
\| \delta^{-1/4} \la v \ra^{-1} f \|_{L^2(\OO)}^2 \lesssim 
\int_{\OO} f^2 \la v \ra^{-2}    \frac{(n_x \cdot v)^{2}}{\delta^{1/2}} + 
\| \nabla_v f \|_{L^2(\OO)}^2.
\eeqn
\end{lem}

\begin{proof}[Proof of Lemma~\ref{lem:EstimL2poids}] 
For $\zeta > 0$, we start by writing
\bean
\int_{\OO} f^2 \la v \ra^{-2} \delta^{-1/8}
&=& 
\int_{\OO}  f^2 \la v \ra^{-2}\delta^{-1/8}   {\mathbf 1}_{(n_x \cdot v)^{2} > \delta^{2\zeta}} 
+ \int_{\OO}  f^2 \la v \ra^{-2} \delta^{-1/8}  {\mathbf 1}_{|n_x \cdot v|\le \delta^{\zeta}}  \\
&=:& T_1 + T_2.
\eean
For the first term, we have 
\bean
T_1 
\le
\int_{\OO} f^2 \la v \ra^{-2}  {\mathbf 1}_{(n_x \cdot v)^{2} > \delta^{2\zeta}}  \frac{(n_x \cdot v)^{2} }{  \delta^{2\zeta+1/8}} 
\le 
\int_{\OO}  f^2  \la v \ra^{-2}  \frac{(n_x \cdot  v)^{2}}{\delta^{1/2}} ,
\eean
by choosing $\zeta = 3/16$.
For the second term, we compute
\bean
T_2
&\le&  
\int_\Omega  \delta^{-1/8}  \left( \int_{\R^3} f^{6}  \right)^{1/3}  \left( \int_{\R^3} \langle v \rangle^{- 3}    {\mathbf 1}_{ |n_x \cdot v | \le \delta^{\zeta}}   \right)^{2/3}
\\
&\lesssim&  
  \int_{\Omega} \delta^{-1/8+2\zeta/3}    \int_{\R^3} |\nabla_v f|^2,
\eean
where we have used the H\"older inequality in the first line and the Sobolev inequality in the second line together with the observation that $\langle v \rangle^{-3}   \in L^\infty(\R;L^{1}(\R^{2}))$. 
We conclude to \eqref{eq:EstimL2poids}.
\end{proof}

We reformulate Proposition~\ref{prop:BB*g_ultracontractivity} in a more convenient way, where the penalization of the neighborhood of the boundary is made clear.

\begin{prop}\label{prop:BB*g_ultracontractivity_bis}
Under the same setting as in Proposition~\ref{prop:BB*g_ultracontractivity} there holds
$$
\begin{aligned}
\| \delta^{-1/4} \la v \ra^{\frac{\gamma-3}{2}} m h \varphi \|_{L^2(\UU)} 
%+ \| h \varphi \|_{L^2( (0,T) ; L^2_x H^{1,*}_v (m))}
\lesssim  \| m h  \sqrt{\varphi (\varphi')_+ } \|_{L^2(\UU)} ,
\end{aligned}
$$
where we recall $\UU = (0,T) \times \OO$.
\end{prop}

\begin{proof}[Proof of Proposition~\ref{prop:BB*g_ultracontractivity_bis}]
Observing that 
$$
\| \nabla_v (h m \la v \ra^{\frac{\gamma-1}{2}}) \|_{L^2_v} 
\lesssim \| \la v \ra^{\frac{\gamma-1}{2}} \nabla_v (h m) \|_{L^2_v}  + \| h m \la v \ra^{\frac{\gamma-3}{2}} \|_{L^2_v} 
\lesssim \| h \|_{H^{1,*}_v(m)}, 
$$
the estimate is a direct consequence of Proposition~\ref{prop:BB*g_ultracontractivity} and Lemma~\ref{lem:EstimL2poids}. 
\end{proof}

On the other hand, we may establish a penalized gain of integrability as a simple consequence of available results known to hold on the whole space \cite{MR3923847}.

\begin{prop}\label{prop:BB*g_EstimLploc}
Under the same setting as in Proposition~\ref{prop:BB*g_ultracontractivity}, for any $p \in (2, 7/3)$ and any $\alpha > p$ there holds
$$
\begin{aligned}
\left\| \delta^{\alpha/p} \la v \ra^{-\frac{5}{2}} m h \varphi \right\|_{L^p(\UU)} \lesssim \left( T^{\frac{7}{p} -3} + T^{\frac{7}{p} - \frac52} \right)  \left( \| m h \sqrt{\varphi (\varphi')_+}  \|_{L^2(\UU)} + \| m h \varphi' \|_{L^2(\UU)} \right) ,
\end{aligned}
$$
where we recall $\UU = (0,T) \times \OO$.
\end{prop}

\begin{proof}[Proof of Proposition~\ref{prop:BB*g_EstimLploc}]
We split the proof into four steps.

\medskip\noindent
\textit{Step 1.}
Let $m_0 = \la v \ra^{-\frac{5}{2}}  m$ and $\zeta \in C^\infty_c (\Omega)$, with $0 \le \zeta \le 1$, and define $\bar h = h \varphi \zeta  m_0$. From \eqref{eq:dual_BBg} and using the shorthands $\widetilde a = a * [ \mu+g]$, $\widetilde  b = b *[ \mu+g]$, and $\widetilde c = c * [ \mu+g]$, 
we see that $\bar h$ satisfies 
\begin{equation}\label{eq:barh0}
-\partial_t  \bar h  - v \cdot \nabla_x   \bar h  = 
- m_0 h \left( \partial_t + v \cdot \nabla_x \right) (\varphi \zeta)
+ \varphi \zeta m_0 \left\{ \partial_{v_i} (\widetilde a_{ij}  \partial_{v_j} h )  +  \widetilde b_i \partial_{v_i} h - M \chi_R h \right\}.
\end{equation}
Observing that 
$$
\begin{aligned} 
\partial_{v_i} \left[\widetilde a_{ij}  \partial_{v_j} (hm_0) \right] 
&= m_0  \partial_{v_i} (\widetilde a_{ij}  \partial_{v_j} h ) +  h \widetilde b_i  \partial_{v_i} m_0  + 2  \widetilde a_{ij}  \partial_{v_j} h \partial_{v_i} m_0  + h \widetilde a_{ij}  \partial_{ij} m_0 \\
&= m_0 \partial_{v_i} (\widetilde a_{ij}  \partial_{v_j} h ) +  m_0 h \widetilde b_i  \frac{\partial_{v_i} m_0}{m_0}  + 2  \widetilde a_{ij}   m_0 \partial_{v_j} h \frac{\partial_{v_i} m_0}{m_0}  + m_0 h \widetilde a_{ij}  \frac{\partial_{ij} m_0}{m_0},
\end{aligned}
$$
we use that $m_0 \partial_{v_i} h = \partial_{v_i} (h m_0) - m_0 h \frac{\partial_{v_i} m_0}{m_0} $ to obtain
$$
\begin{aligned} 
\partial_{v_i} \left[\widetilde a_{ij}  \partial_{v_j} (hm_0) \right] 
&= m_0 \partial_{v_i} (\widetilde a_{ij}  \partial_{v_j} h ) 
+  m_0 h \widetilde b_i  \frac{\partial_{v_i} m_0}{m_0}  
+ 2  \widetilde a_{ij}  \partial_{v_i} (h m_0) \frac{\partial_{v_j} m_0}{m_0}\\
&\quad
- 2 m_0 h  \widetilde a_{ij} \frac{\partial_{v_j} m_0}{m_0} \frac{\partial_{v_i} m_0}{m_0}
+ m_0 h \widetilde a_{ij}  \frac{\partial_{ij} m_0}{m_0}.
\end{aligned}
$$
This implies
$$
\begin{aligned}
\varphi \zeta m_0  \partial_{v_i} (\widetilde a_{ij}  \partial_{v_j} h ) 
&= \partial_{v_i} (\widetilde a_{ij}  \partial_{v_j} \bar h ) 
-  \widetilde b_i  \frac{\partial_{v_i} m_0}{m_0} \bar h
-2  \widetilde a_{ij}\frac{\partial_{v_j} m_0}{m_0}  \partial_{v_i} \bar h \\
&\quad
+2   \widetilde a_{ij} \frac{\partial_{v_j} m_0}{m_0} \frac{\partial_{v_i} m_0}{m_0} \bar h
-  \widetilde a_{ij}  \frac{\partial_{ij} m_0}{m_0} \bar h
\end{aligned}
$$
and
$$
\begin{aligned}
\varphi \zeta m_0 \widetilde b_i \partial_{v_i} h
&= \varphi \zeta \widetilde b_i \partial_{v_i} (h m_0)
- \varphi \zeta m_0 h \widetilde b_i \frac{\partial_{v_i} m_0}{m_0} \\
&=  \widetilde b_i \partial_{v_i} \bar h
-  \widetilde b_i \frac{\partial_{v_i} m_0}{m_0} \bar h.
\end{aligned}
$$

Coming back to \eqref{eq:barh0}, we hence obtain that $\bar h$ is a solution to
\begin{equation}\label{eq:barh}
-\partial_t \bar h - v \cdot \nabla_x \bar h - \Delta_v \bar h = \Div_v S_1 + S_0 \quad \text{in} \quad (0,\infty) \times \R^3_x \times \R^3_v
\end{equation}
where
$$
S_{1,i} =  \left( \widetilde a_{ij} - \delta_{ij} \right) \partial_{v_j} \bar h 
$$
and
$$
\begin{aligned}
S_0
&= \left(-2 \widetilde a_{ij}  \frac{\partial_{v_j} m_0}{m_0} + \widetilde b_i   \right) \partial_{v_i} \bar h \\
&\quad
+ \left( - \widetilde a_{ij}  \frac{\partial_{v_i,v_j} m_0}{m_0}
+ 2 \widetilde a_{ij}  \frac{\partial_{v_i} m_0}{m_0} \frac{\partial_{v_j} m_0}{m_0} - 2 \widetilde b_i \frac{\partial_{v_i} m_0}{m_0} - M \chi_R  \right) \bar h \\
&\quad
- m_0 h \left( \partial_t + v \cdot \nabla_x \right) (\varphi \zeta).
\end{aligned}
$$

\medskip\noindent
\textit{Step 2.}
We now claim that $S_0,S_1 \in L^2_{t,x,v}$ with
\begin{equation}\label{eq:bound_S0_S1}
\begin{aligned}
&\| S_0 \|_{L^2_{t,x,v} ([0,T] \times \R^3_x \times \R^3_v)} + \| S_1 \|_{L^2_{t,x,v} ([0,T] \times \R^3_x \times \R^3_v)} \\
&\qquad
\lesssim  \| \zeta \|_{W^{1,\infty}_x} \left( \| m h \sqrt{\varphi (\varphi')_+}  \|_{L^2_{t,x,v}} + \| m h \varphi' \|_{L^2_{t,x,v}} \right).
\end{aligned}
\end{equation}
Indeed,recalling that $m_0 = \la v \ra^{-\frac{5}{2}} m$, on the one hand we have
$$
\begin{aligned}
|S_{1,i}| 
&\lesssim  \la v \ra^{(\gamma+2)_+} (1+ \| g \|_{\XX_0}) |\nabla_v ( \la v \ra^{-\frac{5}{2}} m h)| \varphi \| \zeta \|_{L^{\infty}_x} \\
&\lesssim (1+ \| g \|_{\XX_0}) \left[ \la v \ra^{(\gamma+2)_+ -\frac{5}{2}}  |\nabla_v (mh)| +  \la v \ra^{(\gamma+2)_+ -\frac{5}{2} -1}  |m h| \right] \varphi \| \zeta \|_{L^{\infty}_x} \\
&\lesssim (1+ \| g \|_{\XX_0}) \left[ \la v \ra^{\frac{\gamma}{2}} |\nabla_v (mh) | + \la v \ra^{\frac{\gamma}{2}+\frac{s}{2}} |mh| \right] \varphi \| \zeta \|_{L^{\infty}_x} ,
\end{aligned}
$$
and thus
$$
\begin{aligned}
\| S_1 \|_{L^2_{t,x,v}} 
&\lesssim  (1+ \| g \|_{\XX_0}) \| h \varphi \|_{L^2_t L^2_x H^{1,*}_v(m)}  \| \zeta \|_{L^{\infty}_x} .
\end{aligned}
$$
On the other hand, thanks to Lemma~\ref{lem:elementary-abc*g}, we have
$$
\left| \widetilde a_{ij}  \frac{\partial_{v_j} m_0}{m_0} \right| 
+\left|(a_{ij} *[\mu + g]) \frac{\partial_{v_i,v_j} m_0}{m_0} \right|
+\left| (a_{ij} *[\mu + g]) \frac{\partial_{v_i} m_0}{m_0}  \frac{\partial_{v_j} m_0}{m_0} \right|
\lesssim \la v \ra^{\gamma +s}  (1+ \| g \|_{\XX_0}),
$$
as well as
$$
\left| \widetilde b_i \frac{\partial_{v_i} m_0}{m_0} \right| 
\lesssim \la v \ra^{\gamma +s}  (1+ \| g \|_{\XX_0}), \quad
| \widetilde b_i | 
\lesssim \la v \ra^{\gamma+1}(1+ \| g \|_{\XX_0}),
$$
which implies 
$$
\begin{aligned}
|S_{0}| 
&\lesssim (1+ \| g \|_{\XX_0} ) \left[ \left( \la v \ra^{\gamma+s} + \la v \ra^{\gamma+1} \right)  |\nabla_v ( \la v \ra^{-\frac{5}{2}} m h)| + \la v \ra^{\gamma+s}  | \la v \ra^{-\frac{5}{2}} m h| \right] \varphi \| \zeta \|_{L^{\infty}_x}  \\
&\quad
+ \la v \ra^{-\frac{5}{2}} | m h| |\varphi'|  \| \zeta \|_{L^{\infty}_x}
+ \la v \ra^{1-\frac{5}{2}} | m h| \varphi  \| \nabla_x \zeta \|_{L^{\infty}_x}  \\
&\lesssim (1+ \| g \|_{\XX_0} ) \left[ \left( \la v \ra^{\gamma+s-\frac{5}{2}} + \la v \ra^{\gamma+1-\frac{5}{2}} \right)  |\nabla_v ( m h)|
+ \la v \ra^{\gamma+s-\frac{5}{2}}  |m h| \right] \varphi \| \zeta \|_{L^{\infty}_x}  \\
&\quad
+ \la v \ra^{-\frac{5}{2}} | m h| |\varphi'|  \| \zeta \|_{L^{\infty}_x}
+ \la v \ra^{-\frac{3}{2}} | m h| \varphi  \| \nabla_x \zeta \|_{L^{\infty}_x}  \\
&\lesssim (1+ \| g \|_{\XX_0}) \left[ \la v \ra^{\frac{\gamma}{2}} |\nabla_v (mh)| + \la v \ra^{\frac{\gamma}{2}+\frac{s}{2}}  |mh| \right]  
+  |mh| |\varphi'|  \| \zeta \|_{L^{\infty}_x}
+ \la v \ra^{\frac{\gamma}{2}+\frac{s}{2}}  |m h| \varphi  \| \nabla_x \zeta \|_{L^{\infty}_x}.
\end{aligned}
$$
We therefore deduce
$$
\begin{aligned}
\| S_0 \|_{L^2_{t,x,v}} 
&\lesssim  (1+ \| g \|_{\XX_0}) \| h \varphi \|_{L^2_t L^2_x H^{1,*}_v(m)}  \| \zeta \|_{W^{1,\infty}_x} 
+\| m h \varphi' \|_{L^2_{t,x,v}}   \| \zeta \|_{L^{\infty}_x},
\end{aligned}
$$
from which we obtain \eqref{eq:bound_S0_S1} by using Proposition~\ref{prop:BB*g_ultracontractivity} to estimate the term $\| h \varphi \|_{L^2_t L^2_x H^{1,*}_v(m)} $.

\medskip\noindent
\textit{Step 3.} 
We observe that from \eqref{eq:barh}, the function $H$ defined by $H(t,x,v) = \bar h(-t,x,-v)$ satisfies the Kolmogorov equation with source term
\begin{equation}\label{eq:H}
\begin{aligned}
\partial_t H  +  v \cdot \nabla_x H - \Delta_v H = -\Div_v R_1 + R_0, \quad \text{in} \quad (-\infty,0) \times \R^3_x \times \R^3_v  ,
\end{aligned}
\end{equation}
with $R_1(t,x,v) = S_1(-t,x,-v)$ and $R_0(t,x,v)=S_0(-t,x,-v)$.  In particular $\| \bar h \|_{L^q ([0,T] \times \R^3_x \times \R^3_v)} = \| H \|_{L^q ([-T,0] \times \R^3_x \times \R^3_v)}$ for any $q \in [1,\infty]$ and any $T>0$.

We recall that the fundamental solution of the Kolmogorov equation is given by (see for instance \cite{MR1503147})
\begin{equation}\label{eq:G}
G(t,x,v) = \frac{c_0}{t^6} \, \exp \left( -\frac{c_1}{t^3} |x - \tfrac{t}{2} v |^2 - \frac{c_2}{t} |v|^2  \right) \quad \text{if} \quad t >0
\end{equation}
for some constants $c_0,c_1,c_2>0$ and $G(t,x,v) = 0$ if $t \le 0$, and it satisfies the bound
\begin{equation}\label{eq:nablavG}
|\nabla_v G (t,x,v)| \lesssim \frac{\bar c_0}{t^{6+1/2}} \, \exp \left( -\frac{\bar c_1}{t^3} |x - \tfrac{t}{2} v |^2 - \frac{\bar c_2}{t} |v|^2  \right)
\end{equation}
for constants $\bar c_0, \bar c_1, \bar c_2>0$. Therefore the solution $H$ of \eqref{eq:H} is given by, for any $(t,x,v) \in (-\infty,0) \times \R^3_x \times \R^3_v$,
$$
\begin{aligned}
H(t,x,v) 
&= \int G(t-t', x-x' - (t-t') v' , v-v') \left[\Div_v R_1(t',x,',v') + R_0 (t',x',v') \right] \d t' \, \d x' \, \d v' \\
&= -\int \nabla_v G(t-t', x-x' - (t-t') v' , v-v') R_1(t',x,',v') \, \d t' \, \d x' \, \d v' \\
&\quad
+ \int G(t-t', x-x' - (t-t') v' , v-v')  R_0 (t',x',v') \, \d t' \, \d x' \, \d v',
\end{aligned}
$$
where we have performed an integration by parts. For any $r \ge 1$, we have from \eqref{eq:G}
$$
\| G \|_{L^r ([0,T] \times \R^3_x \times \R^3_v)} \lesssim T^{\frac{7}{r} - 6} ,
$$
as well as
$$
\| \nabla_v G \|_{L^r ([0,T] \times \R^3_x \times \R^3_v)} \lesssim T^{\frac{7}{r} - 6 - \frac12}
$$
from the estimate \eqref{eq:nablavG}.
Applying Young's inequality to the above representation formula for $H$ gives
$$
\begin{aligned}
\|  H \|_{L^p ([-T,0] \times \R^3_x \times \R^3_v)}
&\lesssim \| \nabla_v G \|_{L^{\frac{2p}{p+2}} ([0,T] \times \R^3_x \times \R^3_v)} \| R_1 \|_{L^2 ([-T,0] \times \R^3_x \times \R^3_v)} \\
&\quad
+ \| G \|_{L^{\frac{2p}{p+2}}  ([0,T] \times \R^3_x \times \R^3_v)} \| R_0 \|_{L^2 ([-T,0] \times \R^3_x \times \R^3_v)} \\
&\lesssim T^{\frac{7}{p} - 3}\| R_1 \|_{L^2 ([-T,0] \times \R^3_x \times \R^3_v)} + T^{\frac{7}{p} - \frac52}
\| R_0 \|_{L^2 ([-T,0] \times \R^3_x \times \R^3_v)},
\end{aligned}
$$
because $2 < p < 7/3$.

Coming back to $\bar h =  m_0 h \varphi \zeta$ and using that $\| R_i \|_{L^2([-T,0] \times \R^3_x \times \R^3_v)} = \| S_i \|_{L^2([0,T] \times \R^3_x \times \R^3_v)}$ together with the bounds of Step~2, we deduce
\begin{equation}\label{eq:barh_L2_Lp}
\| m_0 h \varphi \zeta \|_{L^p ([0,T] \times \R^3_x \times \R^3_v)}
\lesssim C_T \| \zeta \|_{W^{1,\infty}_x} \left( \| m h \sqrt{\varphi (\varphi')_+}  \|_{L^2_{t,x,v}} + \| m h \varphi' \|_{L^2_{t,x,v}} \right)\end{equation}
with $C_T = T^{\frac{7}{p} - 3} + T^{\frac{7}{p} - \frac52}$.

\medskip\noindent
\textit{Step 4.}
We define $\Omega_k = \{ x \in \Omega \mid \delta(x) > 2^{-k} \}$ and choose $\zeta_k \in C^\infty_c (\Omega)$ such that $\mathbf 1_{\Omega_{k+1}} \le \zeta_k \le \mathbf 1_{\Omega_k}$ and $ \| \zeta_k \|_{W^{1,\infty}_x} \lesssim 2^k$ for all $k \in \N^*$. 

Denoting $\UU_{k} = (0,T) \times \Omega_k \times \R^3_v$, we deduce from \eqref{eq:barh_L2_Lp} that 
$$
\| m_0 h \varphi \|_{L^p(\UU_{k+1})}  \lesssim  2^{k} C_T \left( \| m h \sqrt{\varphi (\varphi')_+}  \|_{L^2(\UU)} + \| m h \varphi' \|_{L^2(\UU)} \right), \quad \forall \, k \ge 1.
$$
Summing up and observing that $\alpha>p$, we obtain 
$$
\begin{aligned}
\int_{\UU} \delta^{\alpha}  |m_0 h \varphi|^{p} 
&= \sum_{k=1}^\infty \int_{\UU_{k+1} \backslash \UU_{k}} \delta^{\alpha}  |m_0 h \varphi|^p \\
&\lesssim \sum_{k=1}^\infty  2^{-k \alpha} \int_{\UU_{k+1}}  |m_0 h \varphi|^p \\
&\lesssim  \sum_{k=1}^\infty 2^{ k(p-  \alpha)} C_T^p  \left( \| m h \sqrt{\varphi (\varphi')_+} \|_{L^2(\UU)} + \| m h \varphi' \|_{L^2(\UU)} \right)^p \\
&\lesssim C_T^p  \left( \| m h \sqrt{\varphi (\varphi')_+}  \|_{L^2(\UU)} + \| m h \varphi' \|_{L^2(\UU)} \right)^p, 
\end{aligned}
$$
which completes the proof.
\end{proof}
 
As a consequence of the above bounds, we establish now the following key estimate of the De Giorgi-Nash-Moser theory.

\begin{cor}\label{cor:L2_Lr}
Let $p \in (2, 7/3)$, $\alpha > p$ and consider the same setting as in  Proposition~\ref{prop:BB*g_ultracontractivity}. Then there holds
\begin{equation}\label{eq:estimate_Lp_L2}
\begin{aligned}
\|m_r h \varphi \|_{L^r(\UU)}
  \lesssim T^{\theta \left(\frac{7}{p} - 3 \right)} \left( \| m h \sqrt{\varphi (\varphi')_+} \|_{L^2(\UU)} + \| m h \varphi' \|_{L^2(\UU)}  \right),
\end{aligned}
\end{equation}
for any $T \in (0,1)$, with
\begin{equation}\label{eq:mr}
m_r := \la v \ra^{ -\frac{(3-\gamma)}{2} - \left(\frac{\gamma}{2}+1 \right) \theta } m, \quad \theta := \frac{p}{p+4\alpha} \in (0, \tfrac15) , \quad r := \frac{p+4\alpha}{1+2\alpha} \in (2,p).
\end{equation}
\end{cor}

\begin{proof}[Proof of Corollary~\ref{cor:L2_Lr}]
By interpolation we have
$$
\| m_r h \varphi \|_{L^r(\UU)} 
\le \| \delta^{-1/4} \la v \ra^{\frac{\gamma-3}{2}} m h \varphi \|_{L^2(\UU)}^{1-\theta} \| \delta^{\alpha/p} \la v \ra^{-\frac{5}{2}  } m h \varphi \|_{L^p(\UU)}^\theta,
$$
with 
$$
\frac{1}{r} = \frac{1-\theta}{2}+\frac{\theta}{p}
$$
and
$$
\begin{aligned}
m_r 
&= \left( \delta^{-1/4} \la v \ra^{\frac{\gamma-3}{2}} m\right)^{1-\theta} \left( \delta^{\alpha/p} \la v \ra^{-\frac{5}{2} } m \right)^{\theta} \\
&= \delta^{-\frac14  + \left(\frac14+\frac{\alpha}{p} \right) \theta} \la v \ra^{ -\frac{(3-\gamma)}{2} - \left(\frac{\gamma}{2}+1 \right) \theta }  m.
\end{aligned}
$$
We choose $\theta = \frac{p}{p+4\alpha}$ so that $-\frac14  + \left(\frac14+\frac{\alpha}{p} \right) \theta=0$, which implies $r = \frac{p+4\alpha}{1+2\alpha}$.
We conclude to estimate~\eqref{eq:estimate_Lp_L2} by applying Propositions~\ref{prop:BB*g_ultracontractivity_bis} and~\ref{prop:BB*g_EstimLploc} and by using Young's inequality associated to the exponent $1/\theta$ and its conjugated exponent. 
\end{proof}

\subsection{Proof of the ultracontractivity property}

From the material developed in the previous sections, we first deduce a gain of integrability for solutions to the linear equation \eqref{eq:dual_BBg} associated to $\BB^*_g$. For simplicity, and because it is enough for our purposes, we shall only consider exponential admissible weight functions.
We recall that $\eps_2 , \eps_3, \eps_4>0$ are given by Propositions~\ref{prop:BBg_L2}, \ref{prop:BB*g_Lq}, and \ref{prop:BB*g_ultracontractivity}, respectively.

\begin{theo}\label{theo:BB*g_L1_L2} 
Consider two admissible exponential weight functions $\omega$ and $\omega_1$ such that $\omega_1 \prec \omega$, and define $m := \omega^{-1}$ and  $m_1 := \omega_1^{-1}$. If $\| g \|_{\XX_0} \le \min(\eps_2, \eps_3,\eps_4)$, then for any $T \in (0,1)$, any $h_T \in L^1_{x,v}(m_1)$ and any solution $h$ to the linear equation \eqref{eq:dual_BBg} associated to $\BB^*_g$, there holds, for any $0 \le t < T$,
\beqn\label{eq:theo:BB*g_L1_L2}
\| h(t) \|_{L^2_m} \lesssim {(T-t)^{-\vartheta}} \, \| h_T \|_{L^1_{m_1}}, 
\eeqn
for some $\vartheta \in (0,\infty)$.
\end{theo}

\begin{proof}[Proof of Theorem~\ref{theo:BB*g_L1_L2}]
For simplicity we only consider $t=0$, the general case being similar.
Let $p \in (2,7/3)$, $\alpha >p$, and define $m_r$, $\theta$ and $r$ by \eqref{eq:mr} in Corollary~\ref{cor:L2_Lr}. 
Let $\beta := \frac{r}{2(r-1)} \in (0,1)$ and define the function
$$
m_0 := \la v \ra^{\left[\frac{(3-\gamma)}{2}  + \left(\frac{\gamma}{2}+1 \right) \theta \right]\frac{r}{r-2}  } m \lesssim m_1
$$
in such a way that $m = m_0^{1-\beta} m_r^\beta \lesssim m_1^{1-\beta} m_r^\beta $.
Applying H\"older's inequality, we obtain
\bean
\| m h \varphi' \|_{L^2(\UU)}
\lesssim
\| (\varphi'/\varphi)^{2q} \varphi  m_0 h \|_{L^1(\UU)}^{1-\beta}
\| \varphi m_r h \|_{L^r(\UU)}^{\beta}, 
\eean
where $q := \frac{r-1}{r-2}$, and similarly
\bean
\| m h \sqrt{\varphi'_+\varphi} \|_{L^2(\UU)}
\lesssim 
\| (\varphi'/\varphi)^{q} \varphi  m_0 h \|_{L^1(\UU)}^{1-\beta}
\| \varphi m_r h \|_{L^r(\UU)}^{\beta}. 
\eean
Adding these two estimates, using \eqref{eq:estimate_Lp_L2} from Corollary~\ref{cor:L2_Lr} and then simplifying yields
\begin{equation}\label{eq:theo:BB*g_L1_L2-step1}
\begin{aligned} 
\| m h \varphi' \|_{L^2(\UU)}
&+ \| m h \sqrt{\varphi'_+\varphi} \|_{L^2(\UU)} \\
&\lesssim T^{\frac{\theta \beta}{1-\beta} \left( \frac{7}{p} - 3\right)} \left( \| (\varphi'/\varphi)^{2q} \varphi  m_0 h \|_{L^1(\UU)} + \| (\varphi'/\varphi)^{q} \varphi  m_0 h \|_{L^1(\UU)} \right) \\
&\lesssim T^{\frac{\theta \beta}{1-\beta} \left( \frac{7}{p} - 3\right)} \left( \| (\varphi'/\varphi)^{2q} \varphi  m_1 h \|_{L^1(\UU)} + \| (\varphi'/\varphi)^{q} \varphi  m_1 h \|_{L^1(\UU)} \right).
\end{aligned}
\end{equation}

For a nonconstant nonnegative function $\varphi_0 \in C^1([0,1])$, to be specified below, and $T \in (0,1)$, we set $\varphi (t) := \varphi_0(t/T) $. 
Writing
$$
T^{-1/2} \| \varphi'_0 \|_{L^2_t(0,1)} \| h(0) \|_{L^2(m)}
= \left( \int_0^T \varphi'(t)^2 \, \d t  \, \| h(0) \|_{L^2( m)}^2  \right)^{1/2},
$$ 
we then compute
\bean
&&\left( \int_0^T \varphi'(t)^2 \, \d t  \, \| h(0) \|_{L^2( m)}^2  \right)^{1/2} \\
&&\qquad\lesssim  \left( \int_0^T  \varphi(t)^2 \| h(t) \|_{L^2(m)}^2 \, \d t  \right)^{1/r}
\\
&&\qquad \lesssim  T^{\frac{\theta \beta}{1-\beta} \left( \frac{7}{p} - 3\right)} \left( \int_0^T (\varphi'/\varphi)^{q} \varphi  \| h(t) \|_{L^1(m_1)} \, \d t   +\int_0^T  (\varphi'/\varphi)^{2q} \varphi  \| h(t) \|_{L^1(m_1)} \, \d t  \right)
\\
&&\qquad \lesssim  T^{\frac{\theta \beta}{1-\beta} \left( \frac{7}{p} - 3\right)} \left( \int_0^T (\varphi'/\varphi)^{q} \varphi  \, \d t  +\int_0^T  (\varphi'/\varphi)^{2q} \varphi \, \d t  \right)  \| h_T \|_{L^1(m_1)}
\\
&&\qquad= T^{\frac{\theta \beta}{1-\beta} \left( \frac{7}{p} - 3\right)} \left( T^{1-q} \int_0^1 (\varphi_0'/\varphi_0)^{q} \varphi_0  \, \d \tau  + T^{1-2q} \int_0^1  (\varphi_0'/\varphi_0)^{2q} \varphi_0\, \d \tau  \right)  \| h_T \|_{L^1(m_1)},
\eean 
where we have used Proposition~\ref{prop:BB*g_Lq} with $q:=2$ in the second line, estimate~\eqref{eq:theo:BB*g_L1_L2-step1} in the third line, and Proposition~\ref{prop:BB*g_Lq} with $q:=1$ in the fourth one. In other words, we have established 
$$
\| h(0) \|_{L^2(m)} \lesssim T^{-\vartheta} \| h_T \|_{L^1(m_1)}, \quad \forall \, T \in (0,1), 
$$
with
$$
\vartheta := 2q-\frac{3}{2} - \frac{\theta \beta}{1-\beta} \left( \frac{7}{p} - 3\right)
= \frac{r}{2(r-2)} \left[ 1 - 2 \theta \left(\frac7p - 3 \right) \right] + \frac{1}{(r-2)} >0, 
$$
provided that $\varphi_0$ is such that $A_q < \infty$ and $A_{2q} < \infty$ with 
$$
A_\alpha :=  \int_0^1 (|\varphi_0'|/\varphi_0)^{\alpha} \varphi_0  \, \d \tau. 
$$
These last conditions are for instance fulfilled by $\varphi_0(\tau) := \tau^k (1-\tau)^k$ when $k \ge 2q$.
\end{proof}

We finally formulate the ultracontractivity property in terms of the semigroup $S_{\BB_g}$, which will be obtained as a direct consequence of \eqref{eq:theo:BB*g_L1_L2} and a duality argument.
 
\begin{theo}\label{theo:BBg_L2_Linfty}
Consider some exponential admissible weight functions $\omega, \omega_{\star}$, $\omega_{\star,1}$ such that  $\omega_{\star,1} \prec \omega_\star \preceq \omega$.
If $g \in \XX_0$ is such that  $\| g \|_{\XX_0} \le \min(\eps_2,\eps_3,\eps_4)$, then the non-autonomous semigroup $S_{\BB_g}$ satisfies the ultracontractivity estimate
\beqn\label{eq:ThSBg-L2Lp}
\| S_{\BB_g}(t,\tau) \|_{\BBB(L^2(\omega), L^\infty (\omega_{\star,1}))} \lesssim \frac{\Theta_{\omega, \omega_\star} (t-\tau)}{ \min( (t-\tau)^\vartheta,1)}, \quad \forall \, t > \tau \ge  0,
\eeqn
with $\nu>0$ given by Theorem~\ref{theo:BB*g_L1_L2} and where we take $\omega_\star = \omega$ if $\gamma+ s \ge 0$, so that $\Theta_{\omega,\omega}$ is exponential ; and $\omega_\star \prec \omega$ if $\gamma+s<0$, so that the $\Theta_{\omega,\omega_\star}$ is given by Proposition~\ref{prop:GalWeakDissipEstim}.
\end{theo}

\begin{proof}[Proof of Theorem~\ref{theo:BBg_L2_Linfty}] 
Let $0 \le \tau < t$ and define $m_\star := \omega_\star^{-1}$ and $m_{\star,1} := \omega_{\star,1}^{-1}$.
Let $f_\tau \in L^2(\omega)$ and consider the solution $f$ to the primal forward problem \eqref{eq:BBg} associated to $\BB_g$ such that $f(\tau) = f_\tau$. 

If $0 < t-\tau \le 1$, for any $h_t \in L^1 (m_{\star,1})$, we consider the solution $h$ to the dual backward problem \eqref{eq:dual_BBg} associated to $\BB_g^*$ on the interval $(\tau,t)$ and to the final datum $h_t$. 
We then deduce
$$
\begin{aligned}
\| f(t) \|_{L^\infty(\omega_{\star,1})}
&= \sup_{ \| h_t \|_{L^1 (m_{\star,1})} \le 1} \int_{\OO}  f(t) h_t \\
&= \sup_{ \| h_t \|_{L^1(m_{\star,1})} \le 1} \int_{\OO} f_\tau h(\tau) \\
&\le \sup_{ \| h_t \|_{L^1(m_{\star,1})} \le 1} \| f_s \|_{L^2 (\omega_\star)} \| h(\tau) \|_{L^2(m_{\star})} \\
&\lesssim (t-\tau)^{-\vartheta} \| f_s \|_{L^2 (\omega_\star)} \sup_{ \| h_t \|_{L^1(m_1)} \le 1} \| h_t \|_{L^1(m_{\star,1})} ,
\end{aligned}
$$
where we have used the duality identity~\eqref{eq:duality_identity} at the second line, H\"older's inequality in the third line, and estimate~ \eqref{eq:theo:BB*g_L1_L2} of Theorem~\ref{theo:BB*g_L1_L2} in the last one. From this estimate, it follows
\beqn\label{eq:ThSBg-L2Lp-smallt}
\| S_{\BB_g}(t,\tau) \|_{\BBB(L^2(\omega_\star),L^\infty(\omega_{\star,1}))} \lesssim (t-\tau)^{-\vartheta}, \quad \forall \, 0 < t-\tau \le 1,
\eeqn
which gives \eqref{eq:ThSBg-L2Lp} for $0<t-\tau \le 1$ since $\omega_\star \preceq \omega$.

Otherwise, when  $t-\tau>1,$ we write $f(t) = S_{\BB_g}(t,\tau) f_\tau = S_{\BB_g}(t,t-1) S_{\BB_g}(t-1,\tau) f_\tau$, so that
$$
\begin{aligned}
\| f(t) \|_{L^\infty(\omega_{\star,1})}
&= \| S_{\BB_g}(t,t-1) S_{\BB_g}(t-1,\tau) f_\tau \|_{L^\infty (\omega_{\star,1})} \\
&\lesssim \| S_{\BB_g}(t-1,\tau) f_\tau \|_{L^2 (\omega_\star)} \\
&\lesssim \Theta_{\omega,\omega_\star}(t-\tau-1) \| f_\tau \|_{L^2(\omega)},
\end{aligned}
$$
where we have used \eqref{eq:ThSBg-L2Lp-smallt} in the second line and Proposition~\ref{prop:BBg_Lp} in the third one. The proof is then complete by observing that $\Theta_{\omega,\omega_\star}(t-\tau-1) \lesssim \Theta_{\omega,\omega_\star}(t-\tau)$.
\end{proof}

%%%%%%%%%%%%%%%%%%%%%%%%%%%%%%%%%%%%%%%%%%%%%
\section{Hypocoercivity property of  $\LL_g$}\label{sec-LLg-hypo}

In this section we establish the $L^2$ hypocoercivity property as announced in Step~(3) of Section~\ref{subsec:strategy}
and the straightforward consequence in a semigroup formulation.

\begin{theo}\label{theo:hypo}
There exists an inner product $\la \! \la \cdot , \cdot \ra \!\ra_{L^2_{x,v} (\mu^{-1/2})}$ on $L^2_{x,v} (\mu^{-1/2})$ such that the associated norm $\Nt \cdot \Nt_{L^2_{x,v} (\mu^{-1/2})}$ is equivalent to the usual norm $\| \cdot \|_{L^2_{x,v} (\mu^{-1/2})}$ and for which the linear operator $\LL_g$ satisfies the following coercive estimate. There is $\eps_5>0$ small enough and  some constants $\lambda, \sigma >0$ such that that for any $g \in \XX_0$ 
with $\| g \|_{\XX_0} \le \eps_5$, there holds
\begin{equation}\label{eq:theo:hypo}
\la \! \la \LL_g f , f \ra \!\ra_{L^2_{x,v} (\mu^{-1/2})} \le - \lambda \Nt  \la v \ra^{\frac{\gamma}{2}+1} f \Nt_{L^2_{x,v} (\mu^{-1/2})}^2 - \sigma \| f \|_{L^2_x H^{1,*}_v (\mu^{-1/2})}^2,
\end{equation}
for any $f \in \mathrm{Dom}(\LL_g)$ satisfying the boundary condition and the mass condition $ \la \! \la f \ra\!\ra = 0$ 
(and the additional condition $ \la \! \la f |v|^2\ra\!\ra =  \la \! \la f R\cdot v \ra\!\ra = 0$ for any $R \in \RR_\Omega$ in the pure specular case $\iota\equiv0$).
\end{theo}

\begin{proof}[Proof of Theorem~\ref{theo:hypo}]
We denote by $u[S]=  u \in H^1(\Omega)$ the solution to the Poisson equation 
\begin{equation}\label{eq:elliptic}
\left\{
\begin{aligned}
- \Delta_x u &= S \quad\text{in}\quad \Omega n  \\
(2-\iota(x))\nabla_x u \cdot n_x + \iota(x) u &= 0 \quad\text{on}\quad \partial \Omega,   
\end{aligned}
\right.
\end{equation}
for a scalar source term $S : \Omega \to \R$. 
Remark that  \eqref{eq:elliptic} corresponds to the Poisson equation with homogeneous Neumann boundary condition when $\iota \equiv 0$, and we denote by $u_N[S]$ the corresponding solution in that case. 
Otherwise,~\eqref{eq:elliptic} corresponds to the Poisson equation with homogeneous Robin (or mixed) boundary condition.   We recall (see for instance \cite[Section~2.1]{MR4581432})
that defining $V_\iota := H^1(\Omega)$ if $\iota \not\equiv 0$ and $V_\iota := \{ u \in H^1(\Omega), \ \langle u \rangle = 0 \}$ if $\iota \equiv 0$, 
%for any $S \in L^2(\Omega)$ there exists a unique $u \in V_\iota$ solution to \eqref{eq:elliptic} in the variational sense,  
for any $S \in L^2(\Omega)$, with the additional assumption $\langle S \rangle = 0$ when $\iota\equiv0$,  there exists a unique $u \in V_\iota$ solution to 
\eqref{eq:elliptic} in the variational sense and this one satisfies 
\beqn\label{eq:EstimH2Poisson}
\| u \|_{H^2(\Omega)} \lesssim \| S \|_{L^2(\Omega)}. 
\eeqn

\smallskip
We similarly denote by $U[S]=U \in H^1(\Omega)$ the solution to the elliptic Lamé-type system  
\begin{equation}\label{eq:elliptic-korn}
\left\{
\begin{aligned}
- \Div_x (\nabla^s U)  = S \quad&\text{in}\quad \Omega , \\
U \cdot n_x = 0 \quad&\text{on}\quad \partial \Omega, \\  
(2-\iota) \left[\nabla^s_x U  n_x - (\nabla^s_x U : n_x \otimes n_x) n_x \right] + \iota (x) U = 0 \quad&\text{on}\quad \partial \Omega,
\end{aligned}
\right.\qquad
\end{equation}
for a vector-field source term  $S : \Omega \to \R^3$ and where $\nabla^s U$ stands for the symmetric gradient defined through
$
(\nabla^{s}_x  U)_{ij}  :=  ( \partial_{x_j}{U_i} + \partial_{x_i} U_j )/2$.   
We also define the skew-symmetric gradient of $U$ by $(\nabla^{a}_x U)_{ij} := \frac{1}{2} \left( \partial_{x_j} U_i - \partial_{x_i} U_j \right)$, next the functional spaces
$$
\VV_\iota := \left\{ U : \Omega \to \R^d \mid U \in H^1(\Omega), \; U \cdot n_x = 0 \text{ on } \partial \Omega  \right\},
$$
 if $\iota \not\equiv 0$, and 
 $$
\VV_\iota :=  \left\{ U : \Omega \to \R^d \mid W \in H^1(\Omega), \; U \cdot n_x = 0 \text{ on } \partial \Omega , \; P_\Omega \la \nabla^a U \ra = 0 \right\}, 
$$
 if $\iota \equiv 0$, where $P_\Omega$ denotes the orthogonal projection onto the set $\AA_\Omega = \{ A \in \MM^a_{3 \times 3}(\R); \;  Ax \in \RR_\Omega \}$ of all skew-symmetric matrices giving rise to a centered infinitesimal rigid displacement field preserving $\Omega$ (see \eqref{eq:RROmega} for the definition of $\RR_\Omega$).
From \cite[Theorem~2.11]{MR4581432},  we know that for any $S \in L^2(\Omega)$, with the additional assumption  $\la S , Ax \ra = 0$ for any $Ax \in \RR_\Omega$  when $\iota\equiv0$, 
    there exists a unique $U \in \VV_\iota$ solution to  \eqref{eq:elliptic-korn} in the variational sense, 
and this one satisfies 
\beqn\label{eq:EstimH2LamE}
\| U \|_{H^2(\Omega)} \lesssim \| S \|_{L^2(\Omega)}. 
\eeqn

\smallskip

 We also define the mass, momentum and energy of a function $f : \OO \to \R$  respectively by
$$
\varrho[f](x) = \int_{\R^3} f (x,v) \, dv , \quad 
j[f](x) = \int_{\R^3} v f (x,v) \, dv  
$$
and
$$
\theta[f](x) = \int_{\R^3} \frac{(|v|^2 - 3)}{\sqrt{6}} \, f (x,v) \, dv .
$$

As in \cite{MR4581432}, we define the inner product $\la \! \la \cdot , \cdot \ra \!\ra_{L^2_{x,v} (\mu^{-1/2})}$ in the following way:
$$
\begin{aligned}
\la \! \la f , g \ra \!\ra_{L^2_{x,v} (\mu^{-1/2})}
&= \la f,g \ra_{L^2_{x,v} (\mu^{-1/2})} \\
&\quad + \eta_1 \la -\nabla_x u[\theta[f]] ,  M_p [g]  \ra_{L^2_x(\Omega)}
+ \eta_1\la -\nabla_x u[\theta [g]],  M_p [f]  \ra_{L^2_x(\Omega)} \\
&\quad
+ \eta_2 \la -\nabla_x^s U[j[f]] ,  M_q [g]  \ra_{L^2_x(\Omega)}
+ \eta_2 \la -\nabla_x^s U[j[g]] ,  M_q [f]  \ra_{L^2_x(\Omega)} \\
&\quad
+ \eta_3 \la -\nabla_x u_{\mathrm{N}}[\varrho[f]] ,  m [g]  \ra_{L^2_x(\Omega)}
+ \eta_3 \la -\nabla_x u_{\mathrm{N}}[\varrho[g]] ,  j[f]  \ra_{L^2_x(\Omega)},
\end{aligned}
$$
with contants $0 \ll \eta_3 \ll \eta_2 \ll \eta_1 \ll 1$, where thus $u[\theta[f]]$ is the solution of the Poisson equation~\eqref{eq:elliptic} with source term $\theta[f]$; $U[j[f]]$ is the solution to the Lam\'e system \eqref{eq:elliptic-korn} with source term $j[f]$; $u_{\mathrm{N}}[\varrho[f]]$ is the solution to the Poisson equation~\eqref{eq:elliptic}  with homogeneous Neumann boundary condition  with source term $\varrho[f]$, and similarly for the terms depending on $g$ ; and where the moments $M_p$ and $M_q$ are given by
$$
M_p[h] = \frac{1}{\sqrt{6}}\int_{\R^3} v(|v|^2-5) h \, \d v
$$
and
$$
M_q [h] = \int_{\R^3} (v \otimes v - I) h \, \d v. 
$$
We already observe that 
$$
\| f \|_{L^2_{x,v} (\mu^{-1/2})} \lesssim \Nt f \Nt_{L^2_{x,v} (\mu^{-1/2})} \lesssim \| f \|_{L^2_{x,v} (\mu^{-1/2})}.
$$

\smallskip
Summarizing results from \cite{MR1463805,BM,GuoLandau1,Mouhot-coerc,MS}, (see also   \cite[(2.6)]{MR3625186})  we have 
$$
\la \CC f, f \ra_{L^2_{v} (\mu^{-1/2})} \le - \lambda_m \| (I-\pi) f \|_{H^{1,*}_v(\mu^{-1/2})}  , 
$$
for any $f \in H^{1,*}_v(\mu^{-1/2})$ and for some microscopic coercivity constant $\lambda_m > 0$.
Using next the arguments leading to \cite[Theorem 4.1]{MR4581432}, we  know that we can choose $\eta_i$ such that 
\beqn\label{eq:llaLLffrra}
\la \! \la \LL f , f \ra \!\ra_{L^2_{x,v} (\mu^{-1/2})} \le - \lambda \Nt  \la v \ra^{\frac{\gamma}{2}+1}  f \Nt_{L^2_{x,v} (\mu^{-1/2})}^2
 -  \sigma_0  \| \la v \ra^{\frac{\gamma}{2}} \widetilde \nabla_v f \|_{L^2_v (\mu^{-1/2})}^2
% \| f \|_{L^2_x H^{1,*}_v (\mu^{-1/2})}^2
\eeqn
for some constants $\lambda, \sigma_0 >0$.

\smallskip
We are now in position to  estimate  the term $\la \! \la \LL_g f , f \ra \!\ra_{L^2_{x,v} (\mu^{-1/2})}$. Observing that 
$$
\varrho [Q^\perp(g,f)] = j [Q^\perp(g,f)] = \theta[Q^\perp(g,f)] = 0, 
$$
we have
$$
\begin{aligned}
\la \! \la \LL_g f , f \ra \!\ra_{L^2_{x,v} (\mu^{-1/2})}
&= \la \! \la \LL f , f \ra \!\ra_{L^2_{x,v} (\mu^{-1/2})} + \la Q^\perp(g,f) , f \ra_{L^2_{x,v} (\mu^{-1/2})} 
\\
&\quad 
% + \eta_1 \la -\nabla_x u[\theta[Q(g,f)]] ,  M_p [f]  \ra_{L^2_x(\Omega)}
+ \eta_1\la -\nabla_x u[\theta [f]],  M_p [Q^\perp(g,f)]  \ra_{L^2_x(\Omega)} \\
&\quad
%+ \eta_2 \la -\nabla_x^s U[m[Q(g,f)]] ,  M_q [f]  \ra_{L^2_x(\Omega)}
+ \eta_2 \la -\nabla_x^s U[j[f]] ,  M_q [Q^\perp(g,f)]  \ra_{L^2_x(\Omega)}.
%&\qquad + \eta_3 \la -\nabla_x u_{\mathrm{N}}[\varrho[f]] ,  m[Q(g,f)]  \ra_{L^2_x(\Omega)}. 
\end{aligned}
$$
The first term is bounded by \eqref{eq:llaLLffrra}. 
For the second term in the right-hand side, we use \eqref{eq:Qperpgf_omega} to obtain
$$
\begin{aligned}
\la Q^\perp(g,f) , f \ra_{L^2_{x,v} (\mu^{-1/2})} 
&\lesssim \int_\Omega \| g \|_{L^\infty_{\omega_0}} \| f \|_{H^{1,*}_v(\mu^{-1/2})}^2 \, \d x \\
&\lesssim \| g \|_{\XX_0} \| f \|_{L^2_x H^{1,*}_v(\mu^{-1/2})}^2.
\end{aligned}
$$
We next compute
$$
\begin{aligned}
\la -\nabla_x u[\theta [f]],  M_p [Q^\perp(g,f)]  \ra_{L^2_x(\Omega)}
&\lesssim \| \nabla_x u[\theta [f]] \|_{L^2_x} \| M_p [Q^\perp(g,f)] \|_{L^2_x} \\
&\lesssim \| \theta [f] \|_{L^2_x L^2_v} \| g \|_{L^\infty_x L^\infty_{\omega_0}} \| f \|_{L^2_x H^{1,*}_v(\mu^{-1/2})} \\
&\lesssim  \| g \|_{\XX_0} \| f \|_{L^2_x H^{1,*}_v(\mu^{-1/2})}^2, 
\end{aligned}
$$
where we have used the Cauchy-Schwarz inequality in the first line, the estimates \eqref{eq:EstimH2Poisson} and \eqref{eq:Qgf-moments} in the second line, and the Cauchy-Schwarz inequality again in the last line. 
We finally estimate the fourth term by
$$
\begin{aligned}
\la -\nabla_x^s U[j[f]] ,  M_q [Q^\perp(g,f)]  \ra_{L^2_x(\Omega)}
 &\lesssim \| \nabla_x^s U[j[f]] \|_{L^2_x} \| M_q [Q^\perp(g,f)] \|_{L^2_x} \\
&\lesssim \| j [f]] \|_{L^2_x L^2_v} \| g   \|_{L^\infty_x L^\infty_{\omega_0}} \| f \|_{L^2_x H^{1,*}_v(\mu^{-1/2})} \\
&\lesssim \| g \|_{\XX_0}  \| f \|_{L^2_x H^{1,*}_v(\mu^{-1/2})}^2, 
\end{aligned}
$$
where we have used the Cauchy-Schwarz inequality in the first line, the estimate \eqref{eq:EstimH2LamE} and \eqref{eq:Qgf-moments} in the second line, and the Cauchy-Schwarz inequality again  in the third line.

Gathering the previous estimates, we obtain
$$
\la \! \la \LL_g f , f \ra \!\ra_{L^2_{x,v} (\mu^{-1/2})} \le - \lambda \Nt f \Nt_{L^2_{x,v} (\mu^{-1/2})}^2 - \left( \sigma_0 - C\| g \|_{\XX_0}  \right) \| f \|_{L^2_x H^{1,*}_v (\mu^{-1/2})}^2
$$
for some constant $C>0$. We then conclude by using the condition  $\| g \|_{\XX_0} \le \eps_5$ and choosing $\eps_5 >0$ small enough such that $C\eps_5 \le \sigma_0/2$.
\end{proof}

We conclude this section by formulating the above hypocoercivity result in a semigroup way, which will be useful in the next section.

\begin{prop}\label{prop:SGSgH*}
 For any $g \in \XX_0$, $\| g \|_{\XX_0} \le \eps_5$, any $t_0 \ge 0$ and any $f_{t_0} \in L^2(\mu^{-1/2}) \cap \CCC_\iota$, 
 the solution $f := S_{\LL_g}(\cdot, t_0)  f_{t_0}$ provided by Theorem~\ref{theo-SG-LLg} 
satisfies 
\beqn\label{eq:prop:SGSgH*}
\Nt f_{t_1} \Nt_{L^2_{xv}(\mu^{-1/2})}^2 +  \sigma \int_{t_0}^{t_1} \| f_s \|_{L^2_x H^{1,*}_v (\mu^{-1/2})}^2 ds \le \Nt f_{t_0} \Nt_{L^2_{xv}(\mu^{-1/2})}^2,  
\eeqn
for any $t_1 \in [t_0,\infty)$. 
\end{prop}

%%%%%%%%%%%%%%%%%%%%%%%%%%%%%%%%%%%
\section{Semigroup estimates for $S_{\LL_g}$}\label{sec-proof-LLg}

Using an  extension trick, we deduce from the previous information on $\LL_g$ and $\BB_g$ a similar result on $\LL_g$ as Theorem~\ref{theo:BBg_L2_Linfty} on $\BB_g$. 
We fix hereafter
\begin{equation}\label{def:eps0}
\eps_0 := \min (\eps_1,\eps_2,\eps_3,\eps_4, \eps_5) >0.
\end{equation}

\begin{theo}\label{theo:LLg_Linfty}
Consider an admissible weight function $\omega$. If $\| g \|_{\XX_0} \le \eps_0$, the semigroup $S_{\LL_g}$ associated to  the evolution problem \eqref{eq:linear_g} satisfies the uniform  estimate, for some constant $C_0>0$, 
\beqn\label{eq:ThSLg-Linfty-bound}
 \| S_{\LL_g}(t,\tau ) f_\tau \|_{L^\infty_{\omega } } \le C_0 \| f_\tau \|_{L^\infty_{\omega}}, \quad \forall \, t \ge \tau  \ge 0, \ \forall \, f_\tau \in L^\infty_{\omega} \cap \CCC_\iota, 
%\| S_{\LL_g}(t,s) \|_{\BBB(L^\infty_{\omega},L^\infty_{\omega_{\star}} )} \lesssim \Theta_{\omega,\omega_{\star}} (t-s), \quad \forall \, t \ge s  \ge 0, 
\eeqn
and the decay estimate
\beqn\label{eq:ThSLg-Linfty-decay}
  \| S_{\LL_g}(t,\tau) f_\tau \|_{L^\infty_{\omega_{\sharp}} } \lesssim \Theta_{\omega} (t-\tau) \| f \|_{L^\infty_{\omega}}, \quad \forall \, t \ge \tau  \ge 0, \ \forall \, f_\tau \in L^\infty_{\omega} \cap \CCC_\iota, 
\eeqn
with $\omega_\sharp = \omega$ or $\omega_\sharp = \omega_0$ and $ \Theta_{\omega}$ defined in the statement of~Theorem~\ref{thm:stabNL-inhom}. 
\end{theo}

\begin{proof}[Proof of Theorem~\ref{theo:LLg_Linfty}]
We shall only consider the case in which the admissible weight function $\omega$ verifies $\gamma+s < 0$, the other case $\gamma+ s \ge 0$ being treated in a similar, and even simpler, way.
We split the proof into four steps.

\medskip\noindent
\textit{Step 1: Convolution and Duhamel formula.} 
For $(U(t,\tau))_{0 \le \tau \le t}$ and $(V(t,\tau))_{0 \le \tau \le t}$ two two-parameters family of operators, we define a new two-parameters family $((U \star V)(t,\tau))_{0 \le \tau \le t}$ of operators given by, for all $0 \le \tau \le t$,
$$
(U \star V)(t,\tau) := \int_\tau^t U(t,\theta) V(\theta,\tau) \, \d \theta ,
$$
and iteratively $U^{\star1} := U$, $U^{\star(k+1)} := U^{\star k} \star U$.

Recalling the splitting $\LL_g = \AA_g + \BB_g$ in \eqref{def:BBg}--\eqref{def:AAg}, using the identity $S_{\LL_g} \Pi^\perp =  \Pi^\perp S_{\LL_g}$ established in Theorem~\ref{theo-SG-LLg} and the shorthand notations $S^\perp_{\LL_g}  = \Pi^\perp S_{\LL_g}$, 
 $S^\perp_{\BB_g}  =  S_{\BB_g}\Pi^\perp $ 
and $^\perp S_{\BB_g}  = \Pi^\perp S_{\BB_g}$, Duhamel's formula gives
\begin{equation}\label{eq:duhamel_SL_SB}
S_{\LL_g}^\perp = S_{\BB_g}^\perp +  (S_{\BB_g}  \AA) \star  S_{\LL_g}^\perp 
\quad\text{and}\quad
S_{\LL_g}^\perp = {^\perp S_{\BB_g}} +  S_{\LL_g}^\perp  \star  (\AA S_{\BB_g}) .
\end{equation}
Iterating \eqref{eq:duhamel_SL_SB}  we also have  
\begin{equation}\label{eq:SL_splitting}
\begin{aligned}
S_{\LL_g}^\perp
&= S_{\BB_g}^\perp
+ \sum_{j=1}^{N-1} (S_{\BB_g} \AA)^{\star j} \star S^\perp_{\BB_g} + (S_{\BB_g} \AA)^{\star N} \star  {^\perp S_{\BB_g} } \\
&\quad
+ (S_{\BB_g} \AA)^{\star N} \star {^\perp S_{\BB_g} } \star (\AA S_{\BB_g})
+ (S_{\BB_g} \AA)^{\star N} \star S^\perp_{\LL_g} \star (\AA S_{\BB_g})^{\star 2}
\end{aligned}
\end{equation}
for any integer $N \in \N^*$.

Giving a function $\Theta : \R_+ \ni t \mapsto \Theta(t) \in \R_+$, we can define the function $\TTT_{+} \ni (t,\tau) \mapsto \Theta(t-\tau) \in \R_+$, where $\TTT_{+} := \{ (t,\tau) \in \R^2 \mid 0 \le \tau \le t \}$, and by abuse of notation we also denote this mapping by $\Theta$.
Considering two such functions $\Theta_1$ and $\Theta_2$, we observe that, for all $0 \le s \le t$, we have 
$$
(\Theta_1 \star \Theta_2)(t,\tau) 
= \int_\tau^t \Theta_1 (t-\theta) \Theta_2(\theta - \tau) \, \d \theta
= \int_0^{t-\tau} \Theta_1 (t-\tau - \theta) \Theta_2(\theta) \, \d \theta 
= (\Theta_1 * \Theta_2)(t-\tau)
$$
where $*$ stands for the usual convolution in one variable. In particular if $\Theta_1 \in L^1(\R_+)$ and $\Theta_2 \in L^\infty(\R_+)$, then one has
$(t,\tau) \mapsto (\Theta_1 \star \Theta_2)(t,\tau) \in L^\infty(\TTT_+)$ with
\begin{equation}\label{eq:convolution_Theta1_Theta2}
\| \Theta_1 \star \Theta_2 \|_{L^\infty(\TTT_+)} \lesssim \| \Theta_1 \|_{L^1(\R_+)} \| \Theta_2 \|_{L^\infty(\R_+)}.
\end{equation}
As a consequence we also obtain that, if $\Theta_1, \ldots, \Theta_n \in L^1(\R_+)$ and $\Theta_{n+1} \in L^\infty(\R_+)$, then 
$(t,\tau) \mapsto (\Theta_1 \star \cdots \star \Theta_{n+1})(t,\tau) \in L^\infty(\TTT_+)$ with
\begin{equation}\label{eq:convolution_Theta1_Thetan+1}
\| \Theta_1 \star \cdots \star \Theta_{n+1} \|_{L^\infty(\TTT_+)} \lesssim \| \Theta_1 \|_{L^1(\R_+)} \cdots \| \Theta_n \|_{L^1(\R_+)} \| \Theta_{n+1} \|_{L^\infty(\R_+)}.
\end{equation}

\medskip\noindent
\textit{Step 2. $L^2$ decay in a reference space.} 
Let $\tau \ge 0$ be fixed and $f_\tau \in L^\infty_\omega \cap \CCC_\iota$.
Denoting $f_{t,\tau} = S_{\LL_g}(t,\tau) f_\tau = S_{\LL_g}^\perp(t,\tau) f_\tau$ for all $t \ge \tau$, the hypocoercivity inequality \eqref{eq:theo:hypo} of Theorem~\ref{theo:hypo} yields
\begin{equation}\label{eq:dt_f_Nt}
\frac{\d}{\d t} \Nt f_{t,\tau} \Nt_{L^2 (\mu^{-1/2})}^2 +  \lambda \Nt  \la v \ra^{\frac{\gamma}{2}+1} f_{t,\tau} \Nt_{L^2 (\mu^{-1/2})}^2 \le 0.
\end{equation}
Assume first that $\gamma \in [-3,2)$. We then fix two admissible weight functions $\nu$ and $\bar \nu$ such that $\bar \nu  \succ \nu   \succ \mu^{-1/2}$. Using Proposition~\ref{prop:BBg_Lp}, Proposition~\ref{prop:SGSgH*} and \eqref{eq:duhamel_SL_SB} we compute 
\bean
\|S_{\LL_g}^\perp   \|_{\BBB(L^2_\nu)} 
&\le& \| S_{\BB_g}   \|_{\BBB(L^2_\nu)} +  \| S_{\BB_g}   \AA \|_{\BBB(L^2(\mu^{-1/2}),L^2_\nu)}   \star  \| S_{\LL_g}^\perp   \|_{\BBB(L^2(\nu),L^2(\mu^{-1/2}))}  \\
&\lesssim& \| S_{\BB_g}  \|_{\BBB(L^2_\nu)} +  \| S_{\BB_g}  \|_{\BBB(L^2_{\bar \nu} ,L^2_\nu)}  \| \AA \|_{\BBB(L^2(\mu^{-1/2}) , L^2_{\bar \nu} )}  \star  \| S_{\LL_g}   \|_{\BBB( L^2(\mu^{-1/2}))}  
\\
&\lesssim & 1 + \Theta_{\bar \nu, \nu} \star 1 
\lesssim  1,
\eean
where we have used Lemma~\ref{lem:borneA0} and $L^2_{\nu} \subset L^2(\mu^{-1/2}) $ in the first line as well as the bound \eqref{eq-prop:BBg_Lp1}, the time-integrable decay estimate \eqref{eq-prop:BBg_Lp2} for $\Theta_{\bar \nu,\nu}$ and the convolution rule \eqref{eq:convolution_Theta1_Theta2} in the third line. With this estimate together with \eqref{eq:dt_f_Nt}, we can apply Proposition~\ref{prop:GalWeakDissipEstim} which yields, for any $0\le \tau \le t$,
\begin{equation}\label{eq:SLg_decay_L2}
\| S_{\LL_g}^\perp(t,\tau) \|_{\BBB(L^2_\nu , L^2 (\mu^{-1/2}))} \lesssim \Theta_{\nu , \mu^{-1/2}}(t-\tau),
\end{equation}
and we observe that $\Theta_{\nu , \mu^{-1/2}} \in L^1(\R_+)$.

Otherwise if $\gamma \ge -2$, we immediately deduce from \eqref{eq:dt_f_Nt} and Gr\"onwall's lemma 
$$
\| S_{\LL_g}^\perp (t,\tau) \|_{\BBB (L^2 (\mu^{-1/2}))} \lesssim e^{-\lambda (t-\tau)} ,
$$
so that estimate \eqref{eq:SLg_decay_L2} also holds in this case with $\nu = \mu^{-1/2}$.

\medskip\noindent
\textit{Step 3: Uniform $L^\infty$ estimate.} 
Writing the splitting \eqref{eq:SL_splitting}, we estimate the norm $\| \cdot \|_{\BBB(L^\infty_{\omega} )}$ of each term separately. From Proposition~\ref{prop:BBg_Lp}, we have 
$$
  \| S_{\BB_g}^\perp \|_{\BBB(L^\infty_{\omega} )} \in L^\infty (\TTT_+)
\quad\text{and}\quad
  \| {^\perp S_{\BB_g}} \|_{\BBB(L^\infty_{\omega}  )} \in L^\infty (\TTT_+),
$$
so that in particular the first term in  \eqref{eq:SL_splitting} is adequately bounded.
We now fix an admissible exponential weight function $\varsigma \succ \omega$ and observe that, from Proposition~\ref{prop:BBg_Lp}, we have
$$
 \| S_{\BB_g} \|_{\BBB( L^\infty_{\varsigma} , L^\infty_{\omega} )} \lesssim \Theta_{\varsigma, \omega}
$$
with $\Theta_{\varsigma, \omega} \in L^1(\R_+)$, and similarly for $S_{\BB_g}^\perp$ and ${^\perp S_{\BB_g}}$. Thanks to Proposition~\ref{prop:BBg_Lp} and using Lemma~\ref{lem:borneA0}, we obtain
$$
\begin{aligned}
\| (S_{\BB_g} \AA) \star S_{\BB_g}^\perp  \|_{\BBB( L^\infty_{\omega}   )}
&\lesssim \left(  \| S_{\BB_g} \|_{\BBB( L^\infty_{\varsigma} , L^\infty_{\omega} )} \| \AA \|_{\BBB( L^\infty_{\omega} , L^\infty_{\varsigma} )} \right) \star   \| S_{\BB_g}^\perp \|_{\BBB( L^\infty_{\omega})}  \\
&\lesssim \Theta_{\varsigma,\omega} \star 1 
\lesssim 1,
\end{aligned}
$$
where we have used \eqref{eq:convolution_Theta1_Theta2} in the last inequality.
All the other terms appearing in the second term in \eqref{eq:SL_splitting} can be estimated in the same manner, and we get for all $j=2, \ldots, N-1$
$$
\begin{aligned}
\| (S_{\BB_g} \AA)^{\star j} \star S_{\BB_g}^\perp  \|_{\BBB( L^\infty_{\omega}  )} 
&\lesssim \left( \| S_{\BB_g} \|_{\BBB( L^\infty_{\varsigma} , L^\infty_{\omega} )} \| \AA \|_{\BBB( L^\infty_{\omega} , L^\infty_{\varsigma} )} \right)^{\star j} \star  \| S_{\BB_g}^\perp \|_{\BBB( L^\infty_{\omega} )}  \\
&\lesssim (\Theta_{\varsigma,\omega})^{\star j} \star 1 
\lesssim 1.
\end{aligned}
$$
The third and fourth terms in \eqref{eq:SL_splitting} can also be estimated in a similar fashion, thus we omit the details.

We now investigate the last term in \eqref{eq:SL_splitting}. We fix exponential admissible weight functions $\varsigma_\star$ and $\varsigma_{\star,1}$ such that $\omega \preceq \varsigma_{\star,1} \prec \varsigma_{\star} \prec \varsigma$. We observe that $ \Theta_{\varsigma, \varsigma_\star} \in L^1(\R_+)$ and we shall apply Theorem~\ref{theo:BBg_L2_Linfty} with the weights $(\varsigma, \varsigma_\star, \varsigma_{\star,1})$ using that $L^\infty_{\varsigma_{\star,1}} \subset L^\infty_{\omega}$.

We first claim that for $N \in \N^*$ large enough (namely such that $\vartheta+2-N \in [0,1)$, where $\vartheta$ is given by Theorem~\ref{theo:BBg_L2_Linfty}), there holds
\begin{equation}\label{eq:SBA*SB_L2_Linfty}	
\| [(S_{\BB_g} \AA)^{\star (N-2)} \star  S_{\BB_g}] \|_{\BBB (L^2_{\varsigma} , L^\infty_{\omega} )} \lesssim \Theta_{\varsigma, \varsigma_\star}.
\end{equation}
Indeed, we compute, for all $0 \le \tau \le t$,
\begin{equation}\label{eq:SBA*SB}
\begin{aligned}
\| S_{\BB_g} \AA \star S_{\BB_g} (t,\tau) \|_{\BBB (L^2_{\varsigma} , L^\infty_{\omega} )}
&\lesssim \int_\tau^{(t+\tau)/2} \| S_{\BB_g}(t,\theta) \AA S_{\BB_g}(\theta,\tau) \|_{\BBB (L^2_{\varsigma} , L^\infty_{\omega} )} \, \d \theta \\
&\quad
+ \int_{(t+\tau)/2}^t \| S_{\BB_g}(t,\theta) \AA S_{\BB_g}(\theta,\tau) \|_{\BBB (L^2_{\varsigma} , L^\infty_{\omega} )} \, \d \theta .
\end{aligned}
\end{equation}
For the first term in \eqref{eq:SBA*SB} we write
$$
\| S_{\BB_g}(t,\theta) \AA S_{\BB_g}(\theta, \tau) \|_{\BBB (L^2_{\varsigma} , L^\infty_{\omega} )} \lesssim \| S_{\BB_g}(t,\theta) \|_{\BBB(L^2_{\varsigma} , L^\infty_{\omega})} \| \AA \|_{\BBB (L^2_{\omega} , L^2_{\varsigma})}  \| S_{\BB_g}(\theta, \tau) \|_{\BBB (L^2_{\varsigma} L^2_{\omega} )} .
$$
Using respectively Theorem~\ref{theo:BBg_L2_Linfty}, Lemma~\ref{lem:borneA0} and Proposition~\ref{prop:BBg_Lp}, we deduce
$$
\begin{aligned}
&\int_\tau^{(t+\tau)/2} \| S_{\BB_g}(t,\theta) \AA S_{\BB_g}(\theta, \tau) \|_{\BBB (L^2_{\varsigma} , L^\infty_{\omega} )} \, \d \theta \\
&\qquad
\lesssim \int_\tau^{(t+\tau)/2} \frac{\Theta_{\varsigma,\varsigma_\star}(t-\theta)}{\min((t-\theta)^\vartheta,1)} \,  \Theta_{\varsigma,\omega}(\theta-\tau) \, \d \theta \\
&\qquad
\lesssim \Theta_{\varsigma,\varsigma_\star}((t-\tau)/2) \int_\tau^{(t+\tau)/2} \frac{\Theta_{\varsigma,\omega}(\theta-\tau)}{\min((t-\theta)^\vartheta,1)}  \, \d \theta
\\
&\qquad
\lesssim \frac{\Theta_{\varsigma,\varsigma_\star}((t-\tau)/2)}{\min((t-\tau)^{\vartheta-1},1)} 
\lesssim \frac{\Theta_{\varsigma,\varsigma_\star}(t-\tau)}{\min((t-\tau)^{\vartheta-1},1)} . 
\end{aligned}
$$
For the second term in \eqref{eq:SBA*SB}, we use the same estimates as above but in the reverse order. More precisely, writing
$$
\| S_{\BB_g}(t,\theta) \AA S_{\BB_g}(\theta, \tau) \|_{\BBB (L^2_{\varsigma} , L^\infty_{\omega} )} 
\lesssim \| S_{\BB_g}(t,\theta) \|_{\BBB(L^\infty_{\varsigma} , L^\infty_{\omega})} \| \AA \|_{\BBB (L^\infty_{\omega} , L^\infty_{\varsigma})}  \| S_{\BB_g}(\theta, \tau) \|_{\BBB (L^2_{\varsigma} , L^\infty_{\omega} )} ,
$$ 
we then apply Proposition~\ref{prop:BBg_Lp}, Lemma~\ref{lem:borneA0} and Theorem~\ref{theo:BBg_L2_Linfty} respectively, which gives
$$
\begin{aligned}
&\int_{(t+\tau)/2}^t \| S_{\BB_g}(t,\theta) \AA S_{\BB_g}(\theta,\tau) \|_{\BBB (L^2_{\varsigma} , L^\infty_{\omega} )} \, \d \theta \\
&\qquad
\lesssim \int_{(t+\tau)/2}^t   \Theta_{\varsigma,\omega}(t-\theta) \frac{\Theta_{\varsigma,\varsigma_\star}(\theta - \tau)}{\min((\theta-\tau)^\vartheta,1)}  \, \d \theta \\
&\qquad
\lesssim \Theta_{\varsigma,\varsigma_\star}((t-\tau)/2) \int_{(t+\tau)/2}^t  \frac{\Theta_{\varsigma, \omega}(t-\theta)}{\min((\theta-\tau)^\vartheta,1)}  \, \d \theta
\\
&\qquad
\lesssim \frac{\Theta_{\varsigma,\varsigma_\star}((t-\tau)/2)}{\min((t-\tau)^{\vartheta-1},1)} 
\lesssim \frac{\Theta_{\varsigma,\varsigma_\star}(t-\tau)}{\min((t-\tau)^{\vartheta-1},1)} . 
\end{aligned}
$$
Gathering the previous estimates, it follows
$$
\| S_{\BB_g} \AA \star S_{\BB_g} (t,s) \|_{\BBB (L^2_{\varsigma} , L^\infty_{\omega} )} \lesssim \frac{\Theta_{\varsigma,\varsigma_\star}(t-s)}{\min((t-s)^{\vartheta-1},1)}.
$$
We conclude the claim~\eqref{eq:SBA*SB_L2_Linfty} by iterating this estimate.

Coming back to the last term in \eqref{eq:SL_splitting}, we choose an admissible weight function $ \varsigma_1 \succ \varsigma$ such that $L^\infty_{\varsigma_1} \subset L^2_{\varsigma}$. Using previous estimates and Proposition~\ref{prop:BBg_Lp} again, we then compute
$$
\begin{aligned}
& \| (S_{\BB_g} \AA) \star [(S_{\BB_g} \AA)^{\star (N-2)} \star (S_{\BB_g} \AA)] \star S_{\LL_g}^\perp \star \AA S_{\BB_g}  \star \AA S_{\BB_g}\|_{\BBB( L^\infty_{\omega} )} \\
&\quad
\lesssim \left( \| S_{\BB_g} \|_{\BBB( L^\infty_{\varsigma} , L^\infty_{\omega} )} \| \AA \|_{\BBB( L^\infty_{\omega} , L^\infty_{\varsigma} )} \right) \star
\left(  \| (S_{\BB_g} \AA)^{\star (N-2)} \star S_{\BB_g} \|_{\BBB (L^2_{\varsigma} , L^\infty_{\omega} )} \| \AA \|_{\BBB(L^{2}(\mu^{-1/2}) , L^2_{\varsigma}) } \right) \\
&\quad\quad
\star \left( \| S_{\LL_g}^\perp \|_{\BBB( L^2_{\nu} , L^{2}(\mu^{-1/2}) )} \right) \star \left(  \| \AA \|_{\BBB(L^2_{\omega} , L^{2}_{\nu} )} \| S_{\BB_g} \|_{\BBB( L^2_{\varsigma} , L^2_{\omega} )} \right) 
\star \left(  \| \AA \|_{\BBB(L^\infty_{\omega} , L^\infty_{\varsigma_1} )} \| S_{\BB_g} \|_{\BBB( L^\infty_{\omega} )} \right) \\
&\quad
\lesssim \Theta_{\varsigma,\omega} \star \Theta_{\varsigma,\varsigma_\star} \star \Theta_{\nu, \mu^{-1/2}} \star \Theta_{\varsigma,\omega} \star 1 
\lesssim 1.
\end{aligned}
$$
We conclude the proof of \eqref{eq:ThSLg-Linfty-bound} by putting together the previous estimates.

\medskip\noindent
\textit{Step 4: $L^\infty$ decay.} 
We write the splitting \eqref{eq:SL_splitting} and we estimate the norm $\| \cdot \|_{\BBB(L^\infty_\omega , L^\infty_{\omega_\sharp})}$ of each term separately.

We first fix an admissible weight function $\omega_\star$ in the following way: If $\omega \preceq \mu^{-1/2}$ then we choose $\omega_\star \preceq \omega$ ; otherwise if $\omega \succ \mu^{-1/2}$ then we choose $\mu^{-1/2} \preceq \omega_\star \preceq \omega$. We next consider an admissible exponential weight function $\varsigma$ such that $\varsigma \succ \omega$ and $\Theta_{\omega, \omega_{\star}}^{-1} \Theta_{\varsigma, \omega_\star} \in L^1(\R_+)$. We finally choose an admissible exponential weight function $\nu$ as in Step~2 such that $\Theta_{\omega,\omega_\star}^{-1} \Theta_{\nu , \mu^{-1/2}} \in L^1(\R_+)$.

Thanks to Proposition~\ref{prop:BBg_Lp} we have
$$
\Theta_{\omega,\omega_\star}^{-1} \| S_{\BB_g}^\perp \|_{\BBB(L^\infty_{\omega} , L^\infty_{\omega_{\star}} )} \in L^\infty (\TTT_+)
\quad\text{and}\quad
\Theta_{\omega,\omega_\star}^{-1} \| {^\perp S_{\BB_g}} \|_{\BBB(L^\infty_{\omega} , L^\infty_{\omega_{\star}} )} \in L^\infty (\TTT_+).
$$
Using Lemma~\ref{lem:borneA0} we also deduce
$$
\begin{aligned}
& \Theta_{\omega,\omega_\star}^{-1} \| (S_{\BB_g} \AA) \star S_{\BB_g}^\perp  \|_{\BBB( L^\infty_{\omega} , L^\infty_{\omega_\star} )} \\
&\quad
\lesssim \left( \Theta_{\omega,\omega_\star}^{-1}\| S_{\BB_g} \|_{\BBB( L^\infty_{\varsigma} , L^\infty_{\omega_\star} )} \| \AA \|_{\BBB( L^\infty_{\omega_\star} , L^\infty_{\varsigma} )} \right) \star \left( \Theta_{\omega,\omega_\star}^{-1} \| S_{\BB_g}^\perp \|_{\BBB( L^\infty_{\omega} , L^\infty_{\omega_\star} )} \right) \\
&\quad
\lesssim \Theta_{\omega,\omega_\star}^{-1} \Theta_{\varsigma,\omega_\star} \star 1 
\lesssim 1 ,
\end{aligned}
$$
where we have used Proposition~\ref{prop:BBg_Lp} and \eqref{eq:convolution_Theta1_Theta2}. All the other terms appearing in the second term in \eqref{eq:SL_splitting} can be estimated in the same manner, and we get for all $j=2, \ldots, N-1$
$$
\begin{aligned}
& \Theta_{\omega,\omega_\star}^{-1} \| (S_{\BB_g} \AA)^{\star j} \star S_{\BB_g}^\perp  \|_{\BBB( L^\infty_{\omega} , L^\infty_{\omega_\star} )} \\
&\quad
\lesssim \left( \Theta_{\omega,\omega_\star}^{-1}\| S_{\BB_g} \|_{\BBB( L^\infty_{\varsigma} , L^\infty_{\omega_\star} )} \| \AA \|_{\BBB( L^\infty_{\omega_\star} , L^\infty_{\varsigma} )} \right)^{\star j} \star \left( \Theta_{\omega,\omega_\star}^{-1} \| S_{\BB_g}^\perp \|_{\BBB( L^\infty_{\omega} , L^\infty_{\omega_\star} )} \right) \\
&\quad
\lesssim (\Theta_{\omega,\omega_\star}^{-1} \Theta_{\varsigma,\omega_\star})^{\star j} \star 1 
\lesssim 1 .
\end{aligned}
$$
The third and fourth terms in \eqref{eq:SL_splitting} can also be estimated in a similar fashion, which gives
$$
\begin{aligned}
& \Theta_{\omega,\omega_\star}^{-1} \| (S_{\BB_g} \AA)^{\star N} \star {^\perp S_{\BB_g}} \|_{\BBB( L^\infty_{\omega} , L^\infty_{\omega_\star} )} \\
&\quad
\lesssim \left( \Theta_{\omega,\omega_\star}^{-1}\| S_{\BB_g} \|_{\BBB( L^\infty_{\varsigma} , L^\infty_{\omega_\star} )} \| \AA \|_{\BBB( L^\infty_{\omega_\star} , L^\infty_{\varsigma} )} \right)^{\star N} 
\star \left( \Theta_{\omega,\omega_\star}^{-1} \| {^\perp S_{\BB_g}} \|_{\BBB( L^\infty_{\omega} , L^\infty_{\omega_\star} )} \right) \\
&\quad
\lesssim (\Theta_{\omega,\omega_\star}^{-1} \Theta_{\varsigma,\omega_\star})^{\star N} \star 1 
\lesssim 1,
\end{aligned}
$$
as well as
$$
\begin{aligned}
& \Theta_{\omega,\omega_\star}^{-1} \| (S_{\BB_g} \AA)^{\star N} \star {^\perp S_{\BB_g}} \star (\AA S_{\BB_g}) \|_{\BBB( L^\infty_{\omega} , L^\infty_{\omega_\star} )} \\
&\quad
\lesssim \left( \Theta_{\omega,\omega_\star}^{-1}\| S_{\BB_g} \|_{\BBB( L^\infty_{\varsigma} , L^\infty_{\omega_\star} )} \| \AA \|_{\BBB( L^\infty_{\omega_\star} , L^\infty_{\varsigma} )} \right)^{\star N} 
\star \left( \Theta_{\omega,\omega_\star}^{-1} \| {^\perp S_{\BB_g}} \|_{\BBB( L^\infty_{\varsigma} , L^\infty_{\omega_\star} )} \right) \\
&\qquad
\star \left( \Theta_{\omega,\omega_\star}^{-1} \| \AA \|_{\BBB( L^\infty_{\omega_\star} , L^\infty_{\varsigma} )} \| S_{\BB_g} \|_{\BBB( L^\infty_{\omega} , L^\infty_{\omega_\star} )} \right) \\
&\quad
\lesssim (\Theta_{\omega,\omega_\star}^{-1} \Theta_{\varsigma,\omega_\star})^{\star N} \star \Theta_{\omega,\omega_\star}^{-1} \Theta_{\varsigma,\omega_\star} \star 1 
\lesssim 1.
\end{aligned}
$$

We now investigate the last term in \eqref{eq:SL_splitting}. We fix exponential admissible weight functions $\varsigma_\star$ and $\varsigma_{\star,1}$ such that $\varsigma_{\star,1} \prec \varsigma_{\star} \prec \varsigma$, $\omega_\star \preceq \varsigma_{\star,1}$ and $\Theta^{-1}_{\omega, \omega_\star} \Theta_{\varsigma, \varsigma_\star} \in L^1(\R_+)$, in such a way that we shall be able to apply Theorem~\ref{theo:BBg_L2_Linfty} below with the weights $(\varsigma, \varsigma_\star, \varsigma_{\star,1})$ and observing that $L^\infty_{\varsigma_{\star,1}} \subset L^\infty_{\omega_\star}$.
Arguing exactly as for obtaining \eqref{eq:SBA*SB_L2_Linfty} in Step~3, we deduce
\begin{equation}\label{eq:SBA*SB_L2_Linfty_bis}	
\| (S_{\BB_g} \AA)^{\star (N-2)} \star  S_{\BB_g} \|_{\BBB (L^2_{\varsigma} , L^\infty_{\omega_\star} )} \lesssim \Theta_{\varsigma, \varsigma_\star}.
\end{equation}
Coming back to the last term in \eqref{eq:SL_splitting}, we choose an admissible weight function $ \varsigma_1 \succ \varsigma$ such that $L^\infty_{\varsigma_1} \subset L^2_{\varsigma}$. We then compute, using previous estimates and Proposition~\ref{prop:BBg_Lp} again,
$$
\begin{aligned}
& \Theta_{\omega,\omega_\star}^{-1} \| (S_{\BB_g} \AA) \star [(S_{\BB_g} \AA)^{\star (N-2)} \star (S_{\BB_g} \AA)] \star S_{\LL_g}^\perp \star \AA S_{\BB_g}  \star \AA S_{\BB_g}\|_{\BBB( L^\infty_{\omega} , L^\infty_{\omega_\star} )} \\
&\quad
\lesssim \left( \Theta_{\omega,\omega_\star}^{-1}\| S_{\BB_g} \|_{\BBB( L^\infty_{\varsigma} , L^\infty_{\omega_\star} )} \| \AA \|_{\BBB( L^\infty_{\omega_\star} , L^\infty_{\varsigma} )} \right) \star
\left( \Theta_{\omega,\omega_\star}^{-1} \| (S_{\BB_g} \AA)^{\star (N-2)} \star S_{\BB_g} \|_{\BBB (L^2_{\varsigma} , L^\infty_{\omega_\star} )} \| \AA \|_{\BBB(L^{2}(\mu^{-1/2}) , L^2_{\varsigma}) } \right) \\
&\quad\quad
\star \left( \Theta_{\omega,\omega_\star}^{-1} \| S_{\LL_g}^\perp \|_{\BBB( L^2_{\nu} , L^{2}(\mu^{-1/2}) )} \right) \star \left( \Theta_{\omega,\omega_\star}^{-1} \| \AA \|_{\BBB(L^2_{\omega_{\star}} , L^{2}_{\nu} )} \| S_{\BB_g} \|_{\BBB( L^2_{\varsigma} , L^2_{\omega_{\star}} )} \right) \\
&\quad\quad
\star \left( \Theta_{\omega,\omega_\star}^{-1} \| \AA \|_{\BBB(L^\infty_{\omega_{\star}} , L^\infty_{\varsigma_1} )} \| S_{\BB_g} \|_{\BBB( L^\infty_{\omega} , L^\infty_{\omega_{\star}} )} \right) \\
&\quad
\lesssim \Theta_{\omega,\omega_\star}^{-1} \Theta_{\varsigma,\omega_\star} \star \Theta_{\omega,\omega_\star}^{-1} \Theta_{\varsigma,\varsigma_\star} \star \Theta_{\omega,\omega_\star}^{-1} \Theta_{\nu,\mu^{-1/2}} \star \Theta_{\omega,\omega_\star}^{-1} \Theta_{\varsigma,\omega_\star} \star 1 
\lesssim 1.
\end{aligned}
$$
We conclude the proof of \eqref{eq:ThSLg-Linfty-decay} by gathering the previous estimates.
\end{proof}

%%%%%%%%%%%%%%%%%%%%%%%%
\section{Proof of the main results}\label{sec-proofMainTheo}

\subsection{Proof of Theorem~\ref{thm:stabNL-inhom}}\label{sec:theo1}

We first define the ball
$$
\BBB_0 := \{ g \in \XX_0; \, \| g \|_{\XX_0} \le \eps_0 \},  
$$
where we recall that $\XX_0 = L^\infty_{\omega_0} ( (0,\infty) \times \OO)$ with $\omega_0 = \la v \ra^{k_0}$ defined in \eqref{eq:defomega0} and $\eps_0>0$ is defined in \eqref{def:eps0}. We next define 
$$
\eps^* :=  \frac{\eps_0}{C_0},  
$$
where $C_0>0$ is defined in the statement of Theorem~\ref{theo:LLg_Linfty}.  
We fix $f_0 \in L^\infty_\omega$ such that $\| f_0 \|_{L^\infty_\omega} \le \eps^*$ and we define 
$$
\Phi : \BBB_0 \to \BBB_0, \quad g \mapsto \Phi(g) = G  := S_{\LL_{g}} f_0. 
$$
It is worth emphasizing that the fact that $S_{\LL_{g}} f_0 \in \BBB_0$ is a direct consequence of \eqref{eq:ThSLg-Linfty-bound} in Theorem~\ref{theo:LLg_Linfty} and of the choice of $\eps_0$ and $\eps^*$. 
We endow $\BBB_0$ with the weak-$*$  topology of $\XX_0$, so that $\BBB_0$ is clearly compact, and we claim that $\Phi$ is continuous for this topology.

Indeed, consider a sequence $(g_n)$ in $\BBB_0$ such that  $g_n \wto g$ weakly-$*$ in $\XX_0$ as $n \to \infty$ and define $G_n := S_{\LL_{g_n}} f_0$. 
From \eqref{eq:ThSLg-Linfty-bound} in Theorem~\ref{theo:LLg_Linfty}, we have 
$$
\| G_n (t,\cdot) \|_{L^\infty_{\omega_0}}
\le \| G_n (t,\cdot) \|_{L^\infty_{\omega}}
\le C_0  \| f_0 \|_{L^\infty_\omega}, \quad \forall \, t \ge 0, 
$$
so that $G_n \in \BBB_0$, and thus there exist a subsequence $(G_{n'})$ and $G \in \BBB_0$ such that $G_{n'} \wto G$ weakly-$*$ in $\XX_0$ as $n' \to \infty$. 
On the one hand, from the dissipativity estimate \eqref{eq:SLg-L2H1*} established in Theorem~\ref{theo-SG-LLg}, we know that $(\nabla_v G_{n'})$ is bounded in $L^2((0,T) \times \OO)$ for all $T>0$. 

\smallskip
On the other hand, we observe that 
$$
\partial_t G_{n'} + v \cdot \nabla_x  G_{n'} = S_{n'} :=    Q(\mu,G_{n'}) +    Q(G_{n'},\mu) +    Q^\perp(g_{n'},G_{n'}), 
$$
where $S_{n'} = \partial_{v_i v_j}^2 A_{n',ij} + \partial_{v_i} B_{n',i}  +  C_{n'}$ from the  expression \eqref{eq:oplandau3} of $Q$, with   $(\la v \ra^3 A_{n',ij})$, $(\la v \ra^3 B_{n',i})$ and $(\la v \ra^3 C_{n'})$ bounded in $L^2((0,\infty) \times \OO)$. 
For any truncated (in $t,x$) version $(\bar G_{n'})$ of $(G_{n'})$, we may thus apply \cite[Theorem~1.3]{MR1949176}, which gives that $(\bar G_{n'})$ is bounded in $H^{1/4} ( \R_t \times \R^3_x \times \R^3_v)$. Therefore we deduce that 
$$
(G_{n'} ) \hbox{ is relatively compact in } L^2((0,T) \times \OO_R). 
$$
for any $T,R > 0$, where $\OO_R := \{ (x,v) \in \OO ; \, d(x,\Omega^c) > 1/R, \, |v| < R \}$.

\smallskip
From the already known weak-$*$ convergence in $\XX_0$ and the decay estimate \eqref{eq:ThSLg-Linfty-decay}, we have established (for instance) that 
$$
G_{n'}  \to G \  \hbox{ strongly in } \ L^2((0,\infty) \times \OO), 
$$
as $n'\to \infty$. 
Using the formulation \eqref{eq:oplandau3} of the Landau operator $Q$ and the above convergence, we have 
\bean
\int_0^T \!\! \int_\OO Q(g_{n'},G_{n'}) \varphi 
&=& 
\int_0^T \!\! \int_\OO \left\{ 
 (a_{ij}*g_{n'} ) G_{n'} \partial^2_{v_iv_j} \varphi + 2   (b_i * g_{n'}) G_{n'}  \partial^2_{v_i} \varphi \right\}
\\
&\to&
\int_0^T \!\! \int_\OO Q(g,G) \varphi, 
\eean
as $n' \to \infty$ for any $\varphi \in \DD((0,T) \times \OO)$. From the very definition of $Q^\perp$, we deduce that 
$$
Q^\perp(g_{n'},G_{n'}) \ \wto \ Q^\perp (g,G) \quad\hbox{in}\ \DD'((0,T) \times \OO)
$$
as $n' \to \infty$.  Thanks to the above convergence and Proposition~\ref{prop:Kolmogorov-stability}--(1), we may thus pass  to the limit in the evolution PDE  
$$
\partial_t G_{n'} = \LL_{g_{n'}} G_{n'} + Q^\perp(g_{n'},G_{n'}), \quad \gamma_- G_{n'} = \RRR \gamma_+ G_{n'}  , \quad
(G_{n'})_{| t=0} = f_0, 
$$
associated to the semigroup definition of $G_{n'}$,   that is, $G_{n'}$ is a weak solution to the above equation in the sense of~Theorem~\ref{theo-SG-LLg}. We obtain that $G$ is a weak solution to
$$
\partial_t G = \LL_{g} G+ Q^\perp(g,G), \quad \gamma_- G  = \RRR \gamma_+ G, \quad  G_{| t=0} = f_0, 
$$
in the sense of Theorem~\ref{theo-SG-LLg}, with moreover 
$$
\| G \|_{\XX_0} \le \eps_0, \quad \int_0^T \!\! \int_\OO |\nabla_v  G|^2 \,  \d v \, \d x \, \d t \le  C(T) \| f_0 \|^2_{L^\infty(\omega)}, 
$$
and by uniqueness in Theorem~\ref{theo-SG-LLg}, we get $G = S_{\LL_{g}} f_0$. By the uniqueness of the possible limit, we have thus established that $\Phi$ is continuous. 

Using now the Schauder-Tychonoff fixed-point theorem, the mapping $\Phi$ has at least one fixed point, that is  there exists $f \in \BBB_0$ such that $f = \Phi(f)$. This function $f$ is a global weak solution to the Landau equation~\eqref{eq:Landau_perturb} in the sense of Theorem~\ref{theo-SG-LLg}, which concludes the proof of Theorem~\ref{thm:stabNL-inhom}.  \qed

\subsection{Proof of Theorem~\ref{theo:iftheo}} \label{sec-proof-iftheorem}

We consider  now a global weak solution $F$ to the Landau equation \eqref{eq:Landau_F}--\eqref{eq:reflect_F}, in the sense of Theorem~\ref{theo-SG-LLg}, which satisfies
\eqref{eq:intro-iftheo1} and \eqref{eq:intro-iftheo2}, for some admissible weight function $\omega_\infty$.
By interpolation, for any $p \in (1,\infty)$, we then have 
$$
\| F_t - \mu \|_{L^p_{\omega_p}(\OO)} \le \eps_p (t) \to 0, \ \hbox{as} \ t \to \infty,
$$
for any admissible weight function $\omega_p$ verifying $\omega_0 \prec \omega_p \prec \omega_\infty$.
We define $f := F - \mu$ which satisfies 
$$
\partial_t f = \mathbf{B}_F f + Q(f,\mu), \quad \mathbf{B}_F f := - v \cdot \nabla_x f + Q(F,f). 
$$
We observe that because of \eqref{eq:intro-iftheo1} and \eqref{eq:intro-iftheo2}, there exists $\EE_0,\HH_0 \in (0,\infty)$ such that 
$$
\rho_F(t,x) \ge \rho_0, \quad \EE_F(t,x) \le E_0, \quad \HH_F (t,x) \le H_0, 
$$
with 
$$
\rho_F := \int_{\R^3} F \, \d v, \quad \EE_F := \int_{\R^3} F\, |v|^2 \, \d v, \quad \HH_F := \int_{\R^3} F \log F \, \d v .
$$
From \cite[Proposition~4]{MR1737547} and \cite[Proposition~2.1]{MR3375485}, there exists then $a_0 = a_0(\rho_0,\EE_0, \HH_0) > 0$ such that 
\beqn\label{eq:aij*F-inIFth}
(a_{ij}* F) \xi_i \xi_j \ge a_0 \langle v \rangle^{\gamma}  |\xi|^2, \quad \forall \, \xi \in \R^3. 
\eeqn
Applying to the dual semigroup $S^*_{\mathbf{B}_F}$ associated to the operator 
$$
\mathbf{B}^*_F h= - v \cdot \nabla_x h +  (a_{ij}* F) \partial_{v_iv_j} h  + 2 (b_{i}* F) \partial_{v_i} h
$$
and the dual reflection condition \eqref{eq:dual_BBg}--\eqref{eq:reflection*}
the same job as done in Proposition~\ref{prop:BB*g_Lq}, we may first establish that 
\beqn\label{eq:S*bfF-Lqm}
\| S^*_{\mathbf{B}_F} (t,s) h  \|_{L^q(m')} \le C_1 e^{C_2 (t-s)} 
\|  h  \|_{L^q(m')}, \quad \forall \, t \ge s \ge 0,
\eeqn
for any $m' := (\omega')^{-1}$ associated to an admissible weight function $\omega'$ such that $s' + \gamma \le 0$, some constants $C_i$ and any $h \in L^q(m')$, $q=1,2$.  The key observation is 
that, with obvious notations taken from the proof of  Proposition~\ref{prop:BB*g_Lq}, 
\bean
\int_\OO (\mathbf{B}^*_F h) h |h|^{q-2} \widetilde m'^q
&=&
 -\frac{4(q-1)}{q^2} \int_{\OO}  (a_{ij} * F)\partial_{v_i} H \partial_{v_j} H  
+\int_{\OO}   \varpi^{(\CC^+_{F-\mu})^*}_{\tilde m',q}  |h|^q \widetilde m'^q,
\\
&&
+ \frac1q \int_\OO |h|^q \, v \cdot \nabla_x (\widetilde m'^q) 
+ \int_{\Sigma} |\gamma h|^q \widetilde m'^q (n_x \cdot v),
\eean
where the first term is nonpositive, the second term is bounded by $\| h \|_{L^q(\tilde m')}^q$ (and it is here that we use the condition $s'+\gamma \le 0$), and the two last terms are identical as
those considered during the proof of  Proposition~\ref{prop:BB*g_Lq}. We immediately deduce \eqref{eq:S*bfF-Lqm} by writing the associated evolution equation and using Gr\"onwall's Lemma. 
By duality, we get a similar conclusion as established  in Proposition~\ref{prop:BBg_Lp}, namely 
\beqn\label{eq:SbfF-Lqm}
\| S_{\mathbf{B}_F} (t,s) f  \|_{L^p(\omega')} \le C_1 e^{C_2 (t-s)} 
\|  f  \|_{L^p(\omega')}, \quad \forall \, t \ge s \ge 0,
\eeqn
for $p=2,\infty$.
It is worth emphasizing here that the semigroups  $S^*_{\mathbf{B}_F}$ and  $S_{\mathbf{B}_F}$ are well defined thanks to Theorem~\ref{theo-Kolmogorov-WellP} as used in Theorem~\ref{theo-SG-LLg}. 
 
\smallskip
We next observe that in the present situation exactly the  same conclusion as the one of Proposition~\ref{prop:BB*g_ultracontractivity} holds, namely any solution $h$ to the backward dual problem \eqref{eq:dual_BBg} associated to $\mathbf{B}^*_F$ instead of $\BB^*_g$ satisfies \eqref{eq:prop:BB*g_ultracontractivity}. We just need to repeat the proof  of Proposition~\ref{prop:BB*g_ultracontractivity}  using in a crucial way the estimate \eqref{eq:aij*F-inIFth}. 
 We may then repeat the proofs  (with no changes!) of Section~\ref{sec-SBBg-DeGiorgiNashMoser}, of Theorem~\ref{theo:BB*g_L1_L2} (with the help of \eqref{eq:S*bfF-Lqm})
and Theorem~\ref{theo:BBg_L2_Linfty} in order to get 
\beqn\label{eq:SbfF-L2inftym}
\| S_{\mathbf{B}_F} (t,s) f  \|_{L^\infty(\omega'_\infty)} \le C_1   \frac{e^{C_2(t-s)} }{ (t-s)^{\eta}}
\|  f  \|_{L^2(\omega_2')}, \quad \forall \, t \ge s \ge 0,
\eeqn
for any admissible weight $\omega_\infty' \prec \omega_2'$.

Interpolating \eqref{eq:SbfF-Lqm} and \eqref{eq:SbfF-L2inftym}, there in particular exists $p \in (2,\infty)$   such that  
$$
\| S_{\mathbf{B}_F}(t,s) h \|_{L^\infty(\omega'_\infty)} \le C_1 \frac{e^{C_2(t-s)} }{ (t-s)^{1/2}}\|  h \|_{L^p(\omega'_p)}, \quad \forall \, t > s \ge 0.
$$
For $t \ge 1$, the Duhamel formula writes 
$$
f(t) = S_{\mathbf{B}_F}(t,t-1) f_{t-1} + \int_{t-1}^t S_{\mathbf{B}_F}(t,\tau)  Q(f_\tau ,\mu) \, \d\tau =: f^1(t) + f^2(t).
$$
On the one hand, choosing $\omega'_\infty \succeq \omega $ and $\omega'_p \preceq \omega_p$ we have 
\bean
\| f^1 (t) \|_{L^\infty(\omega)}  &=&\| S_{\mathbf{B}_F}(t,t-1) f_{t-1} \|_{L^\infty(\omega)} 
\\
&\le& {C_1 e^{C_2} }\| f_{t-1} \|_{L^p(\omega_p)} \le C_1 e^{C_2} \eps_p(t-1).
\eean
On the other hand, we have  
\bean
\| f^2 (t)\|_{L^\infty(\omega)}  &\le & \int_{t-1}^t \| S_{\mathbf{B}_F}(t,\tau)  Q(f_\tau ,\mu)  \|_{L^\infty(\omega)}  \, \d \tau    
\\
&\le&  \int_{t-1}^t  C_1 \frac{e^{C_2(t-\tau)} }{ (t-\tau)^{1/2}}  \| Q(f_\tau ,\mu)  \|_{L^p(\omega_p')}  \, \d \tau   
\\
&\lesssim&    \sup_{(t-1,t)} \eps_p(\tau)  ,
\eean
where we have used Lemma~\ref{lem:borneA0} in order to bound $\| Q(f_\tau ,\mu)  \|_{L^p(\omega_p')}$.
Both estimate together implies that there exists $T > 0$ such that 
$$
\sup_{t \ge T} \| f_t \|_{L^\infty(\omega)}  \le \eps_0, 
$$
where $\eps_0 > 0$ is given by Theorem~\ref{thm:stabNL-inhom}. We may thus apply Theorem~\ref{thm:stabNL-inhom} (or repeat the proof of this one) and we deduce that the accurate rate of convergence 
\eqref{eq:decay-g-inhom-bis} also holds for the solution $F$. This completes the proof of Theorem~\ref{theo:iftheo}. \qed

%%%%%%%%%%%%%%%
\bigskip
%\bibliographystyle{acm}
%\bibliography{bib-Landau}

\end{document}